\documentclass[11pt]{amsart}
\usepackage{amsmath,amssymb,amsthm,latexsym,cite,cancel}
\usepackage[small]{caption}
\usepackage{graphicx,wasysym,overpic,tikz,color}
\usepackage{subfigure,color}
\usepackage{cite}
\usepackage[colorlinks=true,urlcolor=blue,
citecolor=red,linkcolor=blue,linktocpage,pdfpagelabels,
bookmarksnumbered,bookmarksopen]{hyperref}
\usepackage[italian,english]{babel}
\usepackage{units}
\usepackage{enumitem}
\usepackage[left=2.1cm,right=2.1cm,top=2.71cm,bottom=2.71cm]{geometry}
\usepackage[hyperpageref]{backref}

\usepackage[colorinlistoftodos]{todonotes}

\usepackage{amsfonts}
\usepackage{amscd}
\usepackage{latexsym}
\usepackage{mathrsfs}

\newtheorem{theorem}{Theorem}[section]

\newtheorem{remark}{Remark}[section]
\newtheorem{lemma}{Lemma}[section]

\newtheorem{proposition}{Proposition}[section]

\numberwithin{equation}{section}

\title[$(2,q)$-Laplacian equations with GENERIC DOUBLE-BEHAVIOUR
NONLINEARITIESs]{Existence and multiplicity of normalized solutions for $(2,q)$-Laplacian equations with GENERIC DOUBLE-BEHAVIOUR
	NONLINEARITIES}

\author[R. Ding]{Rui Ding}

\address[R. Ding]{\newline\indent
	School of Mathematics
	\newline\indent
	East China University of Science and Technology
	\newline\indent
	Shanghai 200237, PR China }
\email{\href{mailto:dingrui18363922468@outlook.com }{dingrui18363922468@outlook.com }}

\author[C. Ji]{Chao Ji}

\address[C. Ji]{\newline\indent
	School of Mathematics
	\newline\indent
	East China University of Science and Technology
	\newline\indent
	Shanghai 200237, PR China }
\email{\href{mailto:jichao@ecust.edu.cn}{jichao@ecust.edu.cn}}

\author[P. Pucci]{Patrizia Pucci}

\address[P. Pucci]{\newline\indent
	Dipartimento di Matematica e Informatica
	\newline\indent
	Universit\`{a} degli Studi di Perugia
	\newline\indent
	Perugia 06123, Italy}
\email{\href{mailto: patrizia.pucci@unipg.it}{patrizia.pucci@unipg.it}}

\date{\today}

\begin{document}
	\maketitle	
	\begin{abstract}
		{In this paper, we study {existence and multiplicity} of normalized solutions for the following $(2, q)$-Laplacian equation
			\begin{equation*}\label{Eq-Equation1}
				\left\{\begin{array}{l}
					-\Delta u-\Delta_q u+\lambda u=f(u)  \quad x \in \mathbb{R}^N , \\
					\int_{\mathbb{R}^N}u^2 d x=c^2, \\
				\end{array}\right.
			\end{equation*}
			where $1<q<N$, $N\geq3$, $\Delta_q=\operatorname{div}\left(|\nabla u|^{q-2} \nabla u\right)$ denotes the $q$-Laplacian operator, $\lambda$ is a Lagrange multiplier and $c>0$ is a constant. The nonlinearity $f:\mathbb{R}\rightarrow \mathbb{R}$ is continuous, with mass-subcritical growth at the origin, mass-supercritical growth at infinity, and is more general than the sum of two powers.
			Under different assumptions, we prove
			the existence of a locally least-energy solution and the existence of
			a second solution with higher energy.}
	\end{abstract}

	\begin{center}
		\begin{minipage}{12cm}
			\tableofcontents
		\end{minipage}
	\end{center}

	\section{Introduction and main results}	
{In this paper, we are interested in finding solutions $(u, \lambda)\in X\times  \mathbb{R}$ to the following nonlinear elliptic problem of $(2, q)$-Laplacian type
	\begin{equation}\label{Eq-Equation}
		\left\{\begin{array}{l}
			-\Delta u-\Delta_q u+\lambda u=f(u) \quad  x \in \mathbb{R}^N , \\
			\int_{\mathbb{R}^N}u^2 d x=c^2, \\
		\end{array}\right.
	\end{equation}
	where $\Delta_q u=\operatorname{div}\left(|\nabla u|^{q-2} \nabla u\right)$ denotes the $q$-Laplacian of $u$, with $u \in X$, $X:=H^1\left(\mathbb{R}^N\right)\cap D^{1, q}\left(\mathbb{R}^N\right)$, $c>0$, $N \geq 3$,  either $ \frac{2 N}{N+2}<q<2$ or $ 2<q<N$.}


In recent years, the  $(p, q)$-Laplacian equation has received considerable
attention. The $(p, q)$-Laplacian equation comes from
the following general reaction-diffusion equation
\begin{equation}\label{Eq-Reaction}
	u_t=\operatorname{div}\Big(D(u) \nabla u\Big)+f(x, u)\,\,\, \text { where }\,\, D(u):=|\nabla u|^{p-2}+|\nabla u|^{q-2}.
\end{equation}
{Equation \eqref{Eq-Reaction}} has a wide range of applications in physics and related sciences
such as plasma physics, biophysics, and chemical reaction design. In such applications,
the function $u$ describes a concentration; $\operatorname{div}\Big(D(u) \nabla u\Big)$ corresponds to the diffusion and $f(x, u)$ is the reaction related to source and loss processes. {For more details we} refer to \cite{Cherfils}.

Taking the stationary version of \eqref{Eq-Reaction}, with $p=2$, we obtain the following $(2, q)$-Laplacian equation
\begin{equation}\label{Eq-Equationnomass}
	-\Delta u-\Delta_q u=f(x, u) \quad  x \in \mathbb{R}^N. \\
\end{equation}
Due to our scope, here we recall some recent results involving $(p, q)$-Laplacian equations. In \cite{ambrosio2023nonlinear}, Ambrosio investigated the following class of $(p, q)$-Laplacian problems
\begin{equation}\label{Ambr}
	\left\{\begin{array}{l}
		-\varepsilon^p \Delta_p v-\varepsilon^q \Delta_q v+V(x)\left(|v|^{p-2} v+|v|^{q-2} v\right)=f(v) \text { in } \mathbb{R}^N, \\
		v \in W^{1, p}\left(\mathbb{R}^N\right) \cap W^{1, q}\left(\mathbb{R}^N\right), v>0 \text { in } \mathbb{R}^N,
	\end{array}\right.
\end{equation}
where  $1<p<q< N$, $N\geq 3$ and the potential $V$	satisfies a local assumption due to del Pino and Felmer \cite{rDF}, and $f: \mathbb{R}\rightarrow \mathbb{R}$ is a subcritical Berestycki-Lions type nonlinearity. Using variational arguments, the author showed {existence and concentration results} of a family of solutions {for \eqref{Ambr}} as $\varepsilon \rightarrow 0$. In \cite{PW}, Pomponio and Watanabe employed the monotonicity trick to study   existence of a positive radially symmetric ground state solution of the following $(p, q)$-Laplacian equation with general nonlinearity
\begin{equation*}
	\left\{\begin{array}{l}
		- \Delta_{p}u- \beta\Delta_{q}u=f(u),\,\,\, \text { in }\,\, \mathbb{R}^N, \\
		u(x)\rightarrow 0, \quad\quad \quad\quad\quad\quad\text{as}\,\, \vert x\vert\rightarrow \infty,
	\end{array}\right.
\end{equation*}
where $\beta>0$, $1<p<q$, $p< N$ and $N\geq 3$. Later, in \cite{ambrosio2024scalar},  Ambrosio dealt with the following class of $(p, q)$-Laplacian problems
\begin{equation*}
	\left\{\begin{array}{l}
		-\Delta_{p} u-\Delta_{q} u=f(u) \text { in } \mathbb{R}^N, \\
		u \in W^{1, p}\left(\mathbb{R}^N\right) \cap W^{1, q}\left(\mathbb{R}^N\right),
	\end{array}\right.
\end{equation*}
where $1<p<q\leq N$ and $N\geq 2$. He
improved and complemented some results in \cite{ambrosio2023nonlinear, PW}. More precisely, by using suitable variational arguments, he demonstrated the existence of a ground state solution  through three distinct approaches. Moreover, he proved the existence of infinitely many radially symmetric solutions.

{In this paper, inspired by the fact that physicists are often {interested in} normalized solutions, we look for solutions of \eqref{Eq-Equationnomass} in $X$ having a prescribed $L^2$-norm}.
This approach seems to be particularly meaningful from the physical point of view, because in nonlinear optics and {in} the theory of Bose-Einstein condensates, there is a conservation of mass, see \cite{Frantzeskakis2010BEC,Malomed2008BEC}.


For the problem of normalized solutions to {$(2,q)$-Laplacian equations}, we  mention  e.g.\cite{baldelli2022normalized}, where
Baldelli and Yang studied the existence of normalized solution of the following $(2, q)$-Laplacian equation in all possible cases according to the value of $p$,
\begin{equation}\label{eqBaldelli}
	\left\{\begin{array}{l}
		-\Delta u-\Delta_q u=\lambda u+|u|^{p-2} u, \quad x \in \mathbb{R}^N, \\
		\int_{\mathbb{R}^N}u^2 d x=c^2 .
	\end{array}\right.
\end{equation}
In the $L^2$-subcritical case, the authors studied a global minimization problem and obtained a ground state solution for \eqref{eqBaldelli}. While in the $L^2$-critical case, they proved several non-existence results, also extended in the $L^q$-critical case. Finally, for the $L^2$-supercritical case, they derived a ground state {as well as} infinitely many radial solutions.

In a recent paper \cite{cai2024normalized}, Cai and R{\u{a}}dulescu studied the following $(p, q)$-Laplacian equation with $L^p$-constraint
\begin{equation}\label{eqLcai}
	\left\{\begin{array}{l}
		-\Delta_p u-\Delta_q u+\lambda|u|^{p-2} u=f(u), \quad x \in \mathbb{R}^N, \\
		\int_{\mathbb{R}^N}|u|^p d x=c^p, \\
		u \in W^{1, p}\left(\mathbb{R}^N\right) \cap W^{1, q}\left(\mathbb{R}^N\right),
	\end{array}\right.
\end{equation}
where $f: \mathbb{R}\rightarrow \mathbb{R}$ is a continuous function and satisfies weak mass supercritical conditions,
{which allow the case in which
$f$ exhibits either} mass-critical or mass-supercritical at the origin and mass-supercritical growth at infinity. They established the existence of ground states, and revealed some basic behaviors of the ground state energy $E_c$ as $c>0$ varies. The analysis in \cite{cai2024normalized} allows them to provide the general growth assumptions for the reaction $f$.

Recently, {Ding, Ji and Pucci \cite{ding2024normalized} investigated} the existence and multiplicity of normalized solutions for the following $(2, q)$-Laplacian equation
\begin{equation}\label{Equatio11n}
	\left\{\begin{array}{l}
		-\Delta u-\Delta_q u+\lambda u=g(u),\quad x \in \mathbb{R}^N, \\
		\int_{\mathbb{R}^N}u^2 d x=c^2. \\
	\end{array}\right.
\end{equation}
The nonlinearity $g:\mathbb{R}\rightarrow \mathbb{R}$ is continuous and the behaviour of $g$ at the origin is allowed to be strongly sublinear, i.e., $\lim \limits _{s \rightarrow 0} g(s) / s=-\infty$, which includes the logarithmic nonlinearity
$$
g(s)= s \log s^2.
$$
{First, {in \cite{ding2024normalized} we} considered a family of approximating problems that can be set in $X$ and proved the existence of the corresponding least-energy solutions. 	Then,
{we}  proved that such a family of solutions converges to a least-energy solution to the original problem \eqref{Equatio11n}. Moreover, under certain {natural assumptions on $g$, we} also showed the existence of infinitely many solutions of \eqref{Equatio11n}.}

In the past decade or so,
starting from the seminal contribution by Tao, Visan and Zhang \cite{tao2007},
the nonlinear Schr\"odinger equation with mixed power nonlinearities has attracted much attention. Due to our scope, {we} mention that
Soave \cite{SoavenonC,SoaveC} {was} the first to study the following nonlinear Schr\"odinger equation with combined nonlinearities
\begin{equation}\label{Eq-Soave}
	-\Delta u=\lambda u+\mu|u|^{q-2} u+|u|^{p-2} u \quad \text { in } \mathbb{R}^N, N \geq 1,
\end{equation}
having prescribed mass
$$
\int_{\mathbb{R}^N}u^2dx=a^2.
$$
Soave studied {existence and nonexistence of  normalized solutions of} equation \eqref{Eq-Soave}, with $\mu \in \mathbb{R}$, $2<q \leq \bar{2} \leq p$, $q<p$, and either $p<2^*$ or, when $N \geq 3$, $p=2^*$, where $2^*:=\dfrac{2N}{N-2}$ and $\bar{2}:=2+\dfrac{4}{N}$.
When $p<2^*$, Soave proved the existence of
a least-energy solution and a second solution of mountain-pass-type.
However, when $2<q<\bar{2}<p=2^*$, the author obtained only the existence of {a} local minimizer for equation \eqref{Eq-Soave}. The existence of the second normalized solution of mountain-pass type for equation \eqref{Eq-Soave} was proved by Jeanjean in \cite{JeanMultiple} for $N \geq 4$, in \cite{WeiwuC} for $N =3$.
{Moreover,} in \cite{Qi2023limit} asymptotic behaviour of mountain-pass solutions was proved. Additionally, for $p=2^*$ and $q \in\left(2+\dfrac{4}{N}, 2^*\right)$, with $N \geq 3$ and sufficiently large $\mu>0$, Alves, Ji and Miyagaki in \cite{AJM} proved the existence of  a positive ground state solution of \eqref{Eq-Soave} on $	S(a)=\left\{u \in H^1\left(\mathbb{R}^N\right) : \int_{\mathbb{R}^N}u^2dx=a^2\right\}$, complementing some results from \cite{SoaveC}.

Very recently,  Bieganowski, d'Avenia and  Schino in \cite{bieganowski2024existence} considered
existence of solutions $\left(u, \lambda \right) \in H^1\left(\mathbb{R}^N \right) \times \mathbb{R}$ to
$$
-\Delta u+\lambda u=f(u) \quad \text { in } \mathbb{R}^N
$$
with $N \geq 3$ and prescribed $L^2$ norm, and the dynamics of the solutions to
$$
\left\{\begin{array}{l}
	\mathrm{i} \partial_t \Psi+\Delta \Psi=f(\Psi) \\
	\Psi(\cdot, 0)=\psi_0 \in H^1\left(\mathbb{R}^N ; \mathbb{C}\right)
\end{array}\right.
$$
with $\psi_0$ close to ${u}$. Here, the nonlinear term $f$ has mass-subcritical growth at the origin, mass-supercritical growth at infinity, and is more general than the sum of two powers. Under different assumptions, they proved the existence of a local least-energy solution, the orbital stability of all such solutions, the existence of a second solution with higher energy, and the strong instability of such a solution.

Motivated by the aforementioned papers, we establish {existence and multiplicity} of normalized solutions to equation \eqref{Eq-Equation}, where $f: \mathbb{R} \rightarrow \mathbb{R}$ is a general nonlinear function that behaves similarly to the sum of two powers. To the best of our knowledge, this type of problem remains unexplored in the literature.

Before stating the main results of the paper, we present
the main assumptions imposed on $f$. Let $F(t)=\int_0^t f(s) d s$, ${q}_{\#}:=\left(1+\frac{2}{N}\right) \min \{2, q\}$, ${q^{\prime}}:=\max \left\{2^*, q^*\right\}$, $q^*:=\dfrac{Nq}{N-q}$. We use $\lesssim$ to denote an inequality up to a positive multiplicative constant.
\begin{enumerate}[label=(F\arabic*), ref=\textup{F\arabic*}]
	\setcounter{enumi}{-1}
	\item \label{F0} {$f: \mathbb{R}\rightarrow \mathbb{R}$ is continuous} and $\vert f(t)\vert \lesssim |t| + |t|^{q^{\prime}-1}$.
	\item \label{F1} $\lim_{t \rightarrow 0} \dfrac{F(t)}{t^2} = 0$.
	\item \label{F2} $\lim_{t \rightarrow 0} \dfrac{F(t)}{|t|^{q_\#}} = +\infty$.
\end{enumerate}

\begin{remark}
{\rm	From \eqref{F0}, if $F(\zeta)>0$ for some $\zeta \neq 0$ (which occurs if \eqref{F2} holds), the number
	\begin{equation}\label{Eq-C0}
		C_0:=\sup _{{\substack{t \in \mathbb{R}\\t\neq0}}} \frac{F(t)}{t^2+|t|^{q^{\prime}}}>0
	\end{equation}
	is well-defined. Note that, whenever $C_0$ is mentioned, we implicitly assume that $F$ is positive somewhere and $C_0<\infty$.}
\end{remark}

Solutions of \eqref{Eq-Equation} can be obtained as critical points of the energy functional $J:X\rightarrow \mathbb{R}$ given by
{	\begin{equation}\label{Eq-Functional}
		J(u)=\int_{\mathbb{R}^N}\left(\frac{1}{2}|\nabla u|^2+\frac{1}{q}|\nabla u|^q-F(u) \right)dx
	\end{equation}}
under the constraint
$$
\mathcal{S}(c):=\left\{u \in X: \int_{\mathbb{R}^N}u^2dx=c^2\right\}.
$$
It is standard to show that $J$ is of class $C^1$ in $X$, and
{that} any critical point $u$ of $\left.J\right|_{S(c)}$ corresponds to a solution to \eqref{Eq-Equation}, with the parameter $\lambda \in \mathbb{R}$ appearing as a Lagrange multiplier.

For $c>0$, define
$$
\mathcal{D}(c):=\left\{u \in X: \int_{\mathbb{R}^N}u^2dx\leq c^2\right\}{\color{red} .}
$$
Additionally, {define $\widetilde{q}=\max\{2,q\}$ and,} for any $R>0$,  let us introduce
$$
\mathcal{U}_R(c):=\left\{u \in \mathcal{D}(c):
\int_{\mathbb{R}^N}|\nabla u|^{\widetilde{q}}dx<R\right\}, \quad  \quad m_R(c):=\inf _{\mathcal{U}_R(c)} {J(u)}.
$$
The idea of working with $\mathcal{D}(c)$ instead of $\mathcal{S}(c)$ was introduced in \cite{Bieganowski2021masscritical} in the context of nonlinearities with {either} mass-critical or mass-supercritical growth at the origin and mass-supercritical growth at infinity.
The main advantage
is that the weak limit of a sequence in $\mathcal{D}(c)$ still belongs to $\mathcal{D}(c)$, while this is not the case with $\mathcal{S}(c)$ because
the embedding $X \hookrightarrow L^2\left(\mathbb{R}^N\right)$ is not compact, even when considering radially symmetric functions. This makes it easier to obtain a minimizer of $J$ over suitable subsets, which is an important step to obtain a solution to \eqref{Eq-Equation}.

Throughout this paper, let $\mathcal{S}_q$ denote the optimal constant for the Sobolev embedding $D^{1, q}\left(\mathbb{R}^N\right) \hookrightarrow L^{q^*}\left(\mathbb{R}^N\right)$. We now state the main results of our paper. First, we prove the existence of a negative-energy solution to \eqref{Eq-Equation}.

\begin{theorem}\label{Theorem-existence}
	{Assume that
\eqref{F0}-\eqref{F2} hold, also assume that
	\begin{equation}\label{Eq-crange}
c^2<\frac{q^{\prime}}{N}\left(\frac{\mathcal{S}_{\widetilde{q}}}{q^{\prime} C_0}\right)^{N / \widetilde{q}} .
	\end{equation}
{Then there exist	 $R_0>0$ (see \eqref{g2} {of} Lemma \ref{Lemma-g} below),} $\bar{u} \in \mathcal{S}(c) \cap \mathcal{U}_{R_0}(c)$ and $\lambda_{\bar{u}}>0$ such that $J(\bar{u})=m_{R_0}(c)<0,\, \bar{u}$ has constant sign and $\left(\bar{u}, \lambda_{\bar{u}}\right)$ is a solution to \eqref{Eq-Equation}.}
\end{theorem}

\begin{remark}\label{remark-urad}
{\rm	By the regularity {properties proved} in \cite{He2008regularity}, {the solution}
$\bar{u} \in L^{\infty}\left(\mathbb{R}^N\right) \cap \mathcal{C}^{1, \alpha}\left(\mathbb{R}^N\right)$ for every $\alpha \in(0,1)$, and $\bar{u}(x) \rightarrow 0$ as $|x| \rightarrow\infty$. We {use} this property in some of the proofs {below}.
	In addition, { $\bar{u}$  can be always assumed to be} radial and radially monotonic, since it has constant sign, {thanks to Schwarz rearrangement}.}
\end{remark}
In some sense, Theorem \ref{Theorem-existence} can be viewed as an extension of the results in \cite{bieganowski2024existence} {to} the $(2, q)$-Laplacian equations.
{As in \cite{bieganowski2024existence}, the only assumption we need for Theorem \ref{Theorem-existence} is that $c$ is sufficiently small. In particular, we
also do not distinguish between nonlinear terms that have Sobolev-critical or -subcritical growth at infinity.
In fact, one of the purposes of this paper is to understand what reasonably minimal hypotheses we need
both for the various steps and for the main results.}
{In \cite{bieganowski2024existence}, when only a single Laplacian term is present, the working space is \( H^1\left(\mathbb{R}^N\right) \). While, due to the presents of the \( q \)-Laplacian term, the appropriate working space for \eqref{Eq-Equation} is
$
X = H^1\left(\mathbb{R}^N\right) \cap D^{1, q}\left(\mathbb{R}^N\right).
$
Note that \( X \) is not a Hilbert space, which introduces additional complexities in certain estimates.}

 To establish the compactness of the minimizing sequence \( (u_n)_n \) for $J$ at level $m_{R_0}(c)$, it is necessary to show that \( \int_{\mathbb{R}^N} \bar{u}^2 \, dx = c^2 \). The case \( \int_{\mathbb{R}^N} \bar{u}^2 \, dx = 0 \) is ruled out due to the energy being negative. To exclude the case \( 0 < \int_{\mathbb{R}^N} \bar{u}^2 \, dx < c^2 \), a subadditivity inequality for \( m_{R_0}(c) \) is established. Unlike the single Laplacian case, it must be shown that
\[
\lim_{n \rightarrow \infty} \left( \left\| \nabla u_n \right\|_q^q - \left\| \nabla (u_n - \bar{u}) \right\|_q^q \right) = \|\nabla \bar{u}\|_q^q.
\]
Because
 \( X \) is not a Hilbert space, we cannot use the inner product. Therefore, it is essential to prove that \( \nabla u_n \rightarrow \nabla \bar{u} \) holds for almost every \( x \in \mathbb{R}^N \).

{From Theorem \ref{Theorem-existence}, we obtain a solution to equation
\eqref{Eq-Equation}.}
While a ground state solution refers to a solution $(u,\lambda)$ to
\eqref{Eq-Equation} with $u$ minimizes the functional $J$ among all the solutions to \eqref{Eq-Equation}.
That is, $u$ satisfies the following condition:
$$
J^{\prime}\mid_{\mathcal{D}(c)}(u)=0  \quad \text { and } \quad
J(u)=\inf \left\{J(v):\left. v \in \mathcal{D}(c) \text { and }J^{\prime}\right|_{\mathcal{D}(c)}(v)=0\right\} .
$$
It makes sense to ask whether the local minimizer found in Theorem \ref{Theorem-existence} is a ground state {solution}.

From \cite[Lemma 2.3]{baldelli2022normalized}, we know that if \(u\) is a solution of equation
\begin{equation}\label{Func}
	-\Delta u-\Delta_q u+\lambda u=f(u) \quad  x \in \mathbb{R}^N,
\end{equation}
then \(u\) satisfies the following Pohozaev identity:
\begin{equation}\label{Eq-poh}
	\frac{N-2}{2} \int_{\mathbb{R}^N}|\nabla u|^2 \, dx + \frac{N-q}{q} \int_{\mathbb{R}^N}|\nabla u|^q \, dx + \frac{\lambda N}{2} \int_{\mathbb{R}^N}|u|^2 \, dx = N \int_{\mathbb{R}^N} F(u) \, dx,
\end{equation}
and the following Nehari identity:
\begin{equation}\label{Eq-Nehari}
	\int_{\mathbb{R}^N}|\nabla u|^2 \, dx + \int_{\mathbb{R}^N}|\nabla u|^q \, dx + \lambda \int_{\mathbb{R}^N}|u|^2 \, dx = \int_{\mathbb{R}^N} f(u) u \, dx.
\end{equation}
Combining \eqref{Eq-poh} with \eqref{Eq-Nehari}, we obtain that \(u\) satisfies
\begin{equation}\label{Eq-poH3}
	P(u) = \int_{\mathbb{R}^N}|\nabla u|^2 \, dx + \left(\delta_q + 1\right) \int_{\mathbb{R}^N}|\nabla u|^q \, dx - \frac{N}{2} \int_{\mathbb{R}^N} H(u) \, dx = 0,
\end{equation}
where \(\delta_q := \dfrac{N(q-2)}{2 q}\) and \(H(u) := f(u) u - 2 F(u)\).

{Now we introduce the set
	\begin{equation}
		\mathcal{P}_c:=\left\{u \in \mathcal{S}(c): P(u)=0\right\},
	\end{equation}
	It is clear that any solution of \eqref{Eq-Equation} stays in $\mathcal{P}_c$.}

Whether
the local minimizer found in Theorem \ref{Theorem-existence} is a ground state solution is closely related to the minimax structure of \(J\mid_{\mathcal{S}(c)}\) and, in particular, to its behavior with respect to dilations preserving the \(L^2\) norm. For \(u \in \mathcal{S}(c)\) and \(s \in \mathbb{R}\), let
$$
(s * u)(x) := e^{\frac{N}{2} s} u\left(e^s x\right), \quad \text{for any}\, x  \in \mathbb{R}^N.
$$
It results that \(s * u \in \mathcal{S}(c)\). For any \(s \in {\mathbb R}\), we define the map:
\begin{equation}\label{Eq-jsu}
	\psi(s) := J(s * u) = \frac{e^{2s}}{2} \int_{\mathbb{R}^N} |\nabla u|^2 \, dx + \frac{e^{q(\delta_q+1)s}}{q} \int_{\mathbb{R}^N} |\nabla u|^q \, dx - e^{-Ns} \int_{\mathbb{R}^N} F(e^{\frac{N}{2} s} u) \, dx.
\end{equation}
It's easy to see that the critical points of \(\psi\) allow us to project a function that satisfies \eqref{Eq-poH3}.  Thus, the monotonicity and convexity properties of $\psi$ strongly affect the structure of $\mathcal{P}$.

Now, we introduce the following abstract assumptions:
\begin{enumerate}[label=(J\arabic*), ref=\textup{J\arabic*}]
	\setcounter{enumi}{-1}
	\item \label{J0} For every \(u \in \mathcal{D}(c) \backslash \{0\}\), the function \((-\infty,\infty) \ni s \mapsto \psi(s) \in \mathbb{R}\) has a unique local maximum point \(t_u\).
	\item \label{J1}  Assume that \eqref{J0} holds, and for every \(u \in \mathcal{D}(c) \backslash \{0\}\), the function $ (e^{t_u},\infty) \ni s \mapsto \phi(s) := \psi( \ln s) $ is concave.
\end{enumerate}

\begin{proposition}\label{Prop-ground}
	{If \eqref{F0}-\eqref{F2}}, \eqref{J1}, and \eqref{Eq-crange} hold, then
	the solution obtained in Theorem \ref{Theorem-existence} is a ground state solution, that is
	$$
	m_{R_0}(c)=\inf \left\{J(u) : u \in \mathcal{D}(c) \text { and }\left.J\right|_{\mathcal{D}(c)} ^{\prime}(u)=0\right\} ,
	$$
	where
	$$	m_{R_0}(c):=\inf _{\mathcal{U}_{R_0}(c)} J.$$
\end{proposition}

After considering the existence and properties of the first solution to \eqref{Eq-Equation}, {we are in}
the position to consider the existence of a second solution to \eqref{Eq-Equation}. Recall that
	$$
	H(t):=f(t) t-2 F(t) \quad \text { for all } t \in \mathbb{R} \text {. }
	$$
	Let us assume that $H=H_1+H_2$, where $H_1$ and $H_2$ satisfy
\begin{enumerate}[label=(H\arabic*), ref=\textup{H\arabic*}]
	\setcounter{enumi}{-1}
	\item \label{H-0}  $H_1, H_2 \in \mathcal{C}^1(\mathbb{R} ; \mathbb{R})$ and there exist $2 < a_1 < a_2 < q_{\#}$, $q^{\#} < b_1 < b_2 < q^{\prime\prime}$ such that
	$$
	H_1(t) \lesssim|t|^{a_1}+|t|^{a_2}, \quad H_2(t) \lesssim|t|^{b_1}+|t|^{b_2} \quad \text{for all} \,\, t \in \mathbb{R},
	$$
	where ${q}^{\#}:=\left(1+\frac{2}{N}\right) \max \{2, q\}$, ${q}_{\#}:=\left(1+\frac{2}{N}\right) \min \{2, q\}$ and
	${q}^{\prime\prime}:= \min \{2^*, q^*\}$.\\
	\item \label{H-1}
	There holds
	$$
	a_1 H_1(t) \leq h_1(t) t \leq a_2 H_1(t), \quad b_1 H_2(t) \leq h_2(t) t \leq b_2 H_2(t) \quad
	\text{for all} \,\, t \in \mathbb{R},
	$$
	where $h_j:=H_j^{\prime}$ for $j \in\{1,2\}$.
\end{enumerate}
{
	In addition, we assume that $F=F_1+F_2$, where $F_1$ and $F_2$ satisfy what follows:}
\begin{enumerate}[label=(HF), ref=\textup{HF}]
	\item \label{H-F}
{Moreover, for all $t \in \mathbb{R}$}
	$$
	(a_{1}-2)F_1(t)\leq H_1(t)\leq (a_2-2)F_1(t), \quad (b_1-2) F_2(t) \leq H_2(t)\leq(b_{2}-2) F_2(t).
	$$
\end{enumerate}
	Conditions \eqref{H-0} and \eqref{H-1} show that $H$ can be divided into $H_1$ and $H_2$,
	where $H_1$ has mass-subcritical growth at the origin and $H_2$ has mass-supercritical growth at the origin.
	Condition \eqref{H-F} plays a crucial role in ensuring the Lagrange multipliers are positive and also plays an significant role in various estimations involving $H$ and $F$.
	
	Unlike the case in Theorem \ref{Theorem-existence}, when proving the existence of the second solution, our Sobolev critical exponent is not $q^{\prime}:= \max\{2^*, q^*\}$ but rather $q^{\prime\prime}:= \min \{2^*, q^*\}$. This is because only under these conditions can we ensure that the case $u \in\left(\mathcal{D}(c) \backslash \mathcal{S}(c)\right)$ does not occur.

\begin{remark}\label{remark-2}
{\rm	(i) When proving the existence of a second solution, exponent $q^{\#}$ plays an important role. Specifically, we need the following condition:
	\begin{enumerate}[label=\textup{(F\arabic*)}, ref=\textup{F\arabic*}]
		\setcounter{enumi}{+2}
		\item \label{F3}
		$
		\lim _{|t| \rightarrow 0} \dfrac{F(t)}{|t|^{q^{\#}}}=+\infty.
		$
		\item \label{F4}
		$
		\lim _{|t| \rightarrow+\infty} \dfrac{F(t)}{|t|^{q^{\#}}}=+\infty.
		$
		\item \label{F5}
		$
		\lim_{|t| \rightarrow+ \infty} \dfrac{F(t)}{|t|^{q^{\prime}}} = 0.
		$
	\end{enumerate}
	Clearly, \eqref{F2} can be derived from \eqref{F3}.

	(ii) \eqref{F1}, \eqref{F4} and \eqref{F5} can be deduced from \eqref{H-0}, \eqref{H-1}, \eqref{H-F}.}
\end{remark}
When proving the existence of the second solution,
a special role will be played by the Pohozaev set
\begin{equation}
	\mathcal{P}:=\left\{u \in X \backslash\{0\} \mid P(u)=0\right\}, \quad \mathcal{P}^{\mathrm{rad}}:=\mathcal{P} \cap X_{\mathrm{rad}},
\end{equation}
where
$$
X_{\mathrm{rad}}=\{u\in X : u(x)=u(\lvert x \rvert)\}.
$$
From the definition, we can see that $\mathcal{P}_c=\mathcal{P} \cap\mathcal{S}(c)$.
Note that
\begin{equation}\label{Eq-Pand2der}
	u \in \mathcal{P} \quad \text { if and only if }\left.\quad \frac{\mathrm{d}}{\mathrm{d} s} J(s * u)\right|_{s=0}=0 \text { and } u \neq 0
\end{equation}
and that $\mathcal{P}$ consists of the disjoint union of the following sets
$$
\begin{aligned}
	\mathcal{P}_0 & :=\left\{u \in \mathcal{P}:\left.\frac{\mathrm{d}^2}{\mathrm{~d} s^2} J(s * u)\right|_{s=0}=0\right\},\\
	\mathcal{P}_{-} & :=\left\{u \in \mathcal{P}:\left.\frac{\mathrm{d}^2}{\mathrm{~d} s^2} J(s * u)\right|_{s=0}<0\right\}, \\
	\mathcal{P}_{+} & :=\left\{u \in \mathcal{P}:\left.\frac{\mathrm{d}^2}{\mathrm{~d} s^2} J(s * u)\right|_{s=0}>0\right\}
\end{aligned}
$$
where
\begin{equation}\label{Eq-J2der}
	\left.\frac{\mathrm{~d}^2}{\mathrm{~d}s^2} J(s * u)\right|_{s=0}=2 \int_{\mathbb{R}^N}|\nabla u|^2 d x+q\left(\delta_q+1\right)^2 \int_{\mathbb{R}^N}|\nabla u|^q d x-\frac{N^2}{4} \int_{\mathbb{R}^N} h(u) u-2H(u) d x.
\end{equation}
Combining \eqref{Eq-poH3} and \eqref{Eq-J2der}, we obtain that for $u \in \mathcal{P}$
\begin{equation}\label{Eq-J2der1}
	\left.\frac{\mathrm{~d}^2}{\mathrm{~d}s^2} J(s * u)\right|_{s=0}=\left(\delta_q+1\right)\left(q\left(\delta_q+1\right)-2\right)\int_{\mathbb{R}^N}|\nabla u|^q d x-\frac{N^2}{4}\left(\int_{\mathbb{R}^N} h(u) u d x-\bar{2}\int_{\mathbb{R}^N} H(u) d x\right)	,
\end{equation}
and
\begin{equation}\label{Eq-J2der2}
	\left.\frac{\mathrm{~d}^2}{\mathrm{~d}s^2} J(s * u)\right|_{s=0}=\left(2-q\left(\delta_q+1\right)\right)\int_{\mathbb{R}^N}|\nabla u|^2 d x-\frac{N^2}{4}\left(\int_{\mathbb{R}^N} h(u) u d x-\bar{q}\int_{\mathbb{R}^N} H(u) d x\right)	,
\end{equation}
where $\bar{2}=2+\frac{4}{N}$ and $\bar{q}=q+\frac{2q}{N}$.

We shall also consider the set $\mathcal{P}_{-}^{\mathrm{rad}}:=\mathcal{P}_{-} \cap H_{\mathrm{rad}}^1\left(\mathbb{R}^N\right)$.

\begin{theorem}\label{Theorem-existence2}
	If \eqref{F0}, \eqref{F3}, \eqref{H-0}, \eqref{H-1}, \eqref{J0} and \eqref{H-F} hold and $c>0$ is sufficiently small, then there exist $\widetilde{u} \in \mathcal{S}(c)$ and $\lambda_{\widetilde{u}}>0$ such that $J(\widetilde{u})=\min _{\mathcal{P}_{-}^{\mathrm{rad}} \cap \mathcal{D}(c)} J>0$ and $\left(\widetilde{u}, \lambda_{\widetilde{u}}\right)$ is a solution to \eqref{Eq-Equation}.
	Moreover, if $f$ is {either} odd or $\left.f\right|_{(-\infty, 0)} \equiv 0$, then $J(\widetilde{u})=\min _{\mathcal{P}_{-}\cap \mathcal{D}(c)} J$ and $\widetilde{u}$ can be chosen to be positive and non-increasing in the radial coordinate.
\end{theorem}
	For the proof of Theorem \ref{Theorem-existence2}, it is essential to show that \(\mathcal{P}_c = \mathcal{P} \cap \mathcal{S}(c)\) is a smooth manifold of codimension 2. This result is closely related to the fact that \(\mathcal{P}_{0} \cap \mathcal{D}(c) = \emptyset\).
	Due to the presence of the \(q\)-Laplacian term, the proof of \(\mathcal{P}_{0} \cap \mathcal{D}(c) = \emptyset\) {is} more complicated than those in \cite{bieganowski2024existence}.	
	Next, {we} established that \(J\) is bounded away from 0 on \(\mathcal{P}_{-} \cap \mathcal{D}(c)\) and coercive on \(\mathcal{P}_{-}^{\mathrm{rad}} \cap \mathcal{D}(c)\).
	So we obtain a minimizing sequence \((u_n)_n\) of \(J\) on \(\mathcal{P}_{-}^{\mathrm{rad}} \cap \mathcal{D}(c)\) that converges weakly and almost everywhere to a minimizer \(\widetilde{u}\).
	
	In contrast to the cases in \cite{bieganowski2024existence,cai2024normalized}, whether \(\widetilde{u} \in \mathcal{P}_{-}^{\mathrm{rad}} \cap \mathcal{D}(c)\) is a non-trivial problem. As demonstrated in Lemma \ref{Lemma-inf-attained}, there exists \(\widetilde{t} \geq 1\) such that \(\widetilde{u}(\widetilde{t} \cdot) \in \mathcal{P}_{-}^{\mathrm{rad}} \cap \mathcal{D}(c)\).
	Since both the \((2,q)\)-Laplacian term and the interaction between \(H_1\) and \(H_2\) are present, a more precise analysis is required. And we will show that
	{$$
	\int_{\mathbb{R}^N} H_1(u) \, dx \lesssim \int_{\mathbb{R}^N}H_2(u) \, dx .
	$$}
	This inequality is a sufficient condition of proving \(\widetilde{t} = 1\), indicating that if it holds, then \(\widetilde{u} \in \mathcal{P}_{-}^{\mathrm{rad}} \cap \mathcal{D}(c)\). Additionally, proving this inequality requires the condition that \(c\) is small, see Lemma \ref{Lemma-HF>0} for further details.
In the last, we show that for any $u \in\mathcal{P}_{-}^{\text{rad}}\cap \left(\mathcal{D}(c) \backslash \mathcal{S}(c)\right)$, the critical inequality $$\inf _{\mathcal{P}_{-}^{\text{rad}}\cap \left(\mathcal{D}(c) \backslash \mathcal{S}(c)\right)} J<J(u)$$ holds, see Lemma \ref{LemmaPJ<J}. Thus the minimizer $\widetilde{u}$ of $J$ on $\mathcal{P}_{-}^{\text{rad}}\cap \left(\mathcal{D}(c) \backslash \mathcal{S}(c)\right)$ is {attained} on $ \mathcal{P}_{-}^{\text{rad}}\cap\mathcal{S}(c)$. Moreover, by analyzing the Lagrange multipliers $\lambda$ and $\mu$ for constraints $\mathcal{S}(c)$ and $\mathcal{P}_{-}^{\text{rad}}$ respectively, we conclude that $\mu=0$ and $(\widetilde{u}, \lambda_{\widetilde{u}})$ is a normalized solution for equation \eqref{Eq-Equation}.

{The paper is organized as follows. {Section \ref{Sect2} provides} some preliminary results. {Section \ref{Sect3}  focuses on} the existence of negative-energy solutions of problem \eqref{Eq-Equation} and {presents the proof of} Theorem \ref{Theorem-existence}. {The final Section~\ref{Sect4} presents the construction of}  a positive-energy solution of problem \eqref{Eq-Equation} and includes {the} proof of Theorem \ref{Theorem-existence2}. {The paper ends with two appendices}.}

\textbf{Notations:}
For $1 \leq p<\infty$ and $u \in L^p\left(\mathbb{R}^N\right)$, we denote $\|u\|_p:=\left(\int_{\mathbb{R}^N}|u|^p d x\right)^{\frac{1}{p}}$. The Hilbert space $H^1\left(\mathbb{R}^N\right)$ is defined as $H^1\left(\mathbb{R}^N\right):=\left\{u \in L^2\left(\mathbb{R}^N\right): \nabla u \in L^2\left(\mathbb{R}^N\right)\right\}$ with inner product $(u, v):=\int_{\mathbb{R}^N}(\nabla u \nabla v+u v) d x$ and norm $\|u\|:=\left(\|\nabla u\|_2^2+\|u\|_2^2\right)^{\frac{1}{2}}$. Similarly, $D^{1, q}\left(\mathbb{R}^N\right)$ is defined as $D^{1, q}\left(\mathbb{R}^N\right):=\left\{u \in L^{q^*}\left(\mathbb{R}^N\right): \nabla u \in L^q\left(\mathbb{R}^N\right)\right\}$ with the norm $\|u\|_{D^{1, q}\left(\mathbb{R}^N\right)}=\|\nabla u\|_q$.
Recalling $X=H^1\left(\mathbb{R}^N\right) \cap D^{1, q}\left(\mathbb{R}^N\right)$ endowed with the norm $\|u\|_X=\|u\|+\|u\|_{D^{1, q}\left(\mathbb{R}^N\right)}$. We use $" \rightarrow$ " and $" \rightharpoonup$ " to denote the strong and weak convergence in the related function spaces respectively. $C$ and $C_i$ will be positive constants. $\langle\cdot, \cdot\rangle$ denote the dual pair for any Banach space and its dual space. Finally, $o_n(1)$ and $O_n(1)$ mean that $\left|o_n(1)\right| \rightarrow 0$ and $\left|O_n(1)\right| \leq C$ as $n \rightarrow\infty$, respectively.

\section{Preliminaries}\label{Sect2}

In this section, we present some preliminary results. First, we give the well-known Sobolev embedding theorems and Gagliardo-Nirenberg
inequalities.

\subsection{Sobolev inequalities and Gagliardo-Nirenberg inequalities}

\begin{lemma}[Best constant for the Sobolev inequality\cite{talenti1976best}]\label{Lemma-Sobolev inequality}
	Let $1<q<N$, there exists an optimal constant $\mathcal{S}_q>0$ depending only on $N,q$, such that
	\begin{equation}\label{Sobolev inequality}
		\mathcal{S}_q\|f\|_{q^*}^q \leq\|\nabla f\|_q^q, \quad \forall f \in D^{1, q}\left(\mathbb{R}^N\right) , \quad \text { (Sobolev inequality) }
	\end{equation}
	and the equality holds if $f= U_{a, \xi_0}(x)$, where
	$$
	U_{a, \xi_0}(x)=\left(\frac{a^{\frac{1}{p-1}} N^{\frac{1}{p}}\left(\frac{N-p}{p-1}\right)^{\frac{p-1}{p}}}{a^{\frac{p}{p-1}}+\left|x-\xi_0\right|^{\frac{p}{p-1}}}\right)^{\frac{N-p}{p}}, \quad a>0, \quad \xi_0 \in \mathbb{R}^N.
	$$
\end{lemma}

\begin{lemma}[The Gagliardo-Nirenberg inequality {\cite[Corollary 2.1]{weinstein1982nonlinear}}]\label{Lemma-GN inequality1}
	Let $p \in\left(2,2^*\right)$ and $\delta_p=\frac{N(p-2)}{2 p}$. Then there exists a constant ${C}_{N, p}=\left(\frac{p}{2\left\|W_p\right\|_2^{p-2}}\right)^{\frac{1}{p}}>0$ such that
	\begin{equation}\label{Eq-GN inequality}
		\|u\|_p \leq {C}_{N, p}\|\nabla u\|_2^{\delta_p}\|u\|_2^{\left(1-\delta_p\right)}, \quad \forall u \in H^1\left(\mathbb{R}^N\right),
	\end{equation}
	where $W_p$ is the unique positive radial solution of $-\Delta W+\left(\frac{1}{\delta_p}-1\right) W=\frac{2}{p \delta_p}|W|^{p-2} W$.
\end{lemma}

\begin{lemma}[$L^q$-Gagliardo-Nirenberg inequality{\cite[Theorem 2.1]{agueH3008sharp}}]\label{Lemma-GN inequality2}
	Let $q \in\left(\frac{2 N}{N+2}, N\right), p \in$ $\left(2, q^*\right)$ and $\nu_{p, q}=\frac{N q(p-2)}{p[N q-2(N-q)]}$. Then there exists a constant {$K_{N, p}>0$} such that
	{\begin{equation}\label{Eq-GN inequality2}
		\|u\|_p \leq K_{N, p}\|\nabla u\|_q^{\nu_{p, q}}\|u\|_2^{\left(1-\nu_{p, q}\right)}, \quad \forall u \in D^{1, q}\left(\mathbb{R}^N\right) \cap L^2\left(\mathbb{R}^N\right),
	\end{equation}	
	where}
	$$
	\begin{gathered}
		K_{N, p}=\left(\frac{K}{\frac{1}{q}\left\|D W_{p, q}\right\|_q^q+\frac{1}{2}\left\|W_{p, q}\right\|_2^2}\right) \\
		K=(N q+p q-2 N) \cdot\left(\frac{[2(N q-p(N-q))]^{p(N-q)-N q}}{[q N(p-2)]^{N(p-2)}}\right)^{1 /[N q+p q-2 N]},
	\end{gathered}
	$$
	and $W_{p, q}$ is the unique nonnegative radial solution of the following equation
	$$
	-\Delta_q W+W=\zeta|W|^{p-2} W
	$$
	where $\zeta=\|\nabla W\|_q^q+\|W\|_2^2$ is the Lagrangian multiplier.
\end{lemma}

{\begin{lemma}[{The Sobolev embedding Theorem \cite[Lemma 1.2]{cai2024normalized}, \cite{baldelli2023normalized}}]\label{LemmaEmbed}
	The space $X$ is embedded continuously into $L^m\left(\mathbb{R}^N\right)$ for $m \in\left[2, q^{\prime}\right]$ and compactly into $L_{\text {loc }}^m\left(\mathbb{R}^N\right)$ for $m \in\left[1, q^{\prime}\right)$,
	where $q^{\prime}:=\max\{2^*,q^*\}$. Denote $X_{\mathrm{rad}}:=\{u \in X : u \text{ is radially symmetric}\}$, then the space $X_{\mathrm{rad}}$ is embedded compactly into $L^m\left(\mathbb{R}^N\right)$ for $m \in \left(2, q^{\prime}\right)$.
\end{lemma}}
	
For the next lemma, we can take a similar argument as that of the classical {Concentration-Compactness principle}. See, for instance,\cite[Lemma 1.21]{Willem1996Minimax}.

\begin{lemma}\label{LemmaLions}
	Let $r>0$. {If $(u_n)_n$} is a bounded sequence in $X$ which satisfies
	$$
	\sup _{x \in \mathbb{R}^N} \int_{B_r(x)}\left|u_n\right|^2 d x \rightarrow 0, \quad \text { as } n \rightarrow +\infty,
	$$
	then,
	$$
	\left\|u_n\right\|_m \rightarrow 0 \quad \text { as } n \rightarrow +\infty
	$$
	holds for any $m \in\left(2,q^{\prime}\right)$, where ${q^{\prime}}=\max \left\{2^*, q^*\right\}$.
\end{lemma}
\subsection{Useful inequalities related to compactness}
\
\newline
For our subsequent estimates, we recall some fundamental inequalities in this subsection. The following interpolation inequality is particularly important, with special attention to the range of $t$.

\begin{lemma}[{Interpolation inequality}]\label{Lemma-Interpolationinequality}
	Assume that
	$f \in L^p(\mathbb{R}^N) \cap L^q(\mathbb{R}^N)$ with $1 \leq p \leq q \leq \infty$, then $f \in L^r(\mathbb{R}^N)$ for all $r\in (p,q)$, and
	$$
	\|f\|_r \leq\|f\|_p^t\|f\|_q^{1-t} \text {, where } \frac{1}{r}=\frac{t}{p}+\frac{1-t}{q}, 0 \leq t \leq 1 \text {,}
	$$
	additionally, $t=\frac{1/r-1/q}{1/p-1/q}$ and $1-t=\frac{1/p-1/r}{1/p-1/q}$.
\end{lemma}

\begin{remark}\label{Remark-inter}
	Under the assumption of Lemma \ref{Lemma-Interpolationinequality}, it's easy to see that
	$r(1-t)<q$, since
	$$
	\frac{r}{q}(1-t) = \frac{r}{q} \frac{1/p - 1/r}{1/p - 1/q} = \frac{r/p - 1}{q/p - 1} < 1.
	$$
\end{remark}

\begin{lemma}[Young's inequality]\label{Young's inequality}
	If $a \geq 0$ and $b \geq 0$ are nonnegative real numbers and if $p>1$, $q>1$ are real numbers such that $\frac{1}{p}+\frac{1}{q}=1$, then for every $\delta>0$
	$$
	a b \leq \frac{a^p}{\delta^p p}+\frac{\delta^qb^q}{q}.
	$$
\end{lemma}

\begin{lemma}[{ Lemma 2.7, \cite{gongbao1994stability}}]\label{Lemma-Converae}
	Assume $s>1$, and let $\Omega$ be an open set in $\mathbb{R}^N, \alpha, \beta$ positive numbers and $a(x, \xi)\in C\left(\Omega \times \mathbb{R}^N, \mathbb{R}^N\right)$ such that\\
	(1) $\alpha|\xi|^s \leq a(x, \xi) \xi$ for all $(x, \xi) \in \Omega \times \mathbb{R}^N$,\\
	(2) $|a(x,\left.\xi)|\leq \beta| \xi|^{s-1}\right.$ for all $(x, \xi) \in \Omega \times \mathbb{R}^N$,\\
	(3) $(a(x, \xi)-a(x, \eta))(\xi- \eta)>0$ for all $(x, \xi) \in \Omega \times \mathbb{R}^N$ with $\xi \neq \eta$,\\
	(4) $a(x, \gamma \xi)=\gamma|\gamma|^{p-2} a(x, \xi)$ for all $(x, \xi) \in \Omega \times \mathbb{R}^N$ and $\gamma \in \mathbb{R} \backslash\{0\}$.\\
	Consider  $\left(u_n\right)_n\subset W^{1, s}(\Omega), u \in W^{1, s}(\Omega)$, then $\nabla u_n \rightarrow \nabla u$ in $L^s(\Omega)$ if and only if
	$$
	\lim _{n \rightarrow \infty} \int_{\Omega}\Big(a\left(x, \nabla u_n(x)\right)-a(x, \nabla u(x))\Big)\left(\nabla u_n(x)-\nabla u(x)\right) d x=0  .
	$$
\end{lemma}
To conclude this section, we recall the following elementary inequality.
	This inequality will be used to show that if $(u_n)_n$ is a minimizing sequence of $J$ and
	$u_n \rightharpoonup {u}$ in $X$, then
	$\nabla u_n \rightarrow \nabla {u}$ for a.e. $x\in \mathbb{R}^N$.
{	
{	\begin{lemma}[Formula 2.2, \cite{Baldelli20222qF}]\label{Lem-aee}
		There exists a constant $C(s)>0$ such that
		for all $x, y \in \mathbb{R}^N$ with $|x|+|y| \neq 0$,
		$$
		\left\langle|x|^{s-2} x-|y|^{s-2} y, x-y\right\rangle \geq C(s) \begin{cases}\dfrac{|x-y|^2}{(|x|+|y|)^{{(2-s)}}}, &1 \leq s<2, \\ |x-y|^s, & s \geq 2 .\end{cases}
		$$
\end{lemma}

\section{Local minimum and the first normalized solution}\label{Sect3}
The aim of the section is to investigate the existence of negative-energy solutions, as stated in Theorem \ref{Theorem-existence} and Proposition \ref{Prop-ground}.

Let us begin by establishing an appropriate estimate for the energy functional $J$ on the disk $\mathcal{D}(c)$ to prove the existence of negative-energy solutions.

Observe that, by \eqref{Eq-C0} and Sobolev inequality \eqref{Sobolev inequality}, for any $u \in \mathcal{D}(c)$, we have
\begin{equation}\label{Eq-J(u)sob}
	J(u) \geq \frac{1}{\widetilde{q}}\|\nabla u\|_{\widetilde{q}}^{\widetilde{q}}-C_0\left(\|u\|_2^2+\|u\|_{q^{\prime}}^{q^{\prime}}\right) \geq \frac{1}{\widetilde{q}}\|\nabla u\|_{\widetilde{q}}^{\widetilde{q}}-C_0 c^2-C_0 \mathcal{S}_{\widetilde{q}}^{-q^{\prime} / \widetilde{q}}\|\nabla u\|_{\widetilde{q}}^{q^{\prime}}
\end{equation}
and so $\left.J\right|_{\mathcal{D}(c)}$ is bounded from  below on bounded subsets in $X$.

Define $g:(0,\infty) \times(0,\infty) \rightarrow \mathbb{R}$ as
\begin{equation}\label{Eq-g(a,t)}
	g(\alpha, t):=\frac{1}{\widetilde{q}}-C_0 \alpha^2 t^{-\widetilde{q}}-C_0 \mathcal{S}_{\widetilde{q}}^{-q^{\prime}/ \widetilde{q}} t^{q^{\prime}-\widetilde{q}}.
\end{equation}
Then, \eqref{Eq-J(u)sob} can be written as
\begin{equation}\label{Eq-J>g}
	J(u) \geq g\left(c,\|\nabla u\|_{\widetilde{q}}\right)\|\nabla u\|_{\widetilde{q}}^{\widetilde{q}}, \quad \text { for all } u \in \mathcal{D}(c).
\end{equation}

\begin{lemma}\label{Lemma-g}
	The following facts hold.
	\begin{enumerate}[label=\textup{(g\arabic*)}, ref=\textup{g\arabic*}]
		\item \label{g1} For every $\alpha>0$, the function $t \mapsto g(\alpha, t) t^{\widetilde{q}}$, has a unique critical point, which is a global maximizer.
		\item \label{g2} If \eqref{Eq-crange} holds, then there exist $R_0, R_1>0$ with $R_0<R_1$, such that $g\left(c, R_0\right)=g\left(c, R_1\right)=0, g(c, t)>0$ for $t \in\left(R_0, R_1\right)$, and $g(c, t)<0$ for $t \in\left(0, R_0\right) \cup\left(R_1,\infty\right)$.
		\item \label{g3} If $t>0$ and $\alpha_1 \geq \alpha_2>0$, then for every $s \in\left[(\frac{\alpha_2}{\alpha_1})^{{2}/{\widetilde{q}}}t, t\right]$, there holds $g\left(\alpha_2, s\right) \geq g\left(\alpha_1, t\right)$.
		\item \label{g4} If $c$ satisfies \eqref{Eq-crange}, then there exists $\varepsilon>0$ such that \eqref{Eq-crange} is verified by every $c^{\prime} \in(c-\varepsilon, c+\varepsilon)$ and the functions $(c-\varepsilon, c+\varepsilon) \ni c^{\prime} \mapsto R_i\left(c^{\prime}\right) \in(0,\infty), i \in\{0,1\}$ with $R_i$ defined in \eqref{g2}, are invertible and of class $\mathcal{C}^1$.
	\end{enumerate}
\end{lemma}

\begin{proof}
	\eqref{g1} follow from direct computations. Now we prove \eqref{g2},
	since
	 \begin{equation}
	 \frac{ \partial}{ \partial t }g(\alpha,t)=C_0 t^{-{\widetilde{q}-1}}\left({\widetilde{q}} \alpha^2-\mathcal{S}_{\widetilde{q}}^{-{q^{\prime}} / {\widetilde{q}}}\left(q^{\prime}-{\widetilde{q}}\right) t^{q^{\prime}}\right),
	 \end{equation}
  	the only critical point of $\frac{ \partial}{ \partial t }g(\alpha,t)$ is
  	\begin{equation}
  		t_0=\left(\frac{\widetilde{q} \alpha^2}{\mathcal{S}_{\widetilde{q}}^{-{q^{\prime}} / {\widetilde{q}}}\left({q^{\prime}}-\widetilde{q}\right)}\right)^{1 / {q^{\prime}}}.
  	\end{equation}
  And it's easy to see that $ \frac{ \partial}{ \partial t }g(\alpha,t)>0$ if $t<t_0$ and
  $ \frac{ \partial}{ \partial t }g(\alpha,t)<0$ if $t>t_0$. Thus, $g(\alpha,t)$ has a unique critical point $t_0$, which is a global maximizer.
  Bring $t_0$ into $g(\alpha,t)$ we get
  $$
  g(\alpha,t_0)=\frac{1}{\widetilde{q}}-\frac{q^{\prime}}{\widetilde{q}} \left(\frac{q^{\prime}}{N}\right)^{-\frac{\widetilde{q}}{N}} C_0\mathcal{S}_{\widetilde{q}}^{-1} \alpha^{\frac{2\widetilde{q}}{N}}.
  $$
  So, if \eqref{Eq-crange} holds, we have that $g(c,t_0)>0$. And
  since $\lim_{t \rightarrow 0}g(c,t)=\lim_{t \rightarrow \infty}g(c,t)=-\infty$, there exist $R_0,R_1>0$, such that $g\left(c, R_0\right)=g\left(c, R_1\right)=0$.
  Given that
  $ \frac{ \partial}{ \partial t }g(\alpha,t)>0$ if $t<t_0$ and
  $ \frac{ \partial}{ \partial t }g(\alpha,t)<0$ if $t>t_0$, we get \eqref{g2}.

 Concerning \eqref{g3}, it is clear that $g\left(\alpha_2, t\right) \geq g\left(\alpha_1, t\right)$ for all $t>0$. Moreover,
 {
	$$
	g\left(\alpha_2, \left(\frac{\alpha_2}{\alpha_1}\right)^{{2}/{\widetilde{q}}} t\right) - g\left(\alpha_1, t\right) = \frac{C_0}{\mathcal{S}_{\widetilde{q}}^{q^{\prime} / \widetilde{q}}}\left(1 - \left(\frac{\alpha_2}{\alpha_1}\right)^{{2}(q^{\prime}-\widetilde{q})/{\widetilde{q}}}\right) t^{q^{\prime}-\widetilde{q}} \geq 0,
	$$
}
from \eqref{g2}, since $g(\alpha, t)$ has a unique critical point, which is a global maximizer, we conclude.

  Finally, as for \eqref{g4}, it follows from the differentiability of $g$, $\partial_t g\left(c, R_0\right)>0$, $\partial_t g\left(c, R_1\right)<0$, and the implicit function theorem.
\end{proof}

{To rule out the vanishing case for the minimizing sequence of $J$, we will use the negativity of energy. To exclude the dichotomy, we need to establish a subadditivity inequality for $m_{R_0}(c)$.}

\begin{lemma}\label{Lemma-mR(a)}
	If \eqref{F0} and \eqref{F2} hold, then $m_R(c) \in(-\infty, 0)$ for every $c, R>0$.
\end{lemma}
\begin{proof}
	First of all, let us observe that, from \eqref{Eq-J(u)sob}, $m_R(c)>-\infty$.
	 Then, fix $u \in \mathcal{D}(c) \cap L^{\infty}\left(\mathbb{R}^N\right) \backslash\{0\}$
	 and recall $\widehat{q}=\min\{2,q\}$,
	 we observe that\\
	 	 \begin{equation}
	 	\begin{aligned}
	 	J(s * u) &= \frac{e^{2s}}{2} \int_{\mathbb{R}^N} |\nabla u|^2 \, dx + \frac{e^{q(\delta_q+1)s}}{q} \int_{\mathbb{R}^N} |\nabla u|^q \, dx -e^{-Ns}\int_{\mathbb{R}^N}F(e^{\frac{N}{2} s} u)  dx\\
	 	&= e^{\widehat{q}(\delta_{\widehat{q}}+1)s}\left( \frac{e^{(2-\widehat{q}(\delta_{\widehat{q}}+1))s}}{2}\int_{\mathbb{R}^N} |\nabla u|^2 \, dx + \frac{e^{(q(\delta_q+1)-\widehat{q}(\delta_{\widehat{q}}+1))s}}{q}  \int_{\mathbb{R}^N} |\nabla u|^q \, dx -\frac{1}{\left(e^{ \frac{N}{2}s}\right)^{{q}_{\#}}}\int_{\mathbb{R}^N}F(e^{\frac{N}{2} s} u)  dx\right).
	 	\end{aligned}
	 \end{equation}	
 Note that, from \eqref{F2}, $F\left(e^{\frac{N}{2}s} u\right)>0$ a.e. in $\operatorname{supp} u$ for sufficiently small $s$. Therefore, from Fatou's lemma and \eqref{F2} again,
 	$$
 	\lim _{s \rightarrow -\infty }\frac{1}{\left(e^{ \frac{N}{2}s}\right)^{{q}_{\#}}}\int_{\mathbb{R}^N}F(e^{\frac{N}{2} s} u)  dx=\infty.
 	$$
 	Since $\widehat{q}=\min\{2,q\}$, $\lim _{s \rightarrow -\infty }{e^{(2-\widehat{q}(\delta_{\widehat{q}}+1))s}}=0$ and $\lim _{s \rightarrow -\infty }e^{(q(\delta_q+1)-\widehat{q}(\delta_{\widehat{q}}+1))s}=0$. This implies that $J(s * u)<0$ for sufficiently small $s$, and since $s * u \in \mathcal{U}_R(c)$ provided $s$ is small, we can conclude.\\
\end{proof}

\begin{remark}\label{Remark-g>0}
	 From \eqref{g2} and Lemma \ref{Lemma-mR(a)}, since $g\left(c, R_0\right)=0$, there exists $\varepsilon>0$ such that
	$$
	0 \geq g(c, s) \geq \frac{m_{R_0}(c)}{2 R_0^{\widetilde{q}}} \text { for all } s \in\left[R_0-\varepsilon, R_0\right] .
	$$
This, \eqref{Eq-J>g}, and Lemma \ref{Lemma-mR(a)} yield that for all $u \in \mathcal{D}(c)$ with $R_0-\varepsilon \leq\|\nabla u\|_{\widetilde{q}} \leq R_0$ there holds
	$$
	J(u) \geq g\left(c,\|\nabla u\|_{\widetilde{q}}\right)\|\nabla u\|_{\widetilde{q}}^{\widetilde{q}} \geq R_0^{\widetilde{q}} \frac{m_{R_0}(c)}{2 R_0^{\widetilde{q}}}>m_{R_0}(c) .
	$$
\end{remark}

Next, we show the subadditivity property of $m_{R_0}\left(c\right)$.
\begin{lemma}\label{Lemma-subadd}
	If \eqref{F0}, \eqref{F2}, and \eqref{Eq-crange} are satisfied, we have for all $ \alpha \in(0, c) $: $$m_{R_0}(c) \leq m_{R_0}(\alpha)+m_{R_0}(\sqrt{ c^2-\alpha^2})$$ and if $m_{R_0}(\alpha)$ or $m_{R_0}(\sqrt{ c^2-\alpha^2})$ is reached then the inequality is strict.
\end{lemma}
\begin{proof}
 	Note that, fixed $\alpha \in(0, c)$, it is sufficient to prove that the following holds
 \begin{equation}\label{Eq-mtha}
 	m_{R_0}(\theta \alpha) \leq \theta^2 m_{R_0}(\alpha), \forall \theta \in\left(1, \frac{c}{\alpha}\right]
 \end{equation}
 and that, if $m_{R_0}(\alpha)$ is reached, the inequality is strict. Indeed, if \eqref{Eq-mtha} holds then we have
 \begin{equation}
 	\begin{aligned}
 		m_{R_0}(c) & =\frac{c^2-\alpha^2}{c^2} m_{R_0}(c)+\frac{\alpha^2}{c^2} m_{R_0}(c)=\frac{c^2-\alpha^2}{c^2} m_{R_0}\left(\frac{c}{\sqrt{c^2-\alpha^2}}(\sqrt{c^2-\alpha^2})\right)+\frac{\alpha^2}{c^2} m_{R_0}\left(\frac{c}{\alpha} \alpha\right) \\
 		& \leq m_{R_0}(\sqrt{c^2-\alpha^2})+m_{R_0}(\alpha)
 	\end{aligned}
 \end{equation}
 with a strict inequality if $m_{R_0}(\alpha)$ is reached.
 To prove that \eqref{Eq-mtha} holds, note that in view of Lemma \ref{Lemma-mR(a)}, for any $\varepsilon>0$ sufficiently small, there exists a $u \in \mathcal{U}_{R_0}(\alpha)$ such that
 \begin{equation}\label{Eq-Phiu<0}
 	J(u) \leq m_{R_0}(\alpha)+\varepsilon \quad \text { and } \quad J(u)<0.
 \end{equation}
 In view of Lemma \ref{Lemma-g}, $g(\alpha, R) \geq 0$ for any $R \in\left[\left( \frac{\alpha}{c} \right)^{2/\widetilde{q}}R_0, R_0\right]$. Hence, we can deduce from Lemma \ref{Lemma-g} and \eqref{Eq-Phiu<0} that
 $$
 \|\nabla u\|_{\widetilde{q}}<\left( \frac{\alpha}{c} \right)^{2/\widetilde{q}}R_0.
 $$
 Consider now $ v:=u\left(\cdot / \theta^{2 / N}\right) $. We first note that $\|v\|_2=\theta\|u\|_2=\theta \alpha$. and
 $$
 \|\nabla v\|_{\widetilde{q}}^{\widetilde{q}}=\theta^{2(N-q)/ N}\|\nabla u\|_{\widetilde{q}}^{\widetilde{q}}<\theta^{2}\left(\frac{\alpha}{c} \right)^{2}{R_0}^{\widetilde{q}}\leq{R_0}^{\widetilde{q}}
 $$
 Thus $v \in V(\theta \alpha)$ and we can write

 $$
 \begin{aligned}
 	m_{R_0}(\theta \alpha) & \leq J(v)\\
 	&=\frac{1}{2} \theta^{2(N-2)/ N}\|\nabla u\|_2^2-\frac{1}{q} \theta^{2(N-q)/ N}\|\nabla u\|_q^q-\theta^{2}\int_{\mathbb{R}^N} F(u) {d} x \\
 	& <\frac{1}{2} \theta^{2}\|\nabla u\|_2^2-\frac{1}{q} \theta^{2}\|\nabla u\|_q^q-\theta^{2}\int_{\mathbb{R}^N} F(u) {d} x \\
 	&=\theta^{2} J(u)\\
 	& \leq \theta^2(m_{R_0}(\alpha)+\varepsilon).
 \end{aligned}
 $$
 Since $\varepsilon>0$ is arbitrary, we have that $m_{R_0}(\theta \alpha) \leq \theta m_{R_0}(\alpha)$. If $m_{R_0}(\alpha)$ is reached then we can let $\varepsilon=0$ in \eqref{Eq-Phiu<0} and thus the strict inequality follows.
\end{proof}

Next we prove the compactness of minimizing sequences {of $J$ at level $m_{R_0}(c)$}.
	\begin{lemma}\label{LemmaPSseq}
		If \eqref{F0}, \eqref{F2}, and \eqref{Eq-crange} are satisfied. Let ${(\tilde{u}_n)_n} \subset \mathcal{U}_{R_0}(c)$ be a minimizing sequence for $J$ at level $m_{R_0}(c)$. Then, there exists another minimizing sequence ${(u_n)_n} \subset \mathcal{U}_{R_0}(c)$ bounded in $X$, and $\lambda \in \mathbb{R}$ such that {for all $\varphi \in X$}
	$$
	\left\|u_n-\tilde{u}_n\right\|_X \rightarrow 0,
	\quad J^{\prime}\left(u_n\right) \varphi+\lambda\int_{\mathbb{R}^N} u_n \varphi d x \rightarrow 0  \text { as } n \rightarrow+\infty.
	$$
	Moreover, if $\lim \limits _{n \rightarrow +\infty}\left\|u_n\right\|_2<c$, then $\lambda=0$.
\end{lemma}

\begin{proof}
	Let {$(\tilde{u}_n)_n$ be} a minimizing sequence for $J$ at level $m_{R_0}(c)$.
	{By} Ekeland's variational principle \cite[Theorem 2.4]{Willem1996Minimax}, we derive a new minimizing sequence ${(u_n)_n} \subset \mathcal{D}(c)$, that is also a Palais-Smale sequence for $J$ on $\mathcal{D}(c)$. By \cite[Proposition 5.12]{Willem1996Minimax}, there exist ${(\lambda_n)_n}\subset\mathbb{R}$ , such that
	{for all $\varphi \in X$}
	$$
	\left\|u_n-\tilde{u}_n\right\|_X \rightarrow 0,\quad J^{\prime}\left(u_n\right) \varphi+\lambda_n \int_{\mathbb{R}^N} u_n \varphi d x \rightarrow 0  \text { as } n \rightarrow+\infty.
	$$	
	Therefore, since ${(\tilde{u}_n)_n} \subset \mathcal{U}_{R_0}(c)$, when $n$ large enough ${({u}_n)_n} \subset \mathcal{U}_{R_0}(c)$.

	Now we prove that
	$({u}_n)_n$ is bounded in $X$.
	First of all,
	since $u_n \in \mathcal{U}_{R_0}\left(c\right)$, $\|\nabla {u}_n\|_{\widetilde{q}} $ and $\|u_n\|_2$ are bounded in $\mathbb{R}^+$. So from Lemma \ref{Lemma-GN inequality2} and Lemma
	\ref{Sobolev inequality},
	$$
	\|u_n\|_p \leq K_{N, p}\|\nabla u_n\|_{\widetilde{q}}^{\nu_{p, \widetilde{q}}}\|u_n\|_2^{\left(1-\nu_{p, \widetilde{q}}\right)}, \quad 	\mathcal{S}_{\widetilde{q}}\|u_n\|_{q^{\prime}}^{\widetilde{q}} \leq\|\nabla u_n\|_{\widetilde{q}}^{\widetilde{q}}.
	$$
	Thus, $\|u_n\|_p$ and $\|u_n\|_{q^{\prime}}$ are bounded in $\mathbb{R}^+$. From \eqref{F0}, $\int_{\mathbb{R}^N}F(u_n)dx$ is bounded in $\mathbb{R}$.
	Because $J\left(u_n\right) \rightarrow m_{R_0}(c)$ as $n \rightarrow \infty$, $ \|\nabla u_n\|_{\widehat{q}} $ is bounded in $  \mathbb{R}^+ $. Otherwise we have $J\left(u_n\right) \rightarrow +\infty$ as $n \rightarrow +\infty$. Therefore, We conclude that the sequence $\left(u_n\right)_n$ is bounded in $X$.

	Hence, there exists ${u\in X}$ such that, {up to a subsequence,}
	${u}_n \rightharpoonup {u}$ in $X$ and ${u}_n \rightarrow {u}$ in $L_{loc}^m\left(\mathbb{R}^N\right)$ for every {$m$,
		with} $2 \leq m<q^{\prime}$ and ${u}_n \rightarrow {u}$ for a.e. in $\mathbb{R}^N$. By Fatou's lemma, it follows that ${u} \in \mathcal{D}(c)$.
	Let $\varphi=u_n$, it is easy to show that {$(\lambda_n)_n$} is bounded in $\mathbb{R}$.
	we may assume $\lambda_n\rightarrow \lambda\text { as } n \rightarrow+\infty $, up to a subsequence if necessary. Hence
	$$
	J^{\prime}\left(u_n\right) \varphi+\lambda\int_{\mathbb{R}^N} u_n \varphi d x \rightarrow 0  \text { as } n \rightarrow+\infty.
	$$
	If  $\lim \limits _{n \rightarrow \infty}\left\|u_n\right\|_2<c$, {then} ${u} \in\mathcal{D}(c)\backslash\mathcal{S}(c) $ and is an interior point of $\mathcal{D}(c)$. {Therefore,} ${u}$ is a local minimizer of $J$ on $X$. Hence
	$$
	J^{\prime}\left({u}\right) \varphi=0 \quad \mbox{for all }\varphi \in {X},
	$$
	which implies that  $\lambda=0$.
\end{proof}
	{\begin{lemma}\label{Lem3-ae}
		If \eqref{F0}, \eqref{F2}, and \eqref{Eq-crange} are satisfied. Let ${(\tilde{u}_n)_n} \subset \mathcal{D}(c)$ be a minimizing sequence for $J$ at level $m_{R_0}(c)$.
		Then, there exists another minimizing sequence $({u}_n)_n \subset \mathcal{D}(c)$ for $J$, such that for some $u\in X$,
		$$
		\nabla u_n \rightarrow \nabla u \text { a.e. } \text{ in } \mathbb{R}^N.
		$$
\end{lemma}}
\begin{proof}
	From Lemma \ref{LemmaPSseq}, we know that there exists anthor bounded sequence $ ({u}_n)_n $ such that for any $v \in {X}$,
	\begin{equation}\label{Eq-aePhi1}
		\begin{aligned}
			o_n(1)=&\left\langle J^{\prime}\left(u_n\right), v \right\rangle+\lambda\int_{\mathbb{R}^N}  u_n vdx\\
			=&\int_{\mathbb{R}^N} \left( \nabla u_n\nabla v
			+ | \nabla u_n|^{q-2} \nabla u_n\nabla v
			+ \lambda u_n v\right)dx\\
			&-\int_{\mathbb{R}^N}f(u_n)v
			dx.\\
		\end{aligned}	
	\end{equation}
	Up to a subsequence, we may assume that
	${u}_n \rightharpoonup {u}$ in $X$.
	Therefore, for any $v \in {X}$,
	\begin{equation}\label{Eq-aePhi2}
		\left\langle J^{\prime}\left(u\right),
		v\right\rangle
		+\lambda\int_{\mathbb{R}^N}uvdx
		=\lim_{n \rightarrow \infty}\left( \left\langle J^{\prime}\left(u_n\right),
		v\right\rangle
		+\lambda\int_{\mathbb{R}^N}u_nvdx \right)
		=0.
	\end{equation}
	Now we use a technique due to Boccardo and Murat \cite{Boccardo1992}. Fix $k\in \mathbb{R}^{+}$, define the function
	$$
	\tau_k(s)= \begin{cases}s & \text { if }|s| \leq k, \\ k s /|s| & \text { if }|s|>k.\end{cases}
	$$
	It's easy to see that $(\tau_k\left(u_n-u\right))_{n} $ is bounded in $X$. Fix a function $\psi \in C_0^{\infty}\left(\mathbb{R}^N\right)$ with $0 \leq \psi \leq 1$ in $\mathbb{R}^N$, $\psi(x)=1$ for $x \in B_{1}(0)$ and $\psi(x)=0$ for $x \in \mathbb{R}^N \backslash B_2(0)$.\
	Now, take $R>0$ and define $\psi_R(x)=\psi(x / R)$ for $x \in \mathbb{R}^N$.
	We obtain from
	\eqref{Eq-aePhi1} and \eqref{Eq-aePhi2} that
	\begin{equation}\label{Eq-aePhi3}
		\begin{aligned}
			o_n(1)=&\left\langle J^{\prime}\left(u_n\right), \tau_k(u_n-u)\psi_R\right\rangle
			+\lambda\int_{\mathbb{R}^N}u_n\tau_k(u_n-u)\psi_Rdx\\
			=& \left\langle J^{\prime}\left(u_n\right)-J^{\prime}(u), \tau_k(u_n-u)\psi_R\right\rangle
			+\lambda\int_{\mathbb{R}^N} \left( u_n- u \right) \tau_k(u_n-u)\psi_Rdx \\
			=&\int_{\mathbb{R}^N} \left(\left|\nabla u_n\right|^{q-2} \nabla u_n-|\nabla u|^{q-2}\nabla u \right)\nabla \left(\tau_k(u_n-u)\psi_R\right)dx\\
			&+\int_{\mathbb{R}^N} \left(\nabla u_n-\nabla u \right)\nabla \left(\tau_k(u_n-u)\psi_R\right)dx\\
			&+\lambda\int_{\mathbb{R}^N} \left( u_n- u \right) \tau_k(u_n-u)\psi_Rdx\\
			&-\int_{\mathbb{R}^N} \left(f(u_n)-f(u) \right) \tau_k(u_n-u)\psi_Rdx\\
		\end{aligned}
	\end{equation}
	and
	\begin{equation}\label{Eq-aePhi3en}
		\begin{aligned}
			o_n(1)=&\left\langle J^{\prime}\left(u_n\right), (u_n-u)\psi_R\right\rangle
			+\lambda\int_{\mathbb{R}^N}u_n(u_n-u)\psi_Rdx\\
			=& \left\langle J^{\prime}\left(u_n\right)-J^{\prime}(u), (u_n-u)\psi_R\right\rangle
			+\lambda\int_{\mathbb{R}^N} \left( u_n- u \right) (u_n-u)\psi_Rdx \\
			=&\int_{\mathbb{R}^N} \left(\left|\nabla u_n\right|^{q-2} \nabla u_n-|\nabla u|^{q-2}\nabla u \right)\nabla \left((u_n-u)\psi_R\right)dx\\
			&+\int_{\mathbb{R}^N} \left(\nabla u_n-\nabla u \right)\nabla \left((u_n-u)\psi_R\right)dx\\
			&+\lambda\int_{\mathbb{R}^N} \left( u_n- u \right)^2\psi_Rdx\\
			&-\int_{\mathbb{R}^N} \left(f(u_n)-f(u) \right) (u_n-u)\psi_Rdx.\\
		\end{aligned}
	\end{equation}
	Since  $ ({u}_n)_n $  is bounded in $X$, {up to a subsequence}, we have
	\begin{equation}\label{Eq-aePhi4}
		\begin{aligned}
			\int_{\mathbb{R}^N} \left( u_n- u \right) \tau_k(u_n-u)\psi_Rdx=o_n(1) \quad \text{and} \quad		\int_{\mathbb{R}^N} \left( u_n- u \right)^2\psi_Rdx=o_n(1)
		\end{aligned}.
	\end{equation}
	From \eqref{F0}, there exists $C_1>0$ such that
	\begin{equation}\label{eq-g1}
		|f(s)| \leq C_1\left(|s|+|s|^{{q}^{\prime}-1}\right) \text{for all } s \in \mathbb{R}.
	\end{equation}
	Therefore, from \eqref{eq-g1} and Lemma \ref{LemmaEmbed}, we have
	\begin{equation}\label{eq-g2}
		\int_{\mathbb{R}^N}|f(u_n)u_n|dx\leq C_1\int_{\mathbb{R}^N}\left(|u_n|^2+|u_n|^{{q}^{\prime}}\right)dx.
	\end{equation}
	Similarly, we can prove, there exists $C^{\prime}>0$ such that
	{$$
		\int_{\mathbb{R}^N}|f(u)u|dx\leq C^{\prime},\,
		\int_{\mathbb{R}^N}|f(u)u_n|dx\leq C^{\prime} \quad \text{ and } \quad
		\int_{\mathbb{R}^N}|f(u_n)u|dx\leq C^{\prime}.
		$$}
	Hence, there exists $C>0$ such that
	\begin{equation}\label{Eq-aePhi5}
		\begin{aligned}
			&\quad\int_{\mathbb{R}^N} |\left(f(u_n)- f(u)\right) (u_n-u)\psi_R|dx\\
			&\leq\int_{\mathbb{R}^N} |\left(f(u_n)- f(u)\right) (u_n-u)|dx\\
			&\leq\int_{\mathbb{R}^N} |f(u)u| +|f(u_n)u_n|+|f(u_n)u|+|f(u)u_n|dx\\
			&\leq C.	
		\end{aligned}
	\end{equation}
	From \eqref{eq-g1}, we may assume that $C$ large enough such that
	{\begin{equation}\label{Eq-aePhi5ekn}
			\begin{aligned}
				&\quad\int_{\mathbb{R}^N} |\left(f(u_n)- f(u)\right) \tau_k(u_n-u)\psi_R|dx\\
				&\leq k	\int_{\mathbb{R}^N} |\left(f(u_n)- f(u)\right) \psi_R|dx\\
				&\leq k	\int_{B_{2R}(0)} |f(u_n)- f(u)|dx\\
				&\leq Ck.	
			\end{aligned}
	\end{equation}}
	Therefore,
	let
	\begin{equation*}
		\begin{aligned}
			e_n(x)=& \left(\left|\nabla u_n\right|^{q-2} \nabla u_n-|\nabla u|^{q-2}\nabla u \right)\nabla \left((u_n-u)\psi_R\right)\\&+
			\left(\nabla u_n-\nabla u \right)\nabla \left((u_n-u)\psi_R\right)
		\end{aligned}
	\end{equation*}
	and
	\begin{equation*}
		\begin{aligned}
			e_{k,n}(x)=& \left(\left|\nabla u_n\right|^{q-2} \nabla u_n-|\nabla u|^{q-2}\nabla u \right)\nabla \left(\tau_k(u_n-u)\psi_R\right)\\&+
			\left(\nabla u_n-\nabla u \right)\nabla \left(\tau_k(u_n-u)\psi_R\right).
		\end{aligned}
	\end{equation*}
	First, we give some estimates for
	$$
	\int_{\mathbb{R}^N} e_n(x) d x \quad \text { and } \quad \int_{\mathbb{R}^N} e_{k, n}(x) d x .
	$$
	From \eqref{Eq-aePhi3en} and \eqref{Eq-aePhi5}, we have
	\begin{equation}\label{enR}
		\int_{\mathbb{R}^N}e_{n}(x)dx\leq C+o_n(1).
	\end{equation}
	And from \eqref{Eq-aePhi3} and \eqref{Eq-aePhi5ekn}, we have
	\begin{equation}\label{eknR}
		\int_{\mathbb{R}^N}e_{k,n}(x)dx\leq Ck+o_n(1).
	\end{equation}
	Next, we give some estimates for
	$$
	\int_{B_{2 R}(0) \backslash B_R(0)} e_n(x) d x \quad \text { and } \quad \int_{B_{2 R}(0) \backslash B_R(0)} e_{k, n}(x) d x.
	$$
	{We may assume that there exist $C_R>0$ such that $|\nabla\psi_R|<C_R$. Then
		\begin{equation*}
			\begin{aligned}
				&\quad\left|\int_{B_{2R}(0) \backslash B_R(0)}\left|\nabla u_n\right|^{q-2} \nabla u_n\left((u_n-u)\nabla\psi_R\right)dx\right|\\
				&\leq
				\int_{B_{2R}(0)\backslash B_R(0)}\left|\nabla u_n\right|^{q-1} |u_n-u||\nabla\psi_R|dx\\
				&\leq C_R\int_{B_{2R}(0) \backslash B_R(0)}\left|\nabla u_n\right|^{q-1}\left|u_n-u\right|dx\\
				&\leq C_R \left(\int_{B_{2R}(0) \backslash B_R(0)}\left|\nabla u_n\right|^{q}dx \right)^{\frac{q-1}{q}} \left(  \int_{B_{2R}(0) \backslash B_R(0)}\left|u_n-u\right|^{q}dx\right)^{\frac{1}{q}} \\
				&=o_n(1).
			\end{aligned}
		\end{equation*}
		Similarly, we can prove
		$$
		\left|\int_{B_{2R}(0) \backslash B_R(0)}          \left|\nabla u\right|^{q-2}    \nabla u\left((u_n-u)\nabla\psi_R\right)dx\right|=o_n(1),
		$$
		$$
		\left|\int_{B_{2R}(0) \backslash B_R(0)}          \nabla u_n\left((u_n-u)\nabla\psi_R\right)dx\right|=o_n(1),
		$$
		$$
		\left|\int_{B_{2R}(0) \backslash B_R(0)}          \nabla u\left((u_n-u)\nabla\psi_R\right)dx\right|=o_n(1).
		$$
		Therefore
		\begin{equation}\label{Eq-eue}
			\begin{aligned}
				\int_{B_{2R}(0) \backslash B_R(0)}e_{n}(x)dx
				=&\int_{B_{2R}(0) \backslash B_R(0)} \left(\left|\nabla u_n\right|^{q-2} \nabla u_n-|\nabla u|^{q-2}\nabla u \right)\nabla \left((u_n-u)\psi_R\right)dx\\
				&+\int_{B_{2R}(0) \backslash B_R(0)}\left(\nabla u_n-\nabla u \right)\nabla \left((u_n-u)\psi_R\right)dx\\
				=&\int_{B_{2R}(0) \backslash B_R(0)} \left(\left|\nabla u_n\right|^{q-2} \nabla u_n-|\nabla u|^{q-2}\nabla u \right)(\nabla u_n-\nabla u)\psi_Rdx\\
				&+\int_{B_{2R}(0) \backslash B_R(0)}\left(\nabla u_n-\nabla u \right) (\nabla u_n-\nabla u)\psi_Rdx\\
				&+\int_{B_{2R}(0) \backslash B_R(0)} \left(\left|\nabla u_n\right|^{q-2} \nabla u_n-|\nabla u|^{q-2}\nabla u \right) \left((u_n-u)\nabla\psi_R\right)dx\\
				&+\int_{B_{2R}(0) \backslash B_R(0)}\left(\nabla u_n-\nabla u \right) \left((u_n-u)\nabla\psi_R\right)dx\\
				\geq&
				\int_{B_{2R}(0) \backslash B_R(0)} \left(\left|\nabla u_n\right|^{q-2} \nabla u_n-|\nabla u|^{q-2}\nabla u \right) \left((u_n-u)\nabla\psi_R\right)dx\\
				&+\int_{B_{2R}(0) \backslash B_R(0)}\left(\nabla u_n-\nabla u \right) \left((u_n-u)\nabla\psi_R\right)dx\\
				=&o_n(1).
			\end{aligned}
	\end{equation}}
	Hence
	\begin{equation}\label{enB21}
		\int_{B_{2R}(0) \backslash B_R(0)}e_{n}(x)dx\geq o_n(1).
	\end{equation}	
	Using the same proof, we obtain
	\begin{equation}\label{eknB21}
		\int_{B_{2R}(0) \backslash B_R(0)}e_{k,n}(x)dx\geq o_n(1).
	\end{equation}	
	{Finally, we give some estimates for
		$$
		\int_{B_{R}(0) } e_n(x) d x \quad \text { and } \quad \int_{B_{R}(0)} e_{k, n}(x) d x.
		$$
		Combining \eqref{enR} and \eqref{enB21}, we obtain that
		\begin{equation}\label{Eq-aePhi6en}
			\begin{aligned}
				\int_{B_R(0)}e_{n}(x)dx
				=&\int_{\mathbb{R}^N}e_{n}(x)dx-\int_{B_{2R}(0) \backslash B_R(0)}e_{n}(x)dx\\
				\leq& C+o_n(1).
			\end{aligned}
		\end{equation}
		Combining \eqref{eknR} and \eqref{eknB21}, we obtain that
		\begin{equation}\label{Eq-aePhi6}
			\begin{aligned}
				\int_{B_R(0)}e_{k,n}(x)dx
				=&\int_{\mathbb{R}^N}e_{k,n}(x)dx-\int_{B_{2R}(0) \backslash B_R(0)}e_{k,n}(x)dx\\
				\leq& Ck+o_n(1).
			\end{aligned}
	\end{equation}}
	Take $0<\theta<1$ and split $B_R(0)$ into
	$$
	S_n^k=\left\{x \in B_R(0)	|\enspace| u_n-u | \leq k\right\}, \quad G_n^k=\left\{x \in B_R(0)| \enspace| u_n-u |>k\right\} .
	$$
	By Lemma \ref{Lem-aee}, $e_n(x) \geq 0$ and $e_{k,n}(x) \geq 0$ in $B_R(0)$, therefore
	\begin{equation}
		\begin{aligned}
			\int_{B_R(0)} e_n^\theta d x & =
			\int_{S_n^k} e_n^\theta d x+\int_{G_n^k} e_n^\theta d x \\
			& \leq\left(\int_{S_n^k} e_n d x\right)^\theta\left|S_n^k\right|^{1-\theta}+\left(\int_{G_n^k} e_n d x\right)^\theta\left|G_n^k\right|^{1-\theta}\\
			& =\left(\int_{S_n^k} e_{k,n} d x\right)^\theta\left|S_n^k\right|^{1-\theta}+\left(\int_{G_n^k} e_n d x\right)^\theta\left|G_n^k\right|^{1-\theta}.
		\end{aligned}
	\end{equation}
	For fixed $k\in \mathbb{R}^{+}$, $\left|G_n^k\right| \rightarrow 0$ as $n \rightarrow \infty$, and from \eqref{Eq-aePhi6en} and \eqref{Eq-aePhi6}, we get
	\begin{equation}
		\begin{aligned}
			\int_{B_R(0)} e_n^\theta d x
			& \leq\left(\int_{S_n^k} e_{k,n} d x\right)^\theta\left|S_n^k\right|^{1-\theta}+\left(\int_{G_n^k} e_n d x\right)^\theta\left|G_n^k\right|^{1-\theta}\\
			& \leq\left(\int_{S_n^k} e_{k,n} d x\right)^\theta\left|S_n^k\right|^{1-\theta}+o_n(1)\\
			&  \leq(C k)^\theta|B_R(0)|^{1-\theta}+o_n(1).
		\end{aligned}
	\end{equation}
	Let $k\rightarrow 0$, we get that $e_n^\theta \rightarrow 0$ in $L^1(B_R(0))$ as $n\rightarrow \infty$. By Lemma \ref{Lemma-Converae}, we have
	$$
	\nabla u_n \rightarrow \nabla u \text { a.e. } \text{ in } B_R(0).
	$$
	Since $R$ is arbitrary,
	by passing to a subsequence, we have
	$$
	\nabla u_n \rightarrow \nabla u \text { a.e. } \text{ in } \mathbb{R}^N.
	$$
\end{proof}

\begin{lemma}\label{LemmaStrongconver}
		If \eqref{F0}-\eqref{F2}, and \eqref{Eq-crange} are satisfied. If  $\tilde{u}_n \in \mathcal{U}_{R_0}(c)$ {is} such that $J\left(\tilde{u}_n\right) \rightarrow m_{R_0}(c)$, then there exists  another minimizing sequence ${(u_n)_n \subset} \mathcal{U}_{R_0}(c)$ such that $\left\|u_n-\tilde{u}_n\right\|_X \rightarrow 0$ and up to translations, $u_n \rightarrow u$ in $L^m\left(\mathbb{R}^N\right)$ for {all $m$, with} $2<m<q^{\prime}$.
\end{lemma}
\begin{proof}
	Let  ${(\tilde{u}_n)_n} \subset \mathcal{U}_{R_0}(c)$ be a minimizing sequence of $J(u)$ at level $m_{R_0}(c)$,
	{by Lemma~\ref{LemmaPSseq}}, $J$ possesses another minimizing sequence ${(u_n)_n} \subset \mathcal{D}(c)$ at level $m_{R_0}(c)$ and $\lambda \in \mathbb{R}$ such that
	{for all $\varphi \in X$}
	$$
	\left\|u_n-\tilde{u}_n\right\|_X \rightarrow 0, \quad J^{\prime}\left(u_n\right) \varphi+\lambda \int_{\mathbb{R}^N} u_n \varphi d x \rightarrow 0  \text { as } n \rightarrow+\infty ,
	$$
	and {$(u_n)_n$} is also a Palais-Smale sequence for $J$ on $\mathcal{D}(c)$. Since
	$
	\left\|u_n-\tilde{u}_n\right\|_X \rightarrow 0 \quad \text{ as } n \rightarrow + \infty,
	$
	{$(u_n)_n\in \mathcal{U}_{R_0}(c)$} for $n$ large enough.
	Similarly to the proof of Lemma \ref{LemmaPSseq} we have that {$(u_n)_n$} is bounded in $X$. If
	$$
	\lim _{n \rightarrow \infty} \sup _{y \in \mathbb{R}^N} \int_{B_R(y)}\left|u_n(x)\right|^2 d x=0
	$$
	for any $R>0$, due to Lemma \ref{LemmaLions},  $\left\|u_n\right\|_m \rightarrow 0$ for any
	{$m$, with}
	$2<m<q^{\prime}$. Fix $m \in\left(2,q^{\prime}\right)$ and $\varepsilon>0$. From \eqref{F1} and \eqref{Eq-C0}, there exists $C=C\left(q, \varepsilon, C_0\right)>0$ such that for every $t \in \mathbb{R}$
	$$
	F(t) \leq \varepsilon t^2+C|t|^m+C_0|t|^{q^{\prime}}.
	$$
	Recalling that $g\left(c, R_0\right)=0$, for $\varepsilon \ll 1$, using Lemma \ref{Lemma-g}, since $   g(c, R_0)=0 $, there holds
	
	$$
	\begin{aligned}
		0 & >m_{R_0}(c)=\lim _{n\rightarrow+\infty} J\left(u_n\right) \geq \limsup _{n\rightarrow+\infty}\left(\frac{1}{\widetilde{q}}\left\|\nabla u_n\right\|_{\widetilde{q}}^{\widetilde{q}}-\varepsilon\left\|u_n\right\|_2^2-C\left\|u_n\right\|_m^m-C_0\left\|u_n\right\|_{q^{\prime}}^{q^{\prime}}\right) \\
		& \geq\left(\frac{1}{\widetilde{q}}-
		{C_0}{\mathcal{S}_{{\widetilde{q}}}^{-q^{\prime} / {\widetilde{q}}}}R_0^{q^{\prime}-{\widetilde{q}}}\right) \limsup_ {n\rightarrow+\infty}\left\|\nabla u_n\right\|_{\widetilde{q}}^{\widetilde{q}}-\varepsilon c^2=
		{C_0 c^2}{R_0^{-\widetilde{q}}} \limsup _{n\rightarrow+\infty}\left\|\nabla u_n\right\|_{\widetilde{q}}^{\widetilde{q}}-\varepsilon c^2 \geq-\varepsilon c^2.
	\end{aligned}
	$$
	{Thus,} $m_{R_0}(c)=\lim \limits_{n \rightarrow+\infty} J\left(u_n\right) \geq 0$, which contradicts Lemma~\ref{Lemma-mR(a)}.
	Hence, there exist   $\varepsilon_0>0$ and a sequence $(y_n)_n \subset \mathbb{R}^N$ such that, {for sufficiently large $R>0$}
	$$
	\int_{B_R\left(y_n\right)}\left|u_n(x)\right|^2 d x \geq \varepsilon_0>0.
	$$
	Moreover, we have $u_n\left(x+y_n\right) \rightharpoonup \bar{u} \not \equiv 0$ in $X$ for some {$\bar{u} \in X$}.
	{Put} $v_n(x):=u_n\left(x+y_n\right)-\bar{u}(x)$.
	{Then,} $v_n \rightharpoonup 0$ in $X$,  and  $u_n\left(x+y_n\right)\rightarrow \bar{u}$ for a.e. $x\in \mathbb{R}^N$ by Lemma~\ref{LemmaEmbed}. Therefore, we obtain
	\begin{equation*}
		\begin{aligned}
			& \left\|\nabla u_n\right\|_2^2=\left\|\nabla u_n\left(\cdot+y_n\right)\right\|_2^2=\left\|\nabla v_n\right\|_2^2+\left\|\nabla \bar{u}\right\|_2^2+o_n(1),\\
			&\left\|u_n\right\|_2^2=\left\|u_n\left(\cdot+y_n\right)\right\|_2^2=\left\|v_n\right\|_2^2+\left\|\bar{u}\right\|_2^2+o_n(1).
		\end{aligned}
	\end{equation*}
	Moreover, from the Br{\'e}zis-Lieb lemma \cite[Lemma 3.2]{Jeanjean2019Nonradial} and Lemma \ref{Lem3-ae}, we have
	\begin{equation*}
		\begin{aligned}
			\int_{\mathbb{R}^N}f\left(u_n\right)dx=&	\int_{\mathbb{R}^N}f\left(u_n\left(\cdot+y_n\right)\right)dx=\int_{\mathbb{R}^N}f\left(v_n\right)dx+\int_{\mathbb{R}^N}f(\bar{u})dx+o_n(1) ,\\
			& \left\|\nabla u_n\right\|_q^q=\left\|\nabla u_n\left(\cdot+y_n\right)\right\|_q^q=\left\|\nabla v_n\right\|_q^q+\left\|\nabla \bar{u}\right\|_q^q+o_n(1).
		\end{aligned}
	\end{equation*}
	We next claim that
	$$
	\lim_{n \rightarrow\infty} \sup _{y \in \mathbb{R}^N} \int_{B_1(y)}\left|v_n\right|^2 d x=0,
	$$
	which, by Lemma \ref{LemmaLions}, will yield the statement of Lemma \ref{LemmaStrongconver}. If this is not true, then, as before, there exist $z_n \in \mathbb{R}^N$ and $\bar{v} \in X \backslash\{0\}$ such that, denoting $w_n(x):=v_n\left(x-z_n\right)-\bar{v}(x)$, we have $w_n \rightharpoonup 0$ in $X, w_n \rightarrow 0$ for a.e. $x\in \mathbb{R}^N$, and
	$$
	\lim_{n \rightarrow\infty}
	\left(J\left(v_n\right)-J\left(w_n\right)
	\right)=J(v) .
	$$
	Note that, once more due to the Br{\'e}zis-Lieb lemma,$$
		\lim_{n \rightarrow\infty}\left( \left\|u_n\right\|_2^2-\left\|w_n\right\|_2^2\right)=\lim_{n \rightarrow\infty}\left(\left\|u_n\right\|_2^2-\left\|v_n\right\|_2^2+\left\|v_n\right\|_2^2-\left\|w_n\right\|_2^2\right)=\|\bar{u}\|_2^2+\|\bar{v}\|_2^2,
		$$
	whence, denoting $\beta:=\|\bar{u}\|_2>0$ and $\gamma:=\|\bar{v}\|_2>0$, there holds
	$$
	c^2-\beta^2-\gamma^2 \geq \liminf _n\left\|u_n\right\|_2^2-\beta^2-\gamma^2=\liminf _n\left\|w_n\right\|_2^2=: \delta^2 \geq 0 .
	$$
	If $\delta>0$, then let us set $\tilde{w}_n:=\frac{\delta}{\left\|w_n\right\|_2} w_n \in \mathcal{S}(\delta)$. {Via explicit computations}, we have
	$$\lim \limits_{n \rightarrow+\infty} \big(J\left(w_n\right)
	-J\left(\tilde{w}_n\right)\big)=0.$$
	Hence, together with Lemma \ref{Lemma-subadd} and since $m_{R_0}(c)$ is non-increasing with respect to $c>0$,
	\begin{equation}\label{eqSub}
		\begin{aligned}
			m_{R_0}(c) & =\lim \limits_{n \rightarrow+\infty} J\left(u_n\right)\\
			&=J(\bar{u})+J(\bar{v})+\lim_{n \rightarrow+\infty} J\left(w_n\right)\\
			&=J(\bar{u})+J(\bar{v})+\lim_{n \rightarrow+\infty} J\left(\tilde{w}_n\right) \\
			& \geq m(\beta)+m(\gamma)+m(\delta)\\
			& \geq m\left(\sqrt{\beta^2+\gamma^2+\delta^2}\right) \\
			&\geq m_{R_0}(c) .
		\end{aligned}
	\end{equation}
	Thus all the inequalities in \eqref{eqSub} are in fact equalities and, in particular, $J(\bar{u})=m(\beta)$ and $J(\bar{v})=$ $m(\gamma)$. Therefore Lemma \ref{Lemma-subadd} yields {that} $m(\beta)+m(\gamma)+m(\delta)>m\left(\sqrt{\beta^2+\gamma^2+\delta^2}\right)$, which contradicts \eqref{eqSub}.
	
	If $\delta=0$, then $w_n \rightarrow 0$ in $L^2\left(\mathbb{R}^N\right)$, which implies {that}
	$$\lim \limits_{n \rightarrow\infty} \int_{\mathbb{R}^N} F\left(w_n\right) d x=0,$$
	whence $\lim \limits \inf _{n\rightarrow+\infty} J\left(w_n\right) \geq 0$. Then \eqref{eqSub} becomes
	$$
	m_{R_0}(c)=\lim_{n \rightarrow\infty} J\left(u_n\right)=J(\bar{u})+J(\bar{v})+\lim_{n \rightarrow\infty} J\left(w_n\right) \geq m(\beta)+m(\gamma) \geq m\left(\sqrt{\beta^2+\gamma^2}\right)=m_{R_0}(c)
	$$
	and we get a contradiction as before.
\end{proof}

\textbf{Proof of Theorem \ref{Theorem-existence}}:
Let  $\tilde{u}_n \in \mathcal{U}_{R_0}(c)$ be a minimizing sequence of $J$ at level $m_{R_0}(c)$. From Lemma \ref{LemmaStrongconver}, there exists $\bar{u} \in \mathcal{D}(c)$ with $\|\nabla \bar{u}\|_{\widetilde{q}} \leq R_0$ such that, there exists another mimimizing sequence ${u}_n \in \mathcal{U}_{R_0}(c)$, $u_n \rightharpoonup \bar{u}$ in $X$ and $u_n \rightarrow \bar{u}$ in $L^m\left(\mathbb{R}^N\right)$ for $2<m<q^{\prime}$  and a.e. in $\mathbb{R}^N$, And by Remark \ref{Remark-g>0}, $J(\bar{u}) \geq m_{R_0}(c)$.

Now we prove that $u_n \rightarrow \bar{u}$ in $X$.
Denote $v_n:=u_n-\bar{u}$. From the Lemma \ref{Lem3-ae},
$$
\lim_{n \rightarrow +\infty}\left(\left\|\nabla u_n\right\|_{\widetilde{q}}^{\widetilde{q}}-\left\|\nabla v_n\right\|_{\widetilde{q}}^{\widetilde{q}}\right)=\|\nabla \bar{u}\|_{\widetilde{q}}^{\widetilde{q}}>0,
$$
thus $\left\|\nabla v_n\right\|_{\widetilde{q}}<R_0$ for $n \gg 1$. Moreover, from the Br{\'e}siz-Lieb lemma \cite[Theorem 2]{Brezis1983BL},
which, together with $g\left(c, R_0\right)=0$, and $\lim_{n \rightarrow \infty}\left\|v_n\right\|_2=0$, implies
$$
\begin{aligned} 0 & \geq \lim_{n \rightarrow \infty} J\left(v_n\right)\\
	 &\geq \lim_{n \rightarrow \infty}\left(\frac{1}{{\widetilde{q}}}\left\|\nabla v_n\right\|_{\widetilde{q}}^{\widetilde{q}}-C_0\left(\left\|v_n\right\|_2^2+\left\|v_n\right\|_{q^{\prime}}^{q^{\prime}}\right)\right) \\
	  & \geq\left(\frac{1}{\widetilde{q}}-\frac{C_0 R_0^{{q^{\prime}}-\widetilde{q}}}{\mathcal{S}_{\widetilde{q}}^{{q^{\prime}}/ \widetilde{q}}}\right) \lim_{n \rightarrow \infty}\left\|\nabla v_n\right\|_{\widetilde{q}}^{\widetilde{q}}\\
	 &=\frac{C_0 c^2}{R_0^{\widetilde{q}}} \lim_{n \rightarrow \infty}\left\|\nabla v_n\right\|_{\widetilde{q}}^{\widetilde{q}}\\
	 & \geq 0,\end{aligned}
$$
i.e., $\lim_{n \rightarrow \infty}\left\|\nabla v_n\right\|_{\widetilde{q}}=0$. And
it's easy to see that $\lim_{n \rightarrow \infty}\left\|\nabla v_n\right\|_{\widehat{q}}=0$.
Therefore, we obtain that $u_n \rightarrow \bar{u}$ in $X$.

Hence, $J(\bar{u})=\lim_{n \rightarrow \infty} J\left(u_n\right)=m_{R_0}(c)<0$ and  from Remark \ref{Remark-g>0}, $\|\nabla \bar{u}\|_{\widetilde{q}}<R_0$. Consequently, there exists $\lambda_{\bar{u}} \in \mathbb{R}$ such that
$$
-\Delta \bar{u}-\Delta_q\bar{u}+\lambda_{\bar{u}} \bar{u}=f(\bar{u}) \quad \text { in } \mathbb{R}^N \text {. }
$$
If $\lambda_{\bar{u}} \leq 0$, then from the Pohozaev identity \eqref{Eq-poH3} we obtain
$$
\frac{(N-2)}{2} \int_{\mathbb{R}^N}|\nabla \bar{u}|^2 d x+\frac{(N-q)}{q} \int_{\mathbb{R}^N}|\nabla \bar{u}|^q d x=\frac{N}{2} \int_{\mathbb{R}^N}\left(2 F(\bar{u})-\lambda_{\bar{u}}|\bar{u}|^2\right) \mathrm{d} x \geq 2 N \int_{\mathbb{R}^N} F(\bar{u}) d x
$$
and so, by \eqref{Eq-Functional}, $0>J(\bar{u}) \geq(\|\nabla u\|_2^2+\|\nabla u\|_q^q) / N\geq0$, which is a contradiction. This implies that $\lambda_{\bar{u}}>0$ and, consequently, $\bar{u} \in \mathcal{S}(c)$, otherwise we would have $\lambda_{\bar{u}}=0$.

Now we prove that $\bar{u}$ has constant sign.
Let $\bar{u}_{ \pm}:=\max \{ \pm \bar{u}, 0\}$ and $c_{ \pm}:=\left\|\bar{u}_{ \pm}\right\|_2$. Thus, $\pm \bar{u}_{ \pm} \in \mathcal{U}_{R_0}\left(c_{ \pm}\right) \subset \mathcal{U}_{R_0}(c)$. If $u_{+} \not\equiv 0, u_{-} \not\equiv 0$, then, from Lemma \ref{Lemma-subadd} and $c^2=c_{+}^2+c_{-}^2$, we have
$$
m_{R_0}(c)=J(\bar{u})=J\left(\bar{u}_{+}\right)+J\left(-\bar{u}_{-}\right) \geq m_{R_0}\left(c_{+}\right)+m_{R_0}\left(c_{-}\right) \geq m_{R_0}(c) .
$$
Hence $m_{R_0}\left(c_{+}\right)$is attained at $\bar{u}_{+}, m_{R_0}\left(c_{-}\right)$is attained at $-\bar{u}_{-}$, and, again from Lemma \ref{Lemma-subadd}, $m_{R_0}\left(c_{+}\right)+$ $m_{R_0}\left(c_{-}\right)>m_{R_0}(c)$, a contradiction.

\textbf{Proof of Proposition \ref{Prop-ground}}:
 We shall prove that $m_{R_0}(c)=\inf \{J(u) \mid u \in \mathcal{D}(c) \quad \text{and} \quad J|_{\mathcal{D}(c)} ^{\prime}(u)=0\}$.
Since, from Theorem \ref{Theorem-existence}, $m_{R_0}(c)$ is attained, we have  $$m_{R_0}(c)\geq\inf \{J(u) \mid u \in \mathcal{D}(c) \quad \text{and} \quad J|_{\mathcal{D}(c)} ^{\prime}(u)=0\}.$$ Assume by contradiction that strict inequality holds, there exists $u \in \mathcal{D}(c) \backslash\{0\}$ such that $\left.J\right|_{\mathcal{D}(c)} ^{\prime}(u)=0$ and $J(u)<m_{R_0}(c)$. From the definition of $m_{R_0}(c)$, there holds $\|\nabla u\|_{\widetilde{q}} \geq R_0$. In fact, since $J(u)<0$, we know from \eqref{Eq-J>g} and \eqref{g2} that $\|\nabla u\|_{\widetilde{q}}>R_1$.

Consider the function $\psi:(-\infty,\infty) \rightarrow \mathbb{R}, \psi(s):=J(s * u)$. Let us recall that, since every critical point of $\left.J\right|_{\mathcal{D}(c)}$ belongs to $\mathcal{P}, \psi^{\prime}(0)=0$ by \eqref{Eq-poH3}. Again from \eqref{Eq-J>g} and \eqref{g2}, $ \psi$ is positive on $\left(\ln\left(R_0 /\|\nabla u\|_{\widetilde{q}}\right), \ln\left(R_1 /\|\nabla u\|_{\widetilde{q}}\right)\right)$.
And from $J(u)<0$, $ \psi(s)$ is
negative at $s=0$. Moreover,
using Remark \ref{remark-urad}
arguing as in the proof of Lemma \ref{Lemma-mR(a)},
$\psi(s)<0$ for $s \ll 0$. Thus $\psi$ has a local maximum point $t_u \in(-\infty,0)$. From \eqref{J0}, $\psi^{\prime}<0$ in a right-hand neighbourhood of $t_u$, hence, from \eqref{J1}, $\psi^{\prime}<0$ in $\left(t_u,\infty\right)$, in contradiction with $\psi^{\prime}(0)=0$.

\section{Positive-energy solution}\label{Sect4}
Section \ref{Sect3} is devoted to finding a positive-energy solution of \eqref{Eq-Equation}.
	In this section, we present the weakest assumptions necessary for each lemma.
\subsection{Properties of $\mathcal{P}_{-}$}
	\
	\newline
	To address the minimization problem on \( \mathcal{P}_{-}^{\mathrm{rad}} \cap \mathcal{D}(c) \), we first demonstrate that \( \mathcal{P} \) is a nonempty \( \mathcal{C}^1 \) manifold. In the following lemma, we will need the following assumption:
\begin{enumerate}[label=(H\arabic*), ref=\textup{H\arabic*}]
	\setcounter{enumi}{+1}
	\item \label{H2}  there exists $\xi \neq 0$ such that $H(\xi)>0$.
\end{enumerate}
\eqref{H2} follows from \eqref{F1}, \eqref{F2} and \eqref{H-1}.
\begin{lemma}\label{Lemma-Pnotempty}
	If \eqref{F0} and \eqref{H2} hold, then $\mathcal{P}$ is non-empty.
\end{lemma}

\begin{proof}
	Let $\xi\neq 0$ be such that $H(\xi)>0$. For $R>1$, define
	$$
	w_R(x)= \begin{cases}\xi & \text { for }|x| \leq R-1 \\ \xi(R+1-|x|) & \text { for } R-1<|x|<R \\ 0 & \text { for }|x| \geq R.\end{cases}
	$$
	As shown in \cite[page 325]{Berestycki1983NSF},  there exists $R_1>0$ such that
	 $\int_{\mathbb{R}^N} H(w_{R_1}(x)) {d} x>0$. Let $u(x):=w_{R_1}(x)$,
	it is easy to show through direct calculation that
	$u \in X$.  Define  $G: \mathbb{R}^+\rightarrow \mathbb{R}$ as follows:
	\begin{equation}\label{Eq-G}
		G(t) :=t^2 \int_{\mathbb{R}^N}|\nabla u|^2 \, dx + t^q\left(\delta_q + 1\right) \int_{\mathbb{R}^N}|\nabla u|^q \, dx - \frac{N}{2} \int_{\mathbb{R}^N} H(u) \, dx .
	\end{equation}
	Clearly, $G(0)<0$ and $\lim_{t \rightarrow +\infty}G(t)>0$.
	 Since
	$G(t)$ is strictly increasing in $\mathbb{R}^+$, we can easily deduce that there exists a unique $t(u) \in \mathbb{R}^+$ such that $G(t(u))=0$. This implies that
	$$u(t(u) \cdot) \in \mathcal{P},$$
	where
	$\mathcal{P}:=\{u \in X \backslash\{0\}: P(u)=0\}$.
	Hence $\mathcal{P}$ is nonempty.
\end{proof}

\begin{lemma}\label{Lemma-codim}
	If  $H(t)$ is of class $\mathcal{C}^1$ on $\mathbb{R}$, \eqref{F0} and \eqref{H2} hold, then
	$\mathcal{P}$ is a $\mathcal{C}^1$-manifold of codimension 1 in $X$.
\end{lemma}
\begin{proof}
	Suppose that $u \in \mathcal{P}$, then $P(u)=0$. Suppose by contradiction that $P^{\prime}(u)=0$, then $u$ is a weak solution to the following equation
	\begin{equation}
		-2 \Delta u-q\left(\delta_q+1\right) \Delta_q u-\frac{N}{2} h(u)=0.
	\end{equation}
	From the Pohozaev identity, we get that
	\begin{equation}\label{eq43}
		\frac{2}{2^*} \int_{\mathbb{R}^N}|\nabla u|^2 d x+\frac{q\left(\delta_q+1\right)}{q^*} \int_{\mathbb{R}^N}|\nabla u|^q d x=\frac{N}{2} \int_{\mathbb{R}^N} H(u) d x.
	\end{equation}
	Combining \eqref{eq43} with \eqref{Eq-poH3} we have
	\begin{equation}
		\frac{2}{2^*} \int_{\mathbb{R}^N}|\nabla u|^2 d x+\frac{q\left(\delta_q+1\right)}{q^*} \int_{\mathbb{R}^N}|\nabla u|^q d x=\int_{\mathbb{R}^N}|\nabla u|^2 \, dx + \left(\delta_q + 1\right) \int_{\mathbb{R}^N}|\nabla u|^q \, dx.
	\end{equation}
	Since $	\frac{2}{2^*}<1$ and $\frac{q\left(\delta_q+1\right)}{q^*}<\delta_q + 1$, we have
	$$
	\int_{\mathbb{R}^N}|\nabla u|^2 d x=\int_{\mathbb{R}^N}|\nabla u|^q d x=0
	$$
which is a contradiction to $u\neq0$.
	Thus, $P^{\prime}(u) \neq 0$ for all $u \in \mathcal{P}$ and the proof is complete.
\end{proof}

In the following lemma, we show that \( \mathcal{P}_0 \cap \mathcal{D}(c) = \emptyset \). This result will be used to establish that \( \mathcal{P} \cap \mathcal{S}(c) \) is a smooth manifold of codimension 2 in \( X \), which allows us to apply the Lagrange multipliers rule. To prove that \( \mathcal{P}_0 \cap \mathcal{D}(c) = \emptyset \), we will need to make the following assumptions regarding \( H \).

\begin{enumerate}[label=\textup{(H\arabic*$^{\prime}$)}, ref=\textup{H\arabic*$^{\prime}$}]
	\setcounter{enumi}{-1}
	\item \label{H0}  $H_1, H_2 \in \mathcal{C}^1(\mathbb{R} ; \mathbb{R})$ and there exist $2 < a < q_{\#}$, $q^{\#} < b < q^{\prime}$ such that
	$$
	H_1(t) \lesssim|t|^{2}+|t|^{a}, \quad H_2(t) \lesssim|t|^{b}+|t|^{q^{\prime}} \quad \text{for all} ,\,t \in \mathbb{R},
	$$
	where ${q}^{\#}:=\left(1+\frac{2}{N}\right) \max \{2, q\}$, ${q}_{\#}:=\left(1+\frac{2}{N}\right) \min \{2, q\}$ and
	${q}^{\prime}:= \max \{2^*, q^*\}$.\\
	\item \label{H1}
	There holds
	$$
	2 H_1(t) \leq h_1(t) t \leq a H_1(t), \quad b H_2(t) \leq h_2(t) t \leq q^{\prime} H_2(t) \quad
	\text{for all} \,\, t \in \mathbb{R},
	$$
	where $h_j:=H_j^{\prime}$ for $j \in\{1,2\}$.
\end{enumerate}
\begin{enumerate}[label=(H\arabic*), ref=\textup{H\arabic*}]
	\setcounter{enumi}{+2}
	\item \label{H3}  $\lim _{|t| \rightarrow\infty} \frac{H(t)}{|t|^{q^{\prime}}}=0$.
\end{enumerate}
\eqref{H0} and \eqref{H1} are less restrictive than \eqref{H-0} and \eqref{H-1}.
\begin{lemma}\label{Lemma-P0}
	If \eqref{H0}, \eqref{H1}, and \eqref{H3} hold and $c>0$ is sufficiently small, then $\mathcal{P}_0 \cap \mathcal{D}(c)=\emptyset$.
\end{lemma}
\begin{proof}
	Suppose by contradiction that there exist $u\in\mathcal{P}_0 \cap \mathcal{D}(c)$,
	by \eqref{Eq-J2der1},
	\begin{equation*}
		\left(\delta_q+1\right)\left(q\left(\delta_q+1\right)-2\right)\int_{\mathbb{R}^N}|\nabla u|^q d x=\frac{N^2}{4}\left(\int_{\mathbb{R}^N} h(u) u d x-\bar{2}\int_{\mathbb{R}^N} H(u) d x\right).
	\end{equation*}
	Denote $A_{N,q}:=\frac{4\left(\delta_q+1\right)\left(q\left(\delta_q+1\right)-2\right)}{N^2}$, we obtain
	\begin{equation}\label{Eq-Hq=}
		A_{N,q}\|\nabla u\|_q^q =\left(\int_{\mathbb{R}^N} h(u) u d x-\bar{2}\int_{\mathbb{R}^N} H(u) d x\right).
	\end{equation}
	It's easy to check that $A_{N,q}>0$ if $2<q<N$, $A_{N,q}<0$ if $\frac{2 N}{N+2}<q<2$.
	
	Simliarly, by \eqref{Eq-J2der2}, we have
	\begin{equation*}
		\left(2-q\left(\delta_q+1\right)\right)\int_{\mathbb{R}^N}|\nabla u|^2 d x=\frac{N^2}{4}\left(\int_{\mathbb{R}^N} h(u) u d x-\bar{q}\int_{\mathbb{R}^N} H(u) d x\right).
	\end{equation*}
	Denote $B_{N,q}=\frac{4(2-q(\delta_q+1))}{N^2}$, we get
	\begin{equation}\label{Eq-H3=}
		B_{N,q}\|\nabla u\|_2^2 =\left(\int_{\mathbb{R}^N} h(u) u d x-\bar{q}\int_{\mathbb{R}^N} H(u) d x\right)	.
	\end{equation}
	 It's easy to check that $B_{N,q}<0$ if $2<q<N$, $B_{N,q}>0$ if $\frac{2 N}{N+2}<q<2$.
	
	Moreover, from \eqref{H0}, \eqref{H1}, and \eqref{H3}, for every $\varepsilon>0$, there exists $C_{\varepsilon}>0$ such that
	\begin{equation}\label{Eq-H2subcritical}
		H_2(t) \leq \varepsilon|t|^{q^{\prime}}+C_{\varepsilon}|t|^b \ \text{for any} \ t \in \mathbb{R}.
	\end{equation}
	Now, we will consider the two cases separately: $2<q<N$ and $\frac{2 N}{N+2}<q<2$.\\
	\textbf{Case 1 :} $2<q<N$.
	Note that from \eqref{Eq-Hq=} and \eqref{H1}, one has
	\begin{equation}\label{Eq-H31}
		\begin{aligned}
			(\bar{2}-2) \int_{\mathbb{R}^N} H_1(u) \mathrm{d} x & \geq \int_{\mathbb{R}^N} \bar{2} H_1(u)-h_1(u) u \mathrm{~d} x \\
			& =\int_{\mathbb{R}^N} h_2(u) u-\bar{2} H_2(u) \mathrm{d} x-A_{N,q}\|\nabla u\|^q_q \\
			&\geq\left(b-\bar{2}\right) \int_{\mathbb{R}^N} H_2(u) \mathrm{d} x-A_{N,q}\|\nabla u\|^q_q  .
		\end{aligned}
	\end{equation}
	Hence, using that $u\in \mathcal{P}$, from \eqref{Eq-poH3} and \eqref{Eq-H31}, we get
	\begin{equation}\label{Eq-Hgeq11}
		\begin{aligned}
			\|\nabla u\|_2^2  + \left(\delta_q + 1\right) \|\nabla u\|_q^q & = \frac{N}{2} \int_{\mathbb{R}^N} H(u)  dx\\
			&=\frac{N}{2} \int_{\mathbb{R}^N} H_1(u) dx+
			\frac{N}{2} \int_{\mathbb{R}^N} H_2(u) dx\\
			&\leq \frac{N}{2} \frac{b-2}{b-\bar{2}} \int_{\mathbb{R}^N} H_1(u) \mathrm{d} x
			+\frac{NA_{N,q}}{2(b-\bar{2})}\|\nabla u\|_q^q.
		\end{aligned}
	\end{equation}
	Let
	$$
	\bar{C}_1:=(\delta_q+1)-\frac{NA_{N,q}}{2(b-\bar{2})}.
	$$
Since \( 2 < q < N \), it can be proved through calculations that \( \bar{C}_1 > 0 \).
	From \eqref{H0}, \eqref{H1} and Lemma \ref{Eq-GN inequality},
	 we obtain
	\begin{equation}\label{Eq-Hgeqver11}
		\begin{aligned}
			\|\nabla u\|^2_2  + \bar{C}_1\|\nabla u\|^q_q
			&\leq \frac{N}{2} \frac{b-2}{b-\bar{2}} \int_{\mathbb{R}^N} H_1(u) \mathrm{d} x\\
			&\leq
			C \frac{N}{2} \frac{b-2}{b-\bar{2}}\left(c^2+\|u\|_a^a\right)\\
			&\leq C \frac{N}{2} \frac{b-2}{b-\bar{2}}\left(c^2+C_{N, a}^ac^{(1-\delta_a)a}
			\|\nabla u\|_2^{\delta_a a}\right)\\
			&=A_c+B_c\|\nabla u\|_2^{\delta_a a},
		\end{aligned}
	\end{equation}
	where
	$$
	A_c:=C \frac{N}{2} \frac{b-2}{b-2} c^2 \text { and } B_c:=C \frac{N}{2} \frac{b-2}{b-2} c^{\left(1-\delta_a\right) a}.
	$$
	Since $\bar{C}_1>0$, we have
	\begin{equation}
		\|\nabla u\|_2^2\leq A_c+B_c\|\nabla u\|_2^{\delta_a a}.
	\end{equation}
	Observe that $A_c\rightarrow0$, $B_c\rightarrow0$ as $c \rightarrow 0^{+}$ and $\delta_{a}a<2$. Therefore
	$$
	\frac{\sqrt{A_c}+\frac{1}{4}}{\left(\sqrt{A_c}+\frac{1}{2}\right)^{\delta_aa}} \rightarrow \frac{1}{2^{2-\delta_aa}}>0 \quad \text { as } c \rightarrow 0^{+} .
	$$
	Thus, for sufficiently small $c>0$, we have that
	$$
	B_c\leq \frac{\sqrt{A_c}+\frac{1}{4}}{\left(\sqrt{A_c}+\frac{1}{2}\right)^{\delta_aa}}.
	$$
	So  Lemma \ref{Lem-A2}  implies that
	\begin{equation}\label{est-u2case1}
		\|\nabla u\|_2 \leq \sqrt{A_c}+\frac{1}{2}.
	\end{equation}
	On the other hand, from \eqref{Eq-Hq=}, using \eqref{H0}, \eqref{H1} and Lemma \ref{Eq-GN inequality2}, we have
	\begin{equation}
		\begin{aligned}
			A_{N,q}\|\nabla u\|_q^q &
			=\int_{\mathbb{R}^N}h_1(u) u-\bar{2} H_1(u) d x+
			\int_{\mathbb{R}^N}h_2(u) u-\bar{2} H_2(u)d x \\
			& \leq \int_{\mathbb{R}^N}(a-\bar{2}) H_1(u) d x+\int_{\mathbb{R}^N}\left(q^*-\bar{2}\right) H_2(u) d x \\
			& \leq\left(q^*-\bar{2}\right) \int_{\mathbb{R}^N} H_2(u) d x \\
			& \leq\left(q^*-\bar{2}\right) \left(C_{\varepsilon}\|u\|_b^b+\varepsilon\|u\|_{q^*}^{q^*}\right) \\
			& \leq\left(q^*-\bar{2}\right) \left(C_{\varepsilon}K_{N, p}^p\|\nabla u\|_q^{\nu_{b, q} b}\|u\|_2^{(1-\nu_{b, q}) b}+\varepsilon\mathcal{S}_q^{-q^* / q}\|\nabla u\|_q^{q^*}\right)\\
			& \leq\left(q^*-\bar{2}\right) \left(C_{\varepsilon}K_{N, p}^pc^{(1-\nu_{b, q}) b}\|\nabla u\|_q^{\nu_{b, q} b}+\varepsilon\mathcal{S}_q^{-q^* / q}\|\nabla u\|_q^{q^*}\right).
		\end{aligned}
	\end{equation}
	Let
	$$
	D_c:=\left(q^*-\bar{2}\right)A_{N,q}^{-1}K_{N, p}^pc^{(1-\nu_{b, q}) b}>0, \quad
	E:=\left(q^*-\bar{2}\right)A_{N,q}^{-1}\mathcal{S}_q^{-q^* / q}>0,
	$$
	 we have
	\begin{equation}
		\|\nabla u\|_q^q\leq C_{\varepsilon} D_c \|\nabla u\|_q^{\nu_{b, q} b}+\varepsilon\|\nabla u\|_q^{q^*}.
	\end{equation}
	Since $\bar{q}<b<q^{*}$, we have $q<\nu_{b, q}b<q^* $. Lemma \ref{Lem-A1} implies that
	\begin{equation}\label{Est-uqcase1}
		\|\nabla u\|_q \geq \min \left\{1,\left(C_{\varepsilon}D_c+\varepsilon E\right)^{\frac{1}{q-\nu_{b, q} b}}\right\}.
	\end{equation}
	Choose
	$$
	\varepsilon \leq \frac{1}{2E}
	$$
	and let $c>0$ be  small enough such that
	$$
	D_c\leq \frac{1}{2C_{\varepsilon}}.
	$$
	Since  $q-\nu_{b, q}b<0$,	
	$$
	\left(C_{\varepsilon}D_c+\varepsilon E\right)^{\frac{1}{q-\nu_{b, q}b}}\geq 1.
	$$
	So, combining \eqref{Eq-Hgeqver11} and \eqref{est-u2case1}, we have
	\begin{equation}
		\begin{aligned}
			\bar{C}_1 & \leq\bar{C}_1\|\nabla u\|_q^q \\
			& \leq\|\nabla u\|_2^2+\bar{C}_1\|\nabla u\|_q^q \\
			& \leq A_c+B_c\|\nabla u\|_2^{\delta_a a} \\
			& \leq A_c+B_c\left(\sqrt{A_c}+\frac{1}{2}\right)^{\delta_a a}.
		\end{aligned}
	\end{equation}	
	Since $A_c\rightarrow0$, $B_c\rightarrow0$ as $c \rightarrow 0^{+}$ and $ 	\bar{C}_1 $ is independent of $c$,
	 we get a contradiction by choosing $c>0$ small enough.\\
	\textbf{Case 2 :} $\frac{2 N}{N+2}<q<2$.
	Note that from \eqref{Eq-H3=} and \eqref{H1}
	\begin{equation}\label{Eq-Hq1w}
		\begin{aligned}
			(\bar{q}-2) \int_{\mathbb{R}^N} H_1(u){d} x
			& \geq \int_{\mathbb{R}^N} \bar{q} H_1(u)-h_1(u) u {~d} x \\
			& =\int_{\mathbb{R}^N} h_2(u) u-\bar{q} H_2(u) {d} x-B_{N,q}
			\|\nabla u\|_2^2\\
			&\geq\left(b-\bar{q}\right) \int_{\mathbb{R}^N} H_2(u) {d} x-B_{N,q}\|\nabla u\|_2^2  .
		\end{aligned}
	\end{equation}
	Hence, using that $u\in \mathcal{P}$, from \eqref{Eq-poH3} and \eqref{Eq-Hq1w}, we get
	\begin{equation}\label{Eq-Hgeq1w}
		\begin{aligned}
			\|\nabla u\|_2^2  + \left(\delta_q + 1\right) \|\nabla u\|_q^q & = \frac{N}{2} \int_{\mathbb{R}^N} H(u) \, dx\\
			&\leq \frac{N}{2} \frac{b-2}{b-\bar{q}} \int_{\mathbb{R}^N} H_1(u) \mathrm{d} x
			+\frac{NB_{N,q}}{2(b-\bar{q})}\|\nabla u\|_2^2.
		\end{aligned}
	\end{equation}
	Let
	 $$
	 \bar{C}_2:=1-\frac{NB_{N,q}}{2(b-\bar{q})}.
	 $$
	Since  $\frac{2 N}{N+2}<q<2$ , it can be proved through calculations that \( \bar{C}_2 > 0 \). From \eqref{H0}, \eqref{H1} and Lemma \ref{Eq-GN inequality},
	\begin{equation}\label{Eq-Hgeqver111}
		\begin{aligned}
			\bar{C}_2\|\nabla u\|^2_2  + \left(\delta_q + 1\right)\|\nabla u\|^q_q &\leq \frac{N}{2} \frac{b-2}{b-\bar{q}} \int_{\mathbb{R}^N} H_1(u) \mathrm{d} x\\
			&\leq C \frac{N}{2} \frac{b-2}{b-\bar{q}}\left(c^2+\|u\|_a^a\right)\\
			&\leq C \frac{N}{2} \frac{b-2}{b-\bar{q}}\left(c^2+C_{N, a}^ac^{(1-\delta_a)a}
			\|\nabla u\|_2^{\delta_a a}\right)\\
			&=A_c^{\prime}+B_c^{\prime}\|\nabla u\|_2^{\delta_a a},
		\end{aligned}
	\end{equation}
	where
	$$
	A_c^{\prime}:=\frac{N}{2} \frac{b-2}{b-\bar{q}} c^2 \text { and } B_c^{\prime}:=\frac{N}{2} \frac{b-2}{b-\bar{q}} c^{\left(1-\delta_a\right) a}.
	$$
	Since $\bar{C}_2>0$, we have
	\begin{equation}
		\|\nabla u\|_2^2\leq A_c^{\prime}+B_c^{\prime}\|\nabla u\|_2^{\delta_a a}.
	\end{equation}
	Observe that $A_c\rightarrow0$, $B_c\rightarrow0$ as $c \rightarrow 0^{+}$ and $\delta_{a}a<2$. Therefore
	$$
	\frac{\sqrt{A_c^{\prime}}+\frac{1}{4}}{\left(\sqrt{A_c^{\prime}}+\frac{1}{2}\right)^{\delta_aa}} \rightarrow \frac{1}{2^{2-\delta_aa}}>0 \quad \text { as } c \rightarrow 0^{+} .
	$$
	Thus, for $c>0$  sufficiently small we have that
	$$
	B_c^{\prime}\leq \frac{\sqrt{A_c^{{\prime}}}+\frac{1}{4}}{\left(\sqrt{A_c^{{\prime}}}+\frac{1}{2}\right)^{\delta_aa}}.
	$$
	So \cite[Lemma A.1.]{bieganowski2024existence} implies that
	\begin{equation}\label{est-u2case2}
		\|\nabla u\|_2 \leq \sqrt{A^{\prime}_c}+\frac{1}{2}.
	\end{equation}
	On the other hand, from \eqref{Eq-H3=}, using \eqref{H0}, \eqref{H1} and Lemma \ref{Eq-GN inequality2}, we have
	\begin{equation}
		\begin{aligned}
			B_{N,q}\|\nabla u\|_2^2 &
			=\int_{\mathbb{R}^N}h_1(u) u-\bar{q} H_1(u) d x+
			\int_{\mathbb{R}^N}h_2(u) u-\bar{q} H_2(u) d x \\
			& \leq \int_{\mathbb{R}^N}(a-\bar{q}) H_2(u) d x+\int_{\mathbb{R}^N}\left(b-\bar{q}\right) H_2(u) d x \\
			& \leq\left(b-\bar{q}\right) \int_{\mathbb{R}^N} H_1(u) d x \\
			& \leq\left(b-\bar{q}\right) \left(C_{\varepsilon}\|u\|_b^b+{\varepsilon}\|u\|_{2^*}^{2^*}\right) \\
			& \leq\left(b-\bar{q}\right) \left(C_{\varepsilon}C_{N, b}^b\|\nabla u\|_2^{\delta_bb }\|u\|_2^{(1-\delta_{b})b}+\varepsilon\mathcal{S}_2^{-2^* / 2}\|\nabla u\|_{2}^{2^*}
			\right)\\
			& \leq\left(b-\bar{q}\right) \left(C_{\varepsilon}C_{N, b}^bc^{(1-\delta_{b})b}\|\nabla u\|_2^{\delta_bb }+\varepsilon\mathcal{S}_2^{-2^* / 2}\|\nabla u\|_{2}^{2^*}
			\right).
		\end{aligned}
	\end{equation}
	Let
	$$
	D_c^{\prime}:=\left(b-\bar{q}\right)B_{N,q}^{-1}C_{N, b}^bc^{(1-\delta_b) b}>0, \quad
	E^{\prime}:=\left(b-\bar{q}\right)B_{N,q}^{-1}\mathcal{S}_2^{-2^* / 2}>0,
	$$
	we get
	\begin{equation}\label{Eq-Hgeqver12}
		\|\nabla u\|_q^q\leq D_c^{\prime} \|\nabla u\|_2^{\delta_bb}+E^{\prime} \|\nabla u\|_2^{2^*}.
	\end{equation}
	Since $\bar{2}<b<2^*$, we have $2<\delta_bb<2^* $.
	Lemma \ref{Lem-A1} implies that
	\begin{equation}\label{Est-uqcase2}
		\|\nabla u\|_2 \geq \min \left\{1,\left(C^{\prime}_{\varepsilon}D_c^{\prime} +\varepsilon E\right)^{\frac{1}{q-\delta_b b}}\right\}.
	\end{equation}
	Choose
	$$
	\varepsilon \leq \frac{1}{2E^{\prime} },
	$$
	and let $c$ be so small that
	$$
	D_c^{\prime} \leq \frac{1}{2C_{\varepsilon}}.
	$$
	Since $q-\delta_bb<0$, we have
	$$
	\left(C_{\varepsilon}D_c+\varepsilon E\right)^{\frac{1}{q-\delta_bb}}\geq1.
	$$
		So, combining \eqref{Eq-Hgeqver111} and \eqref{est-u2case2}, we have
	\begin{equation}
		\begin{aligned}
			\bar{C}_2 & \leq\bar{C}_2\|\nabla u\|_2^2 \\
			& \leq\bar{C}_2\|\nabla u\|_2^2+(\delta_q+1)\|\nabla u\|_q^q \\
			& \leq A_c^{\prime}+B_c^{\prime}\|\nabla u\|_2^{\delta_a a} \\
			& \leq A_c^{\prime}+B_c^{\prime}\left(\sqrt{A_c^{\prime}}+\frac{1}{2}\right)^{\delta_a a}.
		\end{aligned}
	\end{equation}	
Since $A_c^{\prime}\rightarrow0$, $B_c^{\prime}\rightarrow0$ as $c \rightarrow 0^{+}$ and and $\bar{C}_2 $ is independent of $c$,
	 we get a contradiction by choosing sufficiently small $c>0$. Based on the above two cases, we can conclude that $\mathcal{P}_0 \cap \mathcal{D}(c)=\emptyset$.
\end{proof}
\begin{lemma}\label{Lemma-codim2}
	If \eqref{H0}, \eqref{H1}, and \eqref{H3} hold and $c>0$ is sufficiently small, then $\mathcal{P}_c=\mathcal{P}\cap\mathcal{S}(c)$ is a smooth manifold of codimension 2 in $X$.
\end{lemma}
\begin{proof}
	We note that $\mathcal{P}_c=\{u \in$ $X: P(u)=0, G(u)=0\}$, where $G(u)=\|u\|_2^2-c^2$, with $P$ and $G$ being of class $C^1$ in $X$. Thus, it suffices to show that the differential $(G^{\prime}(u), P^{\prime}(u)): X \rightarrow \mathbb{R}^2$ is surjective,
	for every $u \in \mathcal{P}$. 	For $\varphi,\psi\in X$, we give the following system
	$$
	\left\{\begin{array}{l}
		G^{\prime}(u)(\alpha \varphi+\beta \psi)=x \\
		P^{\prime}(u)(\alpha \varphi+\beta \psi)=y.
	\end{array}\right.
	$$
	If this system is solvable with respect to $\alpha, \beta$, for every $(x, y) \in \mathbb{R}^2$,
	then $(G^{\prime}(u), P^{\prime}(u)): X \rightarrow \mathbb{R}^2$ is surjective.

	For this purpose, let $\psi:=u$, we prove that for every $u \in \mathcal{P}_c$ there exists $\varphi \in T_u \mathcal{S}(c)$ such that $P^{\prime}(u)\varphi \neq 0$. Once that the existence of $\varphi$ is established, the system
	
	$$
	\left\{\begin{array} { l }
		{ G^{\prime} ( u ) ( \alpha \varphi + \beta u ) = x } \\
		{  P^{\prime} ( u ) ( \alpha \varphi + \beta u ) = y }
	\end{array} \quad \Longleftrightarrow \quad \left\{\begin{array}{l}
		2\beta c^2=x \\
		\alpha P^{\prime}(u)\varphi+\beta P^{\prime}(u)u=y
	\end{array}\right.\right.
	$$
	is solvable with respect to $\alpha, \beta$, for every $(x, y) \in \mathbb{R}^2$, and hence the surjectivity is proved.

	Now, suppose by contradiction that for $u \in \mathcal{P}_c$ such a tangent vector $\varphi$ does not exist, i.e.
	 $P^{\prime}(u)\varphi=0$ for every $\varphi \in T_u \mathcal{S}(c)$. Then $u$ is a constrained critical point for the functional $P$ on $\mathcal{S}(c)$, and hence by the Lagrange multipliers rule, there exists $\lambda \in \mathbb{R}$ such that
	$$
	-2 \Delta \widetilde{u}-q\left(\delta_q+1\right) \Delta_q \widetilde{u}+\lambda u=\frac{N}{2} h(\widetilde{u})\widetilde{u}.
	$$
	But, by the Pohozaev identity, this implies that
	$$
	2 \|\nabla u\|^2_2+q\left(\delta_q+1\right)^2 \|\nabla u\|_q^q=\frac{N^2}{4} \int_{\mathbb{R}^N} h(u) u-2H(u) d x.
	$$
	That is $u \in \mathcal{P}_0$, from Lemma \ref{Lemma-P0}, we get a contradiction.
\end{proof}
In the following two lemmas, we show that $J$ is bounded away from $0$ on ${\mathcal{P}_{-}^{} \cap \mathcal{D}(c)}$ and coercive on ${\mathcal{P}_{-}^{\mathrm{rad}} \cap \mathcal{D}(c)}$.
\begin{lemma}\label{Lemma-J>0onP}
	If \eqref{F0}, \eqref{F2}, \eqref{F4}, \eqref{J0} and \eqref{Eq-crange} hold, then $\inf _{\mathcal{P} \_\cap \mathcal{D}(c)} J>0$.
\end{lemma}
\begin{proof}
	Let $u \in \mathcal{P}_{-} \cap \mathcal{D}(c)$. Recall $\widetilde{q}=\max\{2,q\}$,
	we observe that
		 \begin{equation*}
		\begin{aligned}
			J(s * u) &= \frac{e^{2s}}{2} \int_{\mathbb{R}^N} |\nabla u|^2 \, dx + \frac{e^{q(\delta_q+1)s}}{q} \int_{\mathbb{R}^N} |\nabla u|^q \, dx -e^{-Ns}\int_{\mathbb{R}^N}F(e^{\frac{N}{2} s} u)  dx\\
			&= e^{\widetilde{q}(\delta_{\widetilde{q}}+1)s}\left( \frac{e^{(2-\widetilde{q}(\delta_{\widetilde{q}}+1))s}}{2}\int_{\mathbb{R}^N} |\nabla u|^2 \, dx + \frac{e^{(q(\delta_q+1)-\widetilde{q}(\delta_{\widetilde{q}}+1))s}}{q}  \int_{\mathbb{R}^N} |\nabla u|^q \, dx -\frac{1}{\left(e^{ \frac{N}{2}s}\right)^{{q}^{\#}}}\int_{\mathbb{R}^N}F(e^{\frac{N}{2} s} u)  dx\right).
		\end{aligned}
	\end{equation*}	
Note that, from \eqref{F4}, $F\left(e^{\frac{N}{2}s} u\right)>0$ a.e. in $\operatorname{supp} u$ for sufficiently large $s>0$. Therefore, from Fatou's lemma and \eqref{F4} again,
$$
\lim _{s \rightarrow \infty }\frac{1}{\left(e^{ \frac{N}{2}s}\right)^{{q}^{\#}}}\int_{\mathbb{R}^N}F(e^{\frac{N}{2} s} u)  dx=\infty.
$$
Since $\widetilde{q}=\max\{2,q\}$, $\lim _{s \rightarrow \infty }{e^{(2-\widetilde{q}(\delta_{\widetilde{q}}+1))s}}=0$ and $\lim _{s \rightarrow \infty }e^{(q(\delta_q+1)-\widehat{q}(\delta_{\widehat{q}}+1))s}=0$. This implies  $\lim _{s \rightarrow\infty} J(s * u)=-\infty$. And from Lemma \ref{Lemma-mR(a)}, we can see that
$\lim _{s \rightarrow-\infty} J(s * u)<0$.

	Now, let $t_{\max }>0$ be the unique maximum point of $t \mapsto g(c, t) t^{\widetilde{q}}$, guaranteed by \eqref{Eq-crange}, where $g$ is defined by \eqref{Eq-g(a,t)}, see \eqref{g1} in Lemma \ref{Lemma-g}, and choose $s_u$ such that $\left\|\nabla\left(s_u * u\right)\right\|_{\widetilde{q}}=t_{\max }$. Then, by \eqref{Eq-J>g},
	$$
	J\left(s_u * u\right) \geq g\left(c,\left\|\nabla\left(s_u * u\right)\right\|_{\widetilde{q}}\right)\left\|\nabla\left(s_u * u\right)\right\|_{\widetilde{q}}^{\widetilde{q}}=g\left(c, t_{\max }\right) t_{\max }^{\widetilde{q}}>0,
	$$
	Due to \eqref{Eq-Pand2der}, $0$ is a local maximum point of $(-\infty,\infty) \ni s \mapsto J(s * u) \in \mathbb{R}$. From \eqref{J0}, 0 is the unique local maximum point of $J(s * u)$.
	Since $\lim _{s \rightarrow\infty} J(s * u)<0$ and $\lim _{s \rightarrow-\infty} J(s * u)=-\infty$,
	$0$ is actually the global maximum point of
	$J(s * u)$ for $s\in\mathbb{R}$. So
	$$
	J(u)=J\left(0 * u\right)\geq J\left(s_u * u\right) >0,
	$$
	which completes the proof.
\end{proof}
To prove that \( J \) is coercive on \( \mathcal{P}_{-}^{\mathrm{rad}} \cap \mathcal{D}(c) \), recall the following condition.
\begin{enumerate}[label=\textup{(F\arabic*)}, ref=\textup{F\arabic*}]
	\setcounter{enumi}{+2}
	\item \label{F--2}
	$
	\lim _{|t| \rightarrow 0} \dfrac{F(t)}{|t|^{q^{\#}}}=+\infty.
	$
\end{enumerate}

\begin{lemma}\label{Lemma-coercive}
	If \eqref{F0}, \eqref{F1}, \eqref{F3}, \eqref{F4}, \eqref{F5}, \eqref{J0} and \eqref{Eq-crange} hold, then $\left.J\right|_{\mathcal{P}_{-}^{\mathrm{rad}} \cap\mathcal{D}(c)}$ is coercive.
\end{lemma}
\begin{proof}
	Suppose that $\left(u_n\right)_n \subset {\mathcal{P}_{-}^{\mathrm{rad}}\cap \mathcal{D}(c)}$ is a sequence in $X$ such that $\left\|u_n\right\|_X\rightarrow \infty$ as $n \rightarrow\infty$. For convenience, we define
	
		$$
		\gamma:=\varlimsup _{n \rightarrow\infty} \frac{\|\nabla u_n\|_2}{\|\nabla u_n\|^{\frac{1}{\delta_q +1}}_q}.
		$$
		We will consider $ \gamma \leq 1 $ and 	$\gamma> 1 $ separately.\\
	\textbf{Case 1 :} $	\gamma \leq 1.$\\
	$\gamma \leq1$ implies that
	\begin{equation}\label{Eq-Compare}
		\varlimsup _{n \rightarrow\infty} \frac{\|\nabla u_n\|_2}{\|\nabla u_n\|^{\frac{1}{\delta_q +1}}_q} \leq 1.
	\end{equation}
 	Since $\left\|u_n\right\|_X\rightarrow \infty$ as $n \rightarrow\infty$,
	\eqref{Eq-Compare} implies that $\|\nabla u_n\|^{\frac{1}{\delta_q +1}}_q\rightarrow \infty$ as $n \rightarrow\infty$.
	If not, both $\|\nabla u_n\|_q$ and
	$\|\nabla u_n\|_2$ would remain bounded in $ \mathbb{R}^+ $, contradicting the fact that $\left\|u_n\right\|_X\rightarrow \infty$ as $n \rightarrow\infty$.
		We set $s_n:=\ln(\|\nabla u\|^{-\frac{1}{\delta_q +1}}_q)$, and note that $s_n \rightarrow -\infty$ as $n \rightarrow\infty$. Define $v_n:=s_n * u_n$.
	Then
	\begin{equation*}
		\left\|v_n\right\|_2^2=\left\|u_n\right\|_2^2 \leq c^2 .
	\end{equation*}
	Furthermore, for $n$ large enough
	\begin{equation*}
		\left\|\nabla v_n\right\|_2^2=e^{2s_n}\left\|\nabla u_n\right\|_2^2=\frac{\left\|\nabla u_n\right\|_2^2}{\left\|\nabla u_n\right\|_q^{\frac{2}{\delta_q+1}}} \leq 2
	\end{equation*}
	and
	\begin{equation*}
		\left\|\nabla v_n\right\|_q^q=e^{q\left(\delta_q+1\right)s_n}\left\|\nabla u_n\right\|_q^q=1 .
	\end{equation*}
	Hence, $\left(v_n\right)_n$ is bounded in $X_{\mathrm{rad}}$. Then, there exists $v \in X_{\mathrm{rad}}$ such that, up to a subsequence, $v_n \rightharpoonup v$ in $X_{\mathrm{rad}}$ and $v_n \rightarrow v$ in $L^{p} \left(\mathbb{R}^N\right)$ for $p\in(2,q^{\prime})$ and a.e. in $\mathbb{R}^N$. Let $F_{ \pm}:=\max \{ \pm F, 0\}$. From \eqref{F1} and \eqref{F4}, we obtain that $F_{-}(t) \lesssim t^2$. Suppose that $v \neq 0$. Then, from Lemma \ref{Lemma-J>0onP}, \eqref{F4} and Fatou's lemma we deduce that
	\begin{equation*}
		\begin{aligned}
			0 \leq \frac{J\left(u_n\right)}{\left\|\nabla u_n\right\|_q^q}
			& =\frac{1}{2}
			\frac{\left\|\nabla u_n\right\|_2^2}{\left\|\nabla u_n\right\|_q^q}+
			\frac{1}{q}-\int_{\mathbb{R}^N} \frac{F\left(u_n\right)}{\left\|\nabla u_n\right\|_q^q} d x \\
			& =\frac{e^{(q\left(\delta_q+1\right)-2)s_n}}{{2}} \frac{\left\|\nabla v_n\right\|_2^2}{\left\|\nabla v_n\right\|_q^q}+\frac{1}{q}-e^{q\left(\delta_q+1\right)s_n} \cdot e^{Ns_n} \int_{\mathbb{R}^N} F\left(u_n\left(e^{s_n}x\right)\right) d x \\
			& \leq {e^{(q\left(\delta_q+1\right)-2)s_n}}+\frac{1}{q}-
			e^{(N+q\left(\delta_q+1\right))s_n} \int_{\mathbb{R}^N} F\left(e^{-\frac{N}{2}s_n} v_n\right) d x\\
			& = {e^{(q\left(\delta_q+1\right)-2)s_n}}+\frac{1}{q}+
			e^{(N+q\left(\delta_q+1\right))s_n} \int_{\mathbb{R}^N} F_{-}\left(e^{-\frac{N}{2}s_n} v_n\right) -F_{+}\left(e^{-\frac{N}{2}s_n} v_n\right)d x\\
			& \leq {e^{(q\left(\delta_q+1\right)-2)s_n}}+\frac{1}{q}
			+Ce^{q\left(\delta_q+1\right)s_n}c^2
			-
			e^{(N+q\left(\delta_q+1\right))s_n} \int_{\mathbb{R}^N} F_{+}\left(e^{-\frac{N}{2}s_n} v_n\right)d x\\
			& \leq {e^{(q\left(\delta_q+1\right)-2)s_n}}+\frac{1}{q}
			+Ce^{q\left(\delta_q+1\right)s_n}c^2
			-
			e^{(N+q\left(\delta_q+1\right))s_n}
			  \int_{\mathbb{R}^N} \frac{F_{+}\left(e^{-\frac{N}{2}s_n} v_n\right)\left(e^{-\frac{N}{2}s_n}\right)^{{q}^{\#}}}{\left(e^{-\frac{N}{2}s_n} \right)^{{q}^{\#}}}d x\\
			& \leq {e^{(q\left(\delta_q+1\right)-2)s_n}}+\frac{1}{q}
			+Ce^{q\left(\delta_q+1\right)s_n}c^2
			-
				e^{\left(q\left(\delta_q+1\right)-\widetilde{q}\left(\delta_{\widetilde{q}}+1\right)\right)s_n}
				\int_{\mathbb{R}^N} \frac{F_{+}\left(e^{-\frac{N}{2}s_n} v_n\right)}{\left(e^{-\frac{N}{2}s_n} \right)^{{q}^{\#}}}d x\\
			&\leq e^{\left(q\left(\delta_q+1\right)-\widetilde{q}\left(\delta_{\widetilde{q}}+1\right)\right)s_n}\left({e^{\left(\widetilde{q}\left(\delta_{\widetilde{q}}+1\right)-2\right)s_n}}
			-
			\int_{\mathbb{R}^N} \frac{F_{+}\left(e^{-\frac{N}{2}s_n} v_n\right)}{\left(e^{-\frac{N}{2}s_n}\right)^{{q}^{\#}}}d x\right)
				+\frac{1}{q}
			+Ce^{q\left(\delta_q+1\right)s_n}c^2
		\end{aligned}
	\end{equation*}
	 Note that, $F_{+}\left(e^{-\frac{N}{2}s_n} v_n\right)>0$ a.e. in $\operatorname{supp} v_n$. Therefore, from Fatou's lemma and \eqref{F3} again,
	$$
	\lim _{n \rightarrow \infty }\int_{\mathbb{R}^N}\frac{F_{+}\left(e^{-\frac{N}{2} s_n} v_n\right) }{\left(e^{ -\frac{N}{2}s_n}\right)^{{q}^{\#}}} dx=\infty.
	$$
	Since $\widetilde{q}=\max\{2,q\}$, $\lim _{n \rightarrow \infty }e^{\left(q\left(\delta_q+1\right)-\widetilde{q}\left(\delta_{\widetilde{q}}+1\right)\right)s_n}=\infty$ and $\lim _{n \rightarrow \infty }{e^{\left(\widetilde{q}\left(\delta_{\widetilde{q}}+1\right)-2\right)s_n}}=0$. This implies that  $ \frac{J\left(u_n\right)}{\left\|\nabla u_n\right\|_q^q}\rightarrow-\infty $ as $n\rightarrow0$, which is a contradiction.
	Hence, $v=0$ and $v_n \rightarrow 0$ in $L^{p}\left(\mathbb{R}^N\right)$ for $p\in(2,q^{\prime})$ as $n \rightarrow \infty$. Note that $u_n=s_n^{-1} * v_n \in \mathcal{P}_{-}^{\text {rad }} \cap \mathcal{D}(c)$. Thus, from \eqref{J0}, \eqref{F3} and \eqref{F4}, arguing as in Lemma \ref{Lemma-J>0onP}, $ s_n^{-1} $ is the unique maximizer of $s \mapsto J\left(s * v_n\right)$. Consequently, we get
	\begin{equation}
		J\left(u_n\right)=J\left(s_n^{-1} * v_n\right) \geq J\left(s * v_n\right)\geq
		\frac{e^{q\left(\delta_q+1\right)s}}{q}-e^{-Ns} \int_{\mathbb{R}^N} F\left(e^{\frac{N}{2}s} v_n\right) \mathrm{d} x
	\end{equation}
	for any $s>0$. Note that, from \eqref{F0}, \eqref{F1}, and \eqref{F5}, for every $\varepsilon>0$, there is $C_{\varepsilon}>0$ such that for every $n$
	$$
	\int_{\mathbb{R}^N} F\left(e^{\frac{N}{2}s} v_n\right) \mathrm{d} x \leq \varepsilon\left(\left\|e^{\frac{N}{2}s} v_n\right\|_2^2+\left\|e^{\frac{N}{2}s} v_n\right\| _{q^{\prime}}^{q^{\prime}}\right)+C_{\varepsilon}\left\|e^{\frac{N}{2}s} v_n\right\|_{q^{\#}}^{q^{\#}}.
	$$
	Since $v_n \rightarrow 0$ in $L^{p}\left(\mathbb{R}^N\right)$ for $p\in(2,q^{\prime})$ as $n \rightarrow \infty$, therefore
	$$
	\int_{\mathbb{R}^N} F\left(e^{\frac{N}{2}s} v_n\right) \mathrm{d} x \rightarrow 0 \quad \text { as } n \rightarrow\infty \text {. }
	$$
	Hence
	$$
	\liminf _{n\rightarrow \infty} J\left(u_n\right) \geq \liminf _{n\rightarrow \infty}\left(\frac{e^{q\left(\delta_q+1\right)s}}{q}-e^{-Ns} \int_{\mathbb{R}^N} F\left(e^{\frac{N}{2}s} v_n\right) \mathrm{d} x\right)=\frac{e^{q\left(\delta_q+1\right)s}}{q}
	$$
	for every $s>0$. Thus, $\lim _{n\rightarrow \infty} J\left(u_n\right)=\infty$.\\
		\textbf{Case 2 :} $	\gamma> 1.$ \\

			$\gamma >1$ implies that
		\begin{equation}\label{Eq-Compare2}
			\varlimsup _{n \rightarrow\infty} \frac{\|\nabla u_n\|_2}{\|\nabla u_n\|^{\frac{1}{\delta_q +1}}_q} > 1.
		\end{equation}
		Since $\left\|u_n\right\|_X\rightarrow \infty$ as $n \rightarrow\infty$,
		\eqref{Eq-Compare} implies that $\|\nabla u_n\|_2\rightarrow \infty$ as $n \rightarrow\infty$ up to a subsequence.
			If not, both $\|\nabla u_n\|_q$ and
		$\|\nabla u_n\|_2$ would remain bounded in $ \mathbb{R}^+ $, contradicting the fact that $\left\|u_n\right\|_X\rightarrow \infty$ as $n \rightarrow\infty$.
		We set $s_n:=\ln(\|\nabla u\|^{-1}_2)$, and note that $s_n \rightarrow -\infty$ as $n \rightarrow\infty$. Define $v_n:=s_n * u_n$.
		The remaining proof is similar, we omit it here.
	So we conclude that $\left.J\right|_{\mathcal{P}_{-}^{\mathrm{rad}} \cap\mathcal{D}(c)}$ is coercive.
\end{proof}
	
			Using Lemmas \ref{Lemma-J>0onP} and \ref{Lemma-coercive}, we establish that \( \inf_{\mathcal{P}_{-}^{\mathrm{rad}}\cap \mathcal{D}(c) } J > 0 \) and that \( J \) is coercive on \( \mathcal{P}_{-}^{\mathrm{rad}} \cap \mathcal{D}(c) \).
			To show that \( \inf_{\mathcal{P}_{-}^{\mathrm{rad}} \cap \mathcal{D}(c)} J \) is attained, we first demonstrate that \( \inf_{\mathcal{P}_{-} \cap \mathcal{D}(c)} \|\nabla u\|_{\widetilde{q}} > 0 \). This result allows us to give some estimates for \( H_1 \) and \( H_2 \). We will then use these estimates to prove that the weak limit of the minimizing sequence \( (u_n)_n \) for \( J \) remains in \( \mathcal{P}_{-}^{\mathrm{rad}} \cap \mathcal{D}(c) \).
	
\begin{lemma}\label{Lemma-u>0onP}
	If \eqref{H0} and \eqref{H1} hold, then
	$\inf _{\mathcal{P}_{-} \cap \mathcal{D}(c)}\|\nabla u\|_{\widetilde{q}}>0$, in addition, if \eqref{Eq-crange} also holds, then
	$\inf _{\mathcal{P}_{-} \cap \mathcal{D}(c)}\|\nabla u\|_{\widetilde{q}}>M$, where $M$ is a constant independent of $c$.
\end{lemma}
\begin{proof}
	First we may assume that $2<q<N$. For $u \in {\mathcal{P}_{-} \cap \mathcal{D}(c)}$,
	from \eqref{H0}, \eqref{H1} and Lemma \ref{Eq-GN inequality} we have
	\begin{equation}
		\begin{aligned}
			\int_{\mathbb{R}^N} H_2(u) d x
			& \leq C \left(\|u\|_b^b+\|u\|_{q^*}^{q^*}\right) \\
			& \leq C\left(K_{N, p}^p\|\nabla u\|_q^{\nu_{b, q} b}\|u\|_2^{(1-\nu_{b, q}) b}+\mathcal{S}_q^{-q^* / q}\|\nabla u\|_q^{q^*}\right)\\
			& \leq C\left(K_{N, p}^pc^{(1-\nu_{b, q}) b}\|\nabla u\|_q^{\nu_{b, q} b}+\mathcal{S}_q^{-q^* / q}\|\nabla u\|_q^{q^*}\right).
		\end{aligned}
	\end{equation}
	Now, from $u \in \mathcal{P}_{-}$ and \eqref{Eq-J2der1}, we have
	\begin{equation}
		\begin{aligned}
			\frac{4}{N^2}\left(\delta_q+1\right)\left(q\left(\delta_q+1\right)-2\right)\|\nabla u\|_q^q &
			<\int_{\mathbb{R}^N}h_1(u) u-\bar{2} H_1(u) d x+
			\int_{\mathbb{R}^N}h_2(u) u-\bar{2} H_2(u)d x \\
			& \leq \int_{\mathbb{R}^N}(a-\bar{2}) H_1(u) d x+\int_{\mathbb{R}^N}\left(q^*-\bar{2}\right) H_2(u) d x \\
			& \leq\left(q^*-\bar{2}\right) \int_{\mathbb{R}^N} H_2(u) d x \\
			& \leq\left(q^*-\bar{2}\right) C\left(K_{N, p}^pc^{(1-\nu_{b, q}) b}\|\nabla u\|_q^{\nu_{b, q} b}+\mathcal{S}_q^{-q^* / q}\|\nabla u\|_q^{q^*}\right).
		\end{aligned}
	\end{equation}
	Note that $\bar{q}<b<q^*$, $q<\nu_{b, q}b<q^* $, using Lemma \ref{Lem-A1}, we get that
	$\|\nabla u\|_q$ is bounded away from 0. If \eqref{Eq-crange} also holds, from Remark \ref{Re-A1}, $\inf _{\mathcal{P}_{-} \cap \mathcal{D}(c)}\|\nabla u\|_{\widetilde{q}}>M_1$, where $M_1$ is a constant independent of $c$.
	
	Similarly, when $\frac{2 N}{N+2}<q<2$, from \eqref{Eq-J2der2} and Lemma \ref{Lemma-GN inequality2}, using the similar arguments, we can deduce that
	$\|\nabla u\|_2$ is bounded away from 0.
	In addition,  if \eqref{Eq-crange} holds, then
	$\inf _{\mathcal{P}_{-} \cap \mathcal{D}(c)}\|\nabla u\|_{2}>M_2$.
	Combining the above two cases, we can  conclude that $\inf _{\mathcal{P}_{-} \cap \mathcal{D}(c)}\|\nabla u\|_{\widetilde{q}}>M:=\min\{M_1, M_2\}$.
\end{proof}

\subsection{Estimations of $H_1$ and $H_2$}
\
\newline
	In this subsection, we will give some estimates for \( H_1 \) and \( H_2 \).
And we assume that $F_1$ and $F_2$ satisfy what follows:
\begin{enumerate}[label=\textup{(HF$^{\prime}$)}, ref=\textup{HF$^{\prime}$}]
	\item \label{HF}
	There exist $2 < a_1<a_2 < q_{\#}$, $q^{\#} < b_1<b_2 < q^{\prime}$ such that
	$$
	(a_{1}-2)F_1(t)\leq H_1(t)\leq (a_2-2)F_1(t), \quad (b_1-2) F_2(t) \leq H_2(t)\leq(b_{2}-2) F_2(t) \quad \text{for all} \quad t \in \mathbb{R}.
	$$
\end{enumerate}
\eqref{HF} is
more relaxed than \eqref{H-F}.
\begin{lemma}\label{Lemma-HF>0}
	If \eqref{F0}, \eqref{F2}, \eqref{F4}, \eqref{J0}, \eqref{H0}, \eqref{HF} and \eqref{Eq-crange} hold and $c>0$ is sufficiently small, then for every $u \in \mathcal{P}^{}_{-} \cap \mathcal{D}(c)$
	$$
	\int_{\mathbb{R}^N} \left( a_1-q^{\#} \right) H_1(u) d x+
	\int_{\mathbb{R}^N} \left(b_1-q^{\#} \right)H_2(u) dx\geq0.
	$$
\end{lemma}
\begin{proof}
	From \eqref{Eq-Functional}, \eqref{Eq-poH3} and Lemma \ref{Lemma-J>0onP}, we have that, for $u \in \mathcal{P}^{}_{-} \cap \mathcal{D}(c)$
	\begin{equation}\label{Eq-J-1P}
		\begin{aligned}
			& J(u)-\frac{P(u)}{q\left(\delta_q+1\right)}
			=\left(\frac{1}{2}-\frac{1}{q\left(\delta_q+1\right)}\right)\|\nabla u\|_2^2+ \int_{\mathbb{R}^N} \frac{1}{\bar{q}-2}H(u) d x-\int_{\mathbb{R}^N} F(u) d x>0 ,
		\end{aligned}
	\end{equation}
	\begin{equation}\label{Eq-J-2P}
		\begin{aligned}
			& J(u)-\frac{P(u)}{2}
			=\left(\frac{1}{q}-\frac{\left(\delta_q+1\right)}{2}\right)\|\nabla u\|_q^q+\int_{\mathbb{R}^N} \frac{N}{4} H(u) d x-\int_{\mathbb{R}^N} F(u) d x>0.
		\end{aligned}
	\end{equation}
	Firstly, we consider the case that $2<q<N$.
	Observe that $\frac{1}{q}-\frac{\left(\delta_q+1\right)}{2}<0$, \eqref{Eq-J-2P} and Lemma \ref{Lemma-u>0onP} implies that, there exist $\widetilde{M}>0$
	independent of $c$ such that
	$$
	\int_{\mathbb{R}^N} \frac{N}{4} H(u) d x-\int_{\mathbb{R}^N} F(u) d x>\widetilde{M}.
	$$
	Thus, from \eqref{H0} and \eqref{HF} it follows
	\begin{equation}\label{Eq-M}
		\begin{aligned}
			\widetilde{M}&<\int_{\mathbb{R}^N} \frac{N}{4} H(u) d x-\int_{\mathbb{R}^N} F(u) d x\\
			&=\int_{\mathbb{R}^N} \frac{N}{4}H_1(u)-F_1(u) d x+
			\int_{\mathbb{R}^N} \frac{N}{4}H_2(u)-F_2(u) d x\\
			&\leq
			\int_{\mathbb{R}^N} \left(\frac{N}{4}-\frac{1}{a_2-2}\right) H_1(u) d x+
			\int_{\mathbb{R}^N} \left(\frac{N}{4}-\frac{1}{b_2-2}\right)H_2(u) dx\\
			&\leq\int_{\mathbb{R}^N} \left(\frac{N}{4}-\frac{1}{b_2-2}\right)H_2(u) dx\\
			&\leq \widetilde{C}_1\left(\|u\|_2^2+\|u\|_b^b\right)
		\end{aligned}
	\end{equation}
	where $$\widetilde{C}_1:=C\left(\frac{N}{4}-\frac{1}{b_2-2}\right).$$
	Since $b_2>\bar{2}$, we have $\widetilde{C}_1>0$.
	Denote $\widetilde{M}_1:=\left(\frac{\widetilde{M}}{2\widetilde{C}_1}\right)^{\frac{1}{b}}$,
	from \eqref{Eq-M},
	for sufficiently small $c>0$,
	$$
	\|u\|_b\geq\left(\frac{\widetilde{M}}{2\widetilde{C}_1}\right)^{\frac{1}{b}}=\widetilde{M}_1.
	$$
	From \eqref{H1}, there exists $C^{\prime}>0$ such that $H_2(t)\geq C^{\prime}|t|^b$ for any $|t|>1$. Since $H_2(t)\geq 0$ for any $|t|\leq 1$, Therefore, we have
	\begin{equation}\label{Eq-H2est}
	H_2(t)\geq C^{\prime}\left(|t|^b-|t|^a \right) \quad \text { for all } t \in \mathbb{R} \text {. }
	\end{equation}
	Since $a_1<\bar{q}$, $b_1>\bar{q}$, from \eqref{Eq-H2est} and \eqref{H0}, we have that	
	\begin{equation}\label{Eq-Hest1}
		\begin{aligned}
			&\quad
			\int_{\mathbb{R}^N} \left(a_1-\bar{q}\right) H_1(u) d x+
			\int_{\mathbb{R}^N} \left(b_1-\bar{q}\right)H_2(u) dx\\
			&\geq -\widetilde{C}_2\left(\|u\|_2^2+\|u\|_a^a\right)+
			\widetilde{C}_3\left(\|u\|_{b}^{b}-
			\|u\|_{a}^{a}\right)
			,
		\end{aligned}
	\end{equation}
	where
	$$
	\widetilde{C}_2:=-C\left(a_1-\bar{q}\right) \text {, } \widetilde{C}_3:=C^{\prime}\left(b_1-\bar{q}\right) \text {.}
	$$	
	Using Lemma \ref{Lemma-Interpolationinequality} and Remark \ref{Remark-inter}, we get
	$$
	\|u\|_a^a\leq\|u\|_2^{at}\|u\|_b^{a(1-t)},
	$$
	where $a(1-t)<b$. So there exists $\delta>0$ such that $b-\delta=a(1-t)$.  Therefore
	
	\begin{equation}\label{Eq-Hest2}
		\begin{aligned}
			&\quad-\widetilde{C}_2\left(\|u\|_2^2+\|u\|_a^a\right)+\widetilde{C}_3\left(\|u\|_{b}^{b}-\|u\|_{a}^{a}\right)\\
			&=\widetilde{C}_3\|u\|_b^b
			-\left(\widetilde{C}_2+\widetilde{C}_3\right)\|u\|_a^a
			-\widetilde{C}_2\|u\|_2^2	\\
			&\geq\widetilde{C}_3\|u\|_b^b-
			\left(\widetilde{C}_2+\widetilde{C}_3\right)\|u\|_2^{at}\|u\|_b^{b-\delta}
			-\widetilde{C}_2\|u\|_2^2\\
			&=\widetilde{C}_3\|u\|_b^{b-\delta}
			\left(\|u\|_b^{\delta}-	\left(\widetilde{C}_2+\widetilde{C}_3\right)c^{at}
			\right)-\widetilde{C}_2c^2\\
			&\geq
			\widetilde{C}_3\|u\|_b^{b-\delta}
			\left(\widetilde{M}_1^{\delta}-	\left(\widetilde{C}_2+\widetilde{C}_3\right)c^{at}
			\right)-\widetilde{C}_2c^2.
		\end{aligned}
	\end{equation}
	Choose
	$$c\leq
	\min \left\{
	\left(\frac{\widetilde{M}_1}{2\left(\widetilde{C}_2+\widetilde{C}_3\right)}\right)^{\frac{\delta}{a t}} , \left(\frac{\widetilde{C}_3\widetilde{M}_1}{2 \widetilde{C}_2}\right)^{\frac{1}{2}}\right\},
	$$
	then
	\begin{equation}\label{Eq-Hest3}
		\begin{aligned}
			&\quad\widetilde{C}_3\|u\|_b^{b-\delta}
			\left(\widetilde{M}_1^{\delta} - \left(\widetilde{C}_2 + \widetilde{C}_3\right) c^{at}\right)
			- \widetilde{C}_2 c^2 \\
			&\geq \frac{1}{2} \widetilde{C}_3 \|u\|_b^{b-\delta} \widetilde{M}_1^{\delta}
			- \widetilde{C}_2 c^2 \\
			&\geq \frac{1}{2} \widetilde{C}_3 \widetilde{M}_1^{b}
			- \widetilde{C}_2 c^2 \\
			&\geq 0.
		\end{aligned}
	\end{equation}
	Now, combining \eqref{Eq-Hest1}, \eqref{Eq-Hest2} and \eqref{Eq-Hest3}, we have
	$$
	\int_{\mathbb{R}^N} \left(a_1-\bar{q}\right) H_1(u) d x+
	\int_{\mathbb{R}^N} \left(b_1-\bar{q}\right)H_2(u) dx\geq0.
	$$
{The proof is similar when $\frac{2 N}{N+2}<q<2$, so we omit it here.}
\end{proof}

	Using Lemma \ref{Lemma-HF>0}, we can prove the following three estimates that will be useful in Lemma \ref{Lemma-inf-attained} below.

\begin{lemma}\label{Remark-Hh}
	Under the assumptions of Lemma \ref{Lemma-HF>0}, for every $u \in \mathcal{P}^{}_{-} \cap \mathcal{D}(c)$ the following inequality holds
	$$
	\int_{\mathbb{R}^N} h(u) u d x-q^{\#} \int_{\mathbb{R}^N} H(u) d x\geq 0.
	$$
\end{lemma}

\begin{proof}
	From \eqref{HF}, we have that
	\begin{equation}
		\begin{aligned}
			& \int_{\mathbb{R}^N} h(u) u d x-q^{\#} \int_{\mathbb{R}^N} H(u) d x \\
			= & \int_{\mathbb{R}^N} h_1(u) u-q^{\#} H_1(u) d x+\int_{\mathbb{R}^N} h_2(u) u-q^{\#} H_2(u) d x \\
			\geq & \int_{\mathbb{R}^N}\left(a_1-q^{\#}\right) H_1(u) d x+\int_{\mathbb{R}^N}\left(b_1-q^{\#}\right) H_2(u) d x \\
			\geq & 0,
		\end{aligned}
	\end{equation}
the last inequality follows from Lemma \ref{Lemma-HF>0}.
\end{proof}
\begin{lemma}\label{Remark-H1H22}
Under the assumptions of Lemma \ref{Lemma-HF>0}, for every $u \in \mathcal{P}^{}_{-} \cap \mathcal{D}(c)$, the following inequality holds
	$$
\int_{\mathbb{R}^N} \left(\frac{1}{{q}^{\#}-2}-\frac{1}{a_1-2}\right) H_1(u) d x+
\int_{\mathbb{R}^N} \left(\frac{1}{{q}^{\#}-2}-\frac{1}{b_1-2}\right)H_2(u) dx\geq0.
	$$

\end{lemma}

\begin{proof}
	It suffices to prove
	$$
	-\frac{\int_{\mathbb{R}^N} \left(\frac{1}{{q}^{\#}-2}-\frac{1}{b_1-2}\right)H_2(u) dx}{\int_{\mathbb{R}^N} \left(\frac{1}{{q}^{\#}-2}-\frac{1}{a_1-2}\right) H_1(u) d x}\geq1.$$
	Since
\begin{equation*}
	\begin{aligned}
		 -\frac{\frac{1}{q^{\#}-2}-\frac{1}{b_1-2}}{\frac{1}{q^{\#}-2}-\frac{1}{a_1-2}}
		= -\frac{1-\frac{q^{\#}-{2}}{b_1-2}}{1-\frac{q^{\#}-2}{a_1-2}}
		= -\frac{b_1-q^{\#}}{a_1-q^{\#}}\left(\frac{b_1-2}{a_1-2}\right)
		\geq  -\frac{b_1-q^{\#}}{a_1-q^{\#}},
	\end{aligned}
\end{equation*}
we have
$$
-\frac{\left(\frac{1}{{q}^{\#}-2}-\frac{1}{b_1-2}\right)\int_{\mathbb{R}^N} H_2(u) dx}{ \left(\frac{1}{{q}^{\#}-2}-\frac{1}{a_1-2}\right)\int_{\mathbb{R}^N} H_1(u) d x}\geq
-\frac{\left(b_1-q^{\#} \right)\int_{\mathbb{R}^N} H_2(u) dx}{ \left( a_1-q^{\#} \right)\int_{\mathbb{R}^N} H_1(u) d x}\geq1.
$$
Thus, the proof of Lemma \ref{Remark-H1H22} is completed.
\end{proof}

\begin{lemma}\label{Remark-H1H2}
	Under the assumptions of Lemma \ref{Lemma-HF>0}, for every $u \in \mathcal{P}^{}_{-} \cap \mathcal{D}(c)$, the following inequality holds
	$$
	\int_{\mathbb{R}^N}\frac{1}{{q^{\#}}-2} H(u) d x-\int_{\mathbb{R}^N} F(u) d x\geq 0.
	$$
\end{lemma}

	\begin{proof}
	Since when $2<q<N$,	we have that
	\begin{equation*}
		\begin{aligned}
			&\quad\int_{\mathbb{R}^N}\frac{1}{\bar{q}-2}  H(u) d x-\int_{\mathbb{R}^N} F(u) d x\\
			&= \int_{\mathbb{R}^N} \frac{1}{\bar{q}-2}H_1(u)-F_1(u) d x+
			\int_{\mathbb{R}^N} \frac{1}{\bar{q}-2}H_2(u)-F_2(u) d x\\
			&\geq
			\int_{\mathbb{R}^N} \left(\frac{1}{\bar{q}-2}-\frac{1}{a_1-2}\right) H_1(u) d x+
			\int_{\mathbb{R}^N} \left(\frac{1}{\bar{q}-2}-\frac{1}{b_1-2}\right)H_2(u) dx\\
			&\geq0.
		\end{aligned}
	\end{equation*}
The proof is similar when $\frac{2 N}{N+2}<q<2$, so we omit it here.
\end{proof}

Using the above estimates, we show that $  \inf _{\mathcal{P}_{-} \cap \mathcal{D}(c)} J $ is attained.
In the next lemma, we need the following condition:
\begin{enumerate}[label=(H\arabic*), ref=\textup{H\arabic*}]
	\setcounter{enumi}{+3}
	\item \label{H4}
	$\lim _{t \rightarrow 0} \frac{H(t)}{t^2}=0$.
\end{enumerate}
\eqref{H4} follows from \eqref{F1} and \eqref{HF}.
\begin{lemma}\label{Lemma-inf-attained}
	If \eqref{F0}, \eqref{F1}, \eqref{F3},
	{\eqref{F5},}
	 \eqref{J0}, \eqref{H0}, \eqref{H-1}, \eqref{HF} and \eqref{Eq-crange} hold and $c$ is sufficiently small, then $\inf _{\mathcal{P}_{-}^{\mathrm{rad}} \cap \mathcal{D}(c)} J$ is attained; if $f$ is odd or $\left.f\right|_{(-\infty, 0)} \equiv 0$, then $\inf _{\mathcal{P}_{-}^{\mathrm{rad}} \cap \mathcal{D}(c)} J=\inf _{\mathcal{P}_{-} \cap \mathcal{D}(c)} J$ and it is attained by a non-negative and non-increasing (in the radial coordinate) function.
\end{lemma}
\begin{proof}
	Let $\left(u_n\right)_n \subset \mathcal{P}_{-}^{\text {rad }} \cap \mathcal{D}(c)$ be a minimizing sequence of $J$.
	From Lemma \ref{Lemma-coercive}, $\left(u_n\right)_n$ is bounded in $X_{\mathrm{rad}}$.
	Therefore, up to a subsequence, $u_n \rightharpoonup \widetilde{u}$ in $X_{\mathrm{rad}}$
	and $u_n \rightarrow \widetilde{u}$ in $L^p\left(\mathbb{R}^N\right)$ for $p\in(2,q^{\prime})$ and a.e. in $\mathbb{R}^N$ for $\widetilde{u} \in \mathcal{D}(c)$. From \eqref{F1}, \eqref{F5} and \eqref{HF}, we can deduce that
	\eqref{H3} and \eqref{H4} hold. From \eqref{H3} and \eqref{H4}, for every $\varepsilon>0$ there exists $C_{\varepsilon}>0$ such that
	$$
	H(t) \leq \varepsilon\left(|t|^2+|t|^{q^{\prime}}\right)+C_{\varepsilon}|t|^p \
	\text{ for every } \ t \in \mathbb{R}.
	$$
	This and $\left(u_n\right)_n \subset \mathcal{P}$ imply that
	\begin{equation}\label{Eq-429}
		\begin{aligned}	
			\frac{N}{2} \int_{\mathbb{R}^N} H(\widetilde{u}) \mathrm{d} x&=\lim_{n \rightarrow \infty} \frac{N}{2} \int_{\mathbb{R}^N} H\left(u_n\right) \mathrm{d} x\\
			& =
			\lim_{n \rightarrow \infty}\left(\left\|\nabla u_n\right\|_2^2 +(\delta_q+1)\left\|\nabla u_n\right\|^q_q  \right)\\
			&\geq\left\|\nabla \widetilde{u}\right\|_2^2 +(\delta_q+1)\left\|\nabla \widetilde{u}\right\|^q_q .
		\end{aligned}
	\end{equation}
	Additionally, $\widetilde{u} \neq 0$ because, otherwise, \eqref{Eq-429} would yield $\lim_{n \rightarrow \infty} \left(\left\|\nabla u_n\right\|_2+(\delta_q+1)\left\|\nabla u_n\right\|^q_q\right) =0$, in contrast with Lemma \ref{Lemma-u>0onP}. There follows
	$$
	0<\left\|\nabla \widetilde{u}\right\|_2^2 +(\delta_q+1)\left\|\nabla \widetilde{u}\right\|^q_q  \leq \frac{N}{2} \int_{\mathbb{R}^N} H(\widetilde{u}) \mathrm{d} x.
	$$
	So we can define $\widetilde{t}:=t(\widetilde{u}) \geq 1$, which is the unique zero point of $G(t)$, see Lemma \ref{Lemma-Pnotempty}.
	From Lemma \ref{Lemma-Pnotempty},
	 $\widetilde{u}(\widetilde{t} x)\in \mathcal{P}^{\mathrm{rad}}$. Note that
	$$
	\|\widetilde{u}(\widetilde{t} x)\|_2^2=\widetilde{t}^{-N}\|\widetilde{u}\|_2^2 \leq \widetilde{t}^{-N} c^2 \leq c^2
	$$
	hence $\widetilde{u}(\widetilde{t} x) \in \mathcal{P}^{\mathrm{rad}} \cap \mathcal{D}(c)$.
	Observe that,  from \eqref{H-1}, we can see \eqref{H3} and \eqref{H4} still hold replacing $H(t)$ with $h(t) t$. Then, when $2<q<N$, from Lemma \ref{Remark-Hh},
	\begin{equation}
		\begin{aligned}
			\int_{\mathbb{R}^N} h(\widetilde{u}(\widetilde{t} x)) \widetilde{u}(\widetilde{t} x)- \bar{q}H(\widetilde{u}(\widetilde{t} x)) d x
			& =\widetilde{t}^{-N} \int_{\mathbb{R}^N} h(\widetilde{u}) \widetilde{u}-\bar{q} H(\widetilde{u}) d x \\
			& =\widetilde{t}^{-N} \lim_{n \rightarrow\infty} \int_{\mathbb{R}^N}  h\left(u_n\right) u_n-\bar{q}H\left(u_n\right) d x \geq 0.
		\end{aligned}
	\end{equation}	
	Since $\left(2-q\left(\delta_q+1\right)\right)<0$,
	\begin{equation}
		\left(2-q\left(\delta_q+1\right)\right)\|\nabla \widetilde{u}(\widetilde{t} x)\|_2^2 -\frac{N^2}{4}\left(\int_{\mathbb{R}^N} h(\widetilde{u}(\widetilde{t} x)) \widetilde{u}(\widetilde{t} x) d x-\bar{q}\int_{\mathbb{R}^N} H(\widetilde{u}(\widetilde{t} x)) d x\right)\leq0.
	\end{equation}
	From \eqref{Eq-J2der2}, we have that  $\widetilde{u}(\widetilde{t} \cdot) \in \mathcal{P}_{-}^{\mathrm{rad}} \cap \mathcal{D}(c)$ or $\widetilde{u}(\widetilde{t} \cdot) \in\mathcal{P}_{0}^{\mathrm{rad}} \cap \mathcal{D}(c)$. From Lemma \ref{Lemma-P0}, Since $\mathcal{P}_{0}^{} \cap \mathcal{D}(c)=\emptyset$, we have $\widetilde{u}(\widetilde{t} \cdot) \in \mathcal{P}_{-}^{\mathrm{rad}} \cap \mathcal{D}(c)$.
	Similarly, when $\frac{2 N}{N+2}<q<2$, the same conclusion can be obtained from \eqref{Eq-J2der1}.
	In addition, when $2<q<N$ and $c>0$ small enough, by Lemma \ref{Remark-H1H2} and \eqref{Eq-J-1P}
	\begin{align*}
		0&<\inf _{\mathcal{P}^{\text {rad }}_{-} \cap \mathcal{D}(c)} J \\
		&\leq J(\widetilde{u}(\widetilde{t} \cdot))\\
		&=\widetilde{t}^{2-N}\left(\frac{1}{2}-\frac{1}{q\left(\delta_q+1\right)}\right)\left\|\nabla \widetilde{u}\right\|_2^2 +\widetilde{t}^{-N}\left(\frac{1}{\bar{q}-2} \int_{\mathbb{R}^N} H\left(\widetilde{u}\right) d x-\int_{\mathbb{R}^N} F\left(\widetilde{u}\right) d x\right) \\
		&\leq
		\left(\frac{1}{2}-\frac{1}{q\left(\delta_q+1\right)}\right)\left\|\nabla \widetilde{u}\right\|_2^2 +\left(\frac{1}{\bar{q}-2} \int_{\mathbb{R}^N} H\left(\widetilde{u}\right) d x-\int_{\mathbb{R}^N} F\left(\widetilde{u}\right) d x\right)\\
		&\leq \varliminf_{n\rightarrow \infty} \left(\frac{1}{2}-\frac{1}{q\left(\delta_q+1\right)}\right)\left\|\nabla u_n\right\|_2^2 +\left(\frac{1}{\bar{q}-2} \int_{\mathbb{R}^N} H\left(u_n\right) d x-\int_{\mathbb{R}^N} F\left(u_n\right) d x\right)\\
		&=\varliminf_{n\rightarrow \infty} J(u_n)\\
		&\leq
		\inf _{\mathcal{P}^{\text {rad }}_{-} \cap \mathcal{D}(c)} J.
	\end{align*}
	Hence $\widetilde{t}=1$, and $\widetilde{u} \in \mathcal{P}_{-}^{\mathrm{rad}} \cap \mathcal{D}(c)$, so $J(\widetilde{u})=\inf _{\mathcal{P}_{-}^{\mathrm{rad}} \cap \mathcal{D}(c)} J$. When $\frac{2 N}{N+2}<q<2$, the proof is similar, therefore, we omit it.
	
	Now, assume that $f$ is odd. If we denote by $u^*$ the Schwarz rearrangement of $|\widetilde{u}|$, we have that
	$$
	\left\|u^*\right\|_2=\|\widetilde{u}\|_2, \quad\left\|\nabla u^*\right\|_2 \leq\|\nabla \widetilde{u}\|_2, \quad\left\|\nabla u^*\right\|_q \leq\|\nabla \widetilde{u}\|_q,
	$$
	$$
	\int_{\mathbb{R}^N} F\left(u^*\right) \mathrm{d} x=\int_{\mathbb{R}^N} F(\widetilde{u}) \mathrm{d} x, \quad \int_{\mathbb{R}^N} H\left(u^*\right) \mathrm{d} x=\int_{\mathbb{R}^N} H(\widetilde{u}) \mathrm{d} x, \quad \int_{\mathbb{R}^N} h\left(u^*\right) u^* \mathrm{~d} x=\int_{\mathbb{R}^N} h(\widetilde{u}) \widetilde{u} \mathrm{~d} x,
	$$
	which implies that $t\left(u^*\right) \geq 1$, Then, arguing as above, we obtain that $t\left(u^*\right)=1$. Therefore, $u^* \in \mathcal{P}_{-}^{\text {rad }}\cap \mathcal{D}(c)$, $\left\|\nabla u^*\right\|_2=\|\nabla \widetilde{u}\|_2$, and $J\left(u^*\right)=J(\widetilde{u})$. If $\left.f\right|_{(-\infty, 0)} \equiv 0$, we consider the Schwarz rearrangement of $\max \{\widetilde{u}, 0\}$ and then an almost same argument applies.
\end{proof}

		In the following lemma, we show that
		 the minimizer $u$ of $J$ on $\mathcal{P}^{\text{rad}}_{-}\cap \mathcal{D}(c) $ is achieved on $ \mathcal{P}^{\text{rad}}_{-}\cap\mathcal{S}(c)$.
		 To show this,
		we prove that for any $u \in\mathcal{P}^{\text{rad}}_{-}\cap \left(\mathcal{D}(c) \backslash \mathcal{S}(c)\right)$, the crucial inequality $\inf _{\mathcal{P}^{\text{rad}}_{-}\cap \mathcal{D}(c)} J<J(u)$ holds. Notice Remark \ref{remark-2} (ii), in the following Lemma \ref{LemmaPJ<J}
		we only need \eqref{F0}, \eqref{F3}, \eqref{H-0}, \eqref{H-1}, \eqref{J0}, \eqref{H-F} and \eqref{Eq-crange} hold.
	\begin{lemma}\label{LemmaPJ<J}
		If \eqref{F0}, \eqref{F3}, \eqref{H-0}, \eqref{H-1}, \eqref{J0} and \eqref{H-F} and \eqref{Eq-crange} hold and $c>0$ is sufficiently small, then for any $u \in\mathcal{P}_{-}^{\text{rad}} \cap\left(\mathcal{D}(c) \backslash \mathcal{S}(c)\right)$, there holds
		$$
		\inf _{\mathcal{P}^{\text{rad}}_{-}\cap{\mathcal{D}(c)}} J(u)<J(u).
		$$
	\end{lemma}
	\begin{proof}
		Assume by contradiction that there exists $\check{u} \in \mathcal{P}^{\text{rad}}_{-}\cap\left(\mathcal{D}(c) \backslash \mathcal{S}(c)\right)$ such that $	\inf _{\mathcal{P}^{\text{rad}}_{-}\cap{\mathcal{D}(c)}} J(u)=J(\check{u}) \leq$ $\inf _{\mathcal{P}_{-}\cap \mathcal{S}(c)} J$. Hence $\check{u}$ is a local minimizer of $J$ on $\mathcal{P}^{\text{rad}}_{-}\cap \left(\mathcal{D}(c) \backslash \mathcal{S}(c)\right)$. On the other hand, $\mathcal{P}^{\text{rad}}_{-}\cap \left(\mathcal{D}(c) \backslash \mathcal{S}(c)\right)$ is an open set in $\mathcal{P}^{\text{rad}}_{-}$, we find that $\check{u}$ is a local minimizer of $J$ on $\mathcal{P}^{\text{rad}}_{-}$. Hence there is a Lagrange multiplier $\check{\mu} \in \mathbb{R}$ such that
		\begin{equation*}
			\begin{aligned}
				& J^{\prime}(\check{u}) v+\check{\mu}
				\left(2 \int_{\mathbb{R}^N}|\nabla \check{u}|^{p-2} \nabla \check{u} \nabla v d x+q\left(\delta_q+1\right) \int_{\mathbb{R}^N}|\nabla \check{u}|^{q-2} \nabla \check{u} \nabla v d x-\frac{N}{2} \int_{\mathbb{R}^N} h(\check{u}) v d x\right)=0
			\end{aligned}
		\end{equation*}
		for any $v \in C_0^{\infty}\left(\mathbb{R}^N\right)$. Hence $\check{u}$ is a weak solution to
		$$
		-(1+2\check{\mu} ) \Delta \check{u}-\left(1+\check{\mu} q\left(\delta_q+1\right)\right) \Delta_q \check{u}=f(\check{u})+\frac{N \check{\mu} }{2} h(\check{u}) .
		$$
		In particular, $\check{u}$ satisfies the following Nehari-type identity
		$$
		(1+\check2{\lambda} )\|\nabla \check{u}\|_2^2+\left(1+\check{\mu} q\left(\delta_q+1\right)\right)\|\nabla \check{u}\|_q^q=\int_{\mathbb{R}^N} f(\check{u}) \check{u} d x+\frac{N \check{\mu} }{2} \int_{\mathbb{R}^N} h(\check{u}) \check{u} d x.
		$$
		If $\check{\mu} =-\frac{1}{q\left(\delta_q+1\right)}$, then
		$$
		\left(1-\frac{2}{q\left(\delta_q+1\right)}\right)\|\nabla \check{u}\|_2^2=\int_{\mathbb{R}^N} f(\check{u}) \check{u} d x-\frac{N}{2 q(\delta_q+1)}  \int_{\mathbb{R}^N} h(\breve{u}) \check{u} d x .
		$$
		When $2<q<N$, we have $1-\frac{2}{q\left(\delta_q+1\right)}>0$, {therefore by} Lemma \ref{Remark-H1H22}
		\begin{equation*}
			\begin{aligned}
				0\leq& \int_{\mathbb{R}^N} f(u) u d x-\frac{N}{2 q(\delta_q+1)} \int_{\mathbb{R}^N} h(u) u d x \\
				= & 2 \int_{\mathbb{R}^N} F(u) d x+\int_{\mathbb{R}^N} H(u) d x-\frac{N}{2 q(\delta_q+1)} \int_{\mathbb{R}^N} h(u) u d x \\
				= & 2 \int_{\mathbb{R}^N} F(u) d x+\int_{\mathbb{R}^N} H_1(u) d x-\frac{1}{\bar{q}-2} \int_{\mathbb{R}^N} h_1(u) u d x +
				\int_{\mathbb{R}^N} H_2(u) d x-\frac{1}{\bar{q}-2} \int_{\mathbb{R}^N} h_2(u) u d x \\
				\leq & 2 \int_{\mathbb{R}^N} F(u) d x+\left(1-\frac{a_1}{\bar{q}-2}\right) \int_{\mathbb{R}^N} H_1(u) d x+\left(1-\frac{b_1}{\bar{q}-2}\right) \int_{\mathbb{R}^N} H_2(u) d x \\
				\leq & 2 \int_{\mathbb{R}^N} F_1(u) d x+\left(1-\frac{a_1}{\bar{q}-2}\right) \int_{\mathbb{R}^N} H_1(u) d x
				+  2 \int_{\mathbb{R}^N} F_2(u) d x+\left(1-\frac{b_1}{\bar{q}-2}\right) \int_{\mathbb{R}^N} H_2(u) d x \\
				\leq & {\left(1+\frac{2}{a_1-2}-\frac{a_1}{\bar{q}-2}\right) \int_{\mathbb{R}^N} H_1(u) d x } +
				{\left(1+\frac{2}{b_1-2}-\frac{b_1}{\bar{q}-2}\right) \int_{\mathbb{R}^N} H_2(u) d x }\\
				= & {a_1}\left(\frac{1}{a_1-2}-\frac{1}{\bar{q}-2}\right) \int_{\mathbb{R}^N} H_1(u) d x +
				{b_1}\left(\frac{1}{b_1-2}-\frac{1}{\bar{q}-2}\right) \int_{\mathbb{R}^N} H_2(u) d x\\
				\leq& {a_1}\left(\frac{1}{a_1-2}-\frac{1}{\bar{q}-2}\right) \int_{\mathbb{R}^N} H_1(u) d x +
				{a_1}\left(\frac{1}{b_1-2}-\frac{1}{\bar{q}-2}\right) \int_{\mathbb{R}^N} H_2(u) d x\\
				\leq & 0.
			\end{aligned}
		\end{equation*}
		This implies that $\int_{\mathbb{R}^N} H(u) d x=0$. Since $u \in \mathcal{P}$, $\check{u}$ satisfies \eqref{Eq-poH3}, we get $u=0$, which contradicts with $u \in \mathcal{P}$. So $\check{\mu} \neq-\frac{1}{q\left(\delta_q+1\right)}$. Similarly, if we take $\check{\mu}=-\frac{1}{2}$, then
		$$
		\left(1+\frac{q\left(\delta_q+1\right)}{2}\right)
		\|\nabla \check{u}\|_q^q=\int_{\mathbb{R}^N} f(\check{u}) \check{u} d x+\frac{N }{4} \int_{\mathbb{R}^N} h(\check{u}) \check{u} d x,
		$$
		when $\frac{2 N}{N+2}<q<2$. Using the same proof method we get $\check{\mu} \neq-\frac{1}{2}$. Combining the above two cases, it can be concluded that $\check{\mu} \neq-\frac{1}{\widetilde{q}\left(\delta_{\widetilde{q}}+1\right)}$.
		
		Moreover, on the one hand, since $\check{u} \in \mathcal{P}$, we get	
		\begin{equation}\label{Eq-pohN1}
			\|\nabla \check{u}\|_2^2+\left(\delta_q+1\right)\|\nabla \check{u}\|_q^q=\frac{N}{2} \int_{\mathbb{R}^N} h(\check{u})u d x .
		\end{equation}	
		On the other hand, $\check{u}$ satisfies Nehari-type and Pohozaev identities. That is, $\check{u}$ satisfies
		\begin{equation}\label{Eq-pohN2}
			\begin{aligned}
				&(1+2\check{\mu} )\|\nabla \check{u}\|_2^2+\left(\delta_q+1\right)\left(1+\check{\mu} q\left(\delta_q+1\right)\right)\|\nabla \check{u}\|_q^q \\
				 =&\frac{N}{2} \int_{\mathbb{R}^N} H(\check{u}) d x+\check{\mu} \int_{\mathbb{R}^N} \frac{N^2}{4} h(\check{u}) \check{u}-\frac{N^2}{2} H(\check{u}) d x .
			\end{aligned}
		\end{equation}
		Combining \eqref{Eq-pohN1} and \eqref{Eq-pohN2}, we deduce that when $2<q<N$
		$$
		\check{\mu} \left(2-q\left(\delta_q+1\right)\right)\|\nabla \check{u}\|_2^2=\check{\mu} \frac{N^2}{4} \int_{\mathbb{R}^N} h(\check{u}) \check{u}-\bar{q} H(\check{u}) d x.
		$$
		And when $\frac{2 N}{N+2}<q<2$
		$$
		\check{\mu} \left(\delta_q+1\right)\left(q\left(\delta_q+1\right)-2\right)\|\nabla \check{u}\|_q^q=\check{\mu} \frac{N^2}{4} \int_{\mathbb{R}^N} h(\check{u}) \check{u}-\bar{2} H(\check{u}) d x .
		$$
	Combining \eqref{Eq-J2der1} and \eqref{Eq-J2der2},
		we find that $\check{\mu} =0$, thus $\check{u}$ is a weak solution to
		$$
		-\Delta \check{u}-\Delta_q \check{u}=f(\check{u}) .
		$$
		Similarly, we also obtain that $\check{u}$ satisfies
		\begin{equation}\label{Eq-lambda=01}
			\|\nabla \check{u}\|_2^2+\|\nabla \check{u}\|_q^q=\int_{\mathbb{R}^N} f(\check{u}) \check{u} d x
		\end{equation}
		and
		\begin{equation}\label{Eq-lambda=02}
			\|\nabla \check{u}\|_2^2+\left(\delta_q+1\right)\|\nabla \check{u}\|_q^q=\frac{N}{2} \int_{\mathbb{R}^N} H(\check{u}) d x.
		\end{equation}
		Combining \eqref{Eq-lambda=01} and \eqref{Eq-lambda=02}, we deduce that
		$$
		\delta_q\|\nabla \check{u}\|_q^q=\frac{N-2}{2}\left(\int_{\mathbb{R}^N} H(\check{u}) d x-\left(2^*-2\right) \int_{\mathbb{R}^N} F(\check{u}) d x\right)
		$$
		and
		$$
		-\delta_q\|\nabla \check{u}\|_2^2=\frac{N-q}{q}\left(\int_{\mathbb{R}^N} H(\check{u}) d x-\left(q^*-2\right) \int_{\mathbb{R}^N} F(\check{u}) d x\right) .
		$$
		By \eqref{H-F}, we obtain a contradiction. So for any $u \in\left(\mathcal{D}(c) \backslash \mathcal{S}(c)\right) \cap \mathcal{P}_{-}^{\text{rad}}$, there holds
		$$
		\inf _{\mathcal{P}_{-}^{\text{rad}}\cap{\mathcal{D}(c)}} J(u)<J(u).
		$$
		The proof of Lemma \ref{LemmaPJ<J} is completed.
	\end{proof}
	Now we can prove the existence of a second solution to \eqref{Eq-Equation}.	
	
	\textbf{Proof of Theorem \ref{Theorem-existence2}}.	
	By Lemma \ref{Lemma-inf-attained} and \ref{LemmaPJ<J}, we derive that $\inf _{\mathcal{P}_{-}^{\mathrm{rad}} \cap \mathcal{D}(c)} J$ is attained.
	Moreover, if $f$ is odd, then by the regularity in \cite{He2008regularity} and Harnack's inequality in \cite{Tru}, we know that $\inf _{\mathcal{P}_{-}^{\mathrm{rad}} \cap \mathcal{D}(c)} J$ is achieved by $\widetilde{u}>0$, which is a radially symmetric function.
	
	Using Lemma \ref{Lemma-P0}, it can be proved that for every $v \in \mathcal{S}(c) \cap \mathcal{P}_{-}^{\text {rad }}$ the functional $\left(G^{\prime}(v), P^{\prime}(v)\right): X_{\text{rad}} \rightarrow \mathbb{R}^2$ is surjective, where $G(v):=\|v\|_2^2-c^2$ and $P$ is defined in \eqref{Eq-poH3}. Then, from \cite[Proposition A.1]{Mederski1}(it is easy to deduce that Proposition A.1 also holds true in Banach space), there exist
	Lagrange multipliers $\lambda\geq 0$ and $\mu \in \mathbb{R}$ such that $\widetilde{u} \in \mathcal{P}_{-}^{\mathrm{rad}}$ solves
	\begin{equation}
		\begin{aligned}
			-\Delta \widetilde{u}-\Delta_q \widetilde{u}-f(\widetilde{u})+\lambda \widetilde{u}+\mu \left(-2 \Delta \widetilde{u}-q\left(\delta_q+1\right) \Delta_q \widetilde{u}-\frac{N}{2} h(\widetilde{u})\right)=0.
		\end{aligned}
	\end{equation}
	That is
	\begin{equation}
		-(1+2\mu ) \Delta \widetilde{u}-\left(1+\mu q\left(\delta_q+1\right)\right) \Delta_q \widetilde{u}+\lambda \widetilde{u}=f(\widetilde{u})+\frac{N}{2} \mu h(\widetilde{u}).
	\end{equation}
	Since $\lambda\geq 0$, similar to Lemma \ref{LemmaPJ<J}, we can find that $\mu=0$. and
	Lemma \ref{LemmaPJ<J} implies that $\lambda\neq0$. The proof of Theorem \ref{Theorem-existence2} is completed.

\section{Appendix A. Useful estimates}
In this appendix, we prove some estimates that will be used in Lemma \ref{Lemma-P0}.
\begin{lemma}\label{Lem-A1}
	Suppose that $x>0$ satisfies $x^q \leq A x^r+B x^s$, where $q<r<s$ and $A, B>0$. Then
	$$
	x \geq \min \left\{1,(A+B)^{1 /(q-r)}\right\}.
	$$
\end{lemma}	
\begin{proof}
	Let us rewrite the given inequality as
	$$
	1 \leq A x^{r-q}+B x^{s-q}
	$$
	and consider the function $f:[0,\infty) \rightarrow \mathbb{R}$ given by
	$$
	f(t)=1-A t^{r-q}-B t^{s-q} .
	$$
	If we compute
	$$
	f^{\prime}(t)=-A(r-q) t^{r-q-1}-B(s-q) t^{s-q-1}=-t^{r-q-1}\left(A(r-q)+B(s-q) t^{s-r}\right),
	$$
	it is clear that $f$ is decreasing. Denote
	$$
	\xi:=\min \left\{1,(A+B)^{1 /(q-r)}\right\} .
	$$
	If $A+B<1$, then $\xi=1$ and
	$$
	f(\xi)=f(1)=1-A-B>0,
	$$
	so, if $f(x) \leq 0$, then $x \geq \xi$.
	
	On the other hand, if $A+B \geq 1$, then $\xi=(A+B)^{1 /(q-r)} \leq 1$ and
	$$
	f(\xi)=1-\frac{A}{A+B}-B\left(\frac{1}{A+B}\right)^{(s-q) /(r-q)} \geq 1-\frac{A}{A+B}-\frac{B}{A+B}=0,
	$$
	so, if $f(x) \leq 0$, then $x \geq \xi$.
\end{proof}

\begin{remark}\label{Re-A1}
	With the assumptions and notations of Lemma \ref{Lem-A1}, if $A<M$, then $
	x \geq \min \left\{1,(M+B)^{1 /(q-r)}\right\}$. This is because $(M+B)^{1 /(q-r)}$ is monotonically decreasing with respect to A.
\end{remark}
For the convenience of readers, we also provide the following estimate. For proof please see \cite[Lemma A.1.]{bieganowski2024existence}
\begin{lemma}\label{Lem-A2}
Suppose that $x>0$ satisfies $x^2 \leq A+B x^p$, where $p \in(0,2)$ and $A, B>0$ satisfy $B(\sqrt{A}+1 / 2)^p \leq \sqrt{A}+1 / 4$. Then

$$
x \leq \sqrt{A}+\frac{1}{2}.
$$

\end{lemma}

\section{Appendix B. Examples for \eqref{J0} and \eqref{J1}}
In the first part of this appendix, we provide examples of a nonlinear term $f$ that satisfies \eqref{J0} and, under additional assumptions, \eqref{J1}. It generalizes the case given by two different powers.

In the second part, we present some sufficient conditions on the nonlinear term for \eqref{J1} to hold, paired with an example that does not consist merely of powers.\\

\subsection{Multiple powers with rational index}
\
\newline
Consider the nonlinearity $f: \mathbb{R} \rightarrow \mathbb{R}$ given by
$$
f(t):=\sum_{k=0}^K \alpha_k|t|^{a_k-2} t+\sum_{\ell=0}^L \beta_{\ell}|t|^{b_{\ell}-2} t
$$
where $2<a_0<\cdots<a_K<q_{\#}<q^{\#}<b_0<\cdots<b_L \leq q^{\prime}$, $a_k, b_{\ell} \in \mathbb{Q}$, $\alpha_k, \beta_{\ell}>0$, and $K, L \geq 0$. Fix $u \in \mathcal{D}(c) \backslash\{0\}$, with $c$ satisfying \eqref{Eq-crange}. Note that
$$
\psi(s):=J(s * u)=e^{2s}\frac{\|\nabla u\|_2^2}{2}
+ e^{q(\delta_q+1)s}\frac{\|\nabla u\|_q^q}{q}
-\sum_{k=0}^K \frac{\alpha_k}{a_k} e^{a_k\delta_{a_k}s}\|u\|_{a_k}^{a_k}-\sum_{\ell=0}^L \frac{\beta_{\ell}}{b_{\ell}} e^{b_k\delta_{b_k}s}\|u\|_{b_{\ell}}^{b_{\ell}}
$$
and
$$
\begin{aligned}
	\psi^{\prime}(s)
	& =e^{2s}\|\nabla u\|_2^2+(\delta_q+1)e^{q(\delta_q+1)s}\|\nabla u\|_q^q -\sum_{k=0}^K a_k\delta_{a_k} \frac{\alpha_k}{a_k} e^{a_k\delta_{a_k}s}\|u\|_{a_k}^{a_k}-\sum_{\ell=0}^L b_k\delta_{b_k} \frac{\beta_{\ell}}{b_{\ell}} e^{b_k\delta_{b_k}s}\|u\|_{b_{\ell}}^{b_{\ell}} \\
	& =-\sum_{k=0}^K a_k\delta_{a_k} \frac{\alpha_k}{a_k} e^{a_k\delta_{a_k}s}\|u\|_{a_k}^{a_k}+e^{2s}\|\nabla u\|_2^2+(\delta_q+1)e^{q(\delta_q+1)s}\|\nabla u\|_q^q -\sum_{\ell=0}^L b_k\delta_{b_k} \frac{\beta_{\ell}}{b_{\ell}} e^{b_k\delta_{b_k}s}\|u\|_{b_{\ell}}^{b_{\ell}}. \\
\end{aligned}
$$
Note also that
$$
0<a_k\delta_{a_k}<\widehat{q}(\delta_{\widehat{q}}+1)<\widetilde{q}(\delta_{\widetilde{q}}+1)< b_k\delta_{b_k}.
$$
Thus, from Lemma \ref{Lemma-g}, $\psi$ has a global maximum point at a positive level; by calculation, it also has a local minimum point at a negative level.

Let us write
$$
a_k\delta_{a_k}=\frac{A_k}{B_k}, \quad b_k\delta_{b_k}=\frac{C_{\ell}}{D_{\ell}}
$$
for some $A_k, B_k, C_{\ell}, D_{\ell} \in \mathbb{N}$ such that $\operatorname{gcd}\left(A_k, B_k\right)=\operatorname{gcd}\left(C_{\ell}, D_{\ell}\right)=1$, and $m:=\operatorname{lcm}\left(2, B_0, \ldots, B_K, D_0, \ldots, D_L\right)$. Then,
$$
\psi^{\prime}(s)= P\left(e^{s/ m}\right),
$$
where $P$ is the polynomial given by
$$
P(t):=-\sum_{k=0}^K \frac{\alpha_k}{a_k} \frac{A_k}{B_k} t^{m A_k / B_k}\|u\|_{a_k}^{a_k}+ t^{2 m}\|\nabla u\|_2^2
+(\delta_q+1)t^{q(\delta_q+1) m}\|\nabla u\|_q^q
-\sum_{\ell=0}^L \frac{\beta_{\ell}}{b_{\ell}} \frac{C_{\ell}}{D_{\ell}} t^{m C_{\ell} / D_{\ell}}\|u\|_{b_{\ell}}^{b_{\ell}}.
$$
From Descartes' rule of signs, $P$ has at most two positive roots. Thus, from the argument above, $\psi^{\prime}$ has exactly two roots, one of them is a local minimum point of $\psi$ and one of them, $t_u$, is a local (hence global) maximum point of $\psi$. Therefore, \eqref{J0} is satisfied.

Now, to see whether this nonlinearity satisfies \eqref{J1}, let us compute
\begin{equation*}
	\begin{aligned}
		\phi^{\prime \prime}(s) & = s^{-2} \left(
		-\sum_{k=0}^K \frac{\alpha_k}{a_k} \frac{A_k}{B_k} \left(\frac{A_k}{B_k}-1\right)\|u\|_{a_k}^{a_k} s^{A_k / B_k}
		+ s^2 \|\nabla u\|_2^2
		\right. \\
		&\quad \left.
		+ \left(\delta_q+1\right) \left(q \left(\delta_q+1\right)-1\right) s^{q \left(\delta_q+1\right)} \|\nabla u\|_2^2
		- \sum_{\ell=0}^L \frac{\beta_{\ell}}{b_{\ell}} \frac{C_{\ell}}{D_{\ell}} \left(\frac{C_{\ell}}{D_{\ell}}-1\right) \|u\|_{b_{\ell}}^{b_{\ell}} s^{C_{\ell} / D_{\ell}}
		\right)
	\end{aligned}
\end{equation*}
and observe that, for a real number $r, N(r-2) / 2-1 \lesseqgtr 0$ if and only if $r \lesseqgtr 2+2 / N$. Then $\varphi$ is convex (respectively, concave) in a right-hand neighbourhood of the origin if $a_0<2+2 / N$ (respectively, $a_0>2+2 / N$), and, if $a_0=2+2 / N$, then $\varphi$ is convex (respectively, concave) in a right-hand neighbourhood of the origin if $K=0$ (respectively, $K \geq 1$ ).

Let us also recall that $\phi$ is convex in a neighbourhood of its local minimizer, concave in a neighbourhood of its global maximizer, and that, therefore, its second derivative changes sign between such two points. Now, let us consider the following cases for the sign pattern of the non-zero coefficients of $\phi^{\prime \prime}$, i.e. the sequence of the signs of the coefficients ordered by ascending variable exponent.

If $a_K \leq 2+2 / N$, then the sign pattern is $+, \ldots,+,-, \ldots,-$ and $\phi^{\prime \prime}$ changes sign once on $\left(0, e^{t_u}\right)$, hence $\phi$ is concave on $\left(e^{t_u},\infty\right)$.

If $a_0>2+2 / N$ or $q_0=2+2 / N$ and $K \geq 1$, then the sign pattern is $-, \ldots,-,+,-, \ldots,-$ and $\phi^{\prime \prime}$ changes sign twice on $\left(0, e^{t_u}\right)$, hence $\varphi$ is concave on $\left(e^{t_u},\infty\right)$.
Finally, a similar argument applies to the case when the powers $a_k(0 \leq k \leq K)$ and $b_{\ell}(0 \leq \ell \leq L)$ are (positive and) rational multiples of a given real number.

\subsection{Concrete assumptions on $F$}
\
\newline
We introduce
$$
G(t):=h(t) t-\left(2+\frac{2}{N}\right) H(t), \quad t \in \mathbb{R}
$$
where $H(t):=f(t)t-2F(t)$. We may assume that $H$ is of class $\mathcal{C}^1$ and $h=H^{\prime}$ satisfies $|h(t)| \lesssim|t|+|t|^{2^*-1}$ for every $t \in \mathbb{R}$ and there exists $\xi \neq 0$ such that $H(\xi)>0$. In addition, we assume that $G$
satisfies the following conditions:
\begin{enumerate}[label=(G\arabic*), ref=\textup{G\arabic*}]
	\setcounter{enumi}{-1}
	\item \label{G0} $G$ is even.
	\item \label{G1} $\lim \sup _{t \rightarrow 0} \frac{G(t)}{|t|^{q_\#}} \leq 0$.
	\item \label{G2} $\lim _{|t| \rightarrow\infty} \frac{G(t)}{|t|^{q^\#}}=\infty$.
	\item \label{G3}  $t \mapsto \frac{G(t)}{t^{q^{\#}}}$ is increasing on $(0,\infty)$.
\end{enumerate}
Then, for fixed $u \in X \backslash\{0\}$,
\begin{equation}\label{App1}
	\begin{aligned}
	\phi^{\prime \prime}(s)&=\frac{\mathrm{d}^2}{\mathrm{~d} s^2} J(s * u)=\|\nabla u\|_2^2+
	\left(q \left(\delta_q+1\right)-1\right) s^{q \left(\delta_q+1\right)-2}\|\nabla u\|_q^q
	-\frac{N^2}{4} s^{-N-2} \int_{\mathbb{R}^N} G\left(s^\frac{N}{ 2} u\right) d x.\\
	\end{aligned}
\end{equation}
First, recall $\widehat{q}=\min\{2,q\}$, from \eqref{App1}, we observe that
\begin{equation*}
	\begin{aligned}
		\phi^{\prime \prime}(s)&
		=	
		s^{\widehat{q}(\delta_{\widehat{q}}+1)-2}
		\left(s^{2-\widehat{q}(\delta_{\widehat{q}}+1)}\|\nabla u\|_2^2+
		\left(q \left(\delta_q+1\right)-1\right) s^{{(q(\delta_q+1)-\widehat{q}(\delta_{\widehat{q}}+1))}}\|\nabla u\|_q^q
		-\frac{N^2}{4} \int_{\mathbb{R}^N} \frac {G\left(s^\frac{N}{2} u\right)} {\left(s^{ \frac{N}{2}}\right)^{{q}_{\#}}} d x\right).
	\end{aligned}
\end{equation*}
Note that, from \eqref{G1}, $G\left(e^{\frac{N}{2}s} u\right)\leq0$ a.e. in $\operatorname{supp} u$ for sufficiently small $s>0$. Therefore, from Fatou's lemma and \eqref{G1} again,
$$
\lim _{s \rightarrow 0}\int_{\mathbb{R}^N} \frac {G\left(s^\frac{N}{2} u\right)} {\left(s^{ \frac{N}{2}}\right)^{{q}_{\#}}} d x\leq0.
$$
Since $\widehat{q}=\min\{2,q\}$, we have
$2>\widehat{q}(\delta_{\widehat{q}}+1)$ and $q(\delta_q+1)>\widehat{q}(\delta_{\widehat{q}}+1)$. So $\lim _{s \rightarrow 0 }{s^{(2-\widehat{q}(\delta_{\widehat{q}}+1))}}=0$ and $\lim _{s \rightarrow 0 }s^{(q(\delta_q+1)-\widehat{q}(\delta_{\widehat{q}}+1))}=0$. This implies that $	\phi^{\prime \prime}(s)\geq0$ for sufficiently small $s>0$.

Next, recall $\widetilde{q}=\max\{2,q\}$, again from \eqref{App1}, we observe that
\begin{equation*}
	\begin{aligned}
		\phi^{\prime \prime}(s)&
		=	
		s^{\widetilde{q}(\delta_{\widetilde{q}}+1)-2}
		\left(s^{2-\widetilde{q}(\delta_{\widetilde{q}}+1)}\|\nabla u\|_2^2+
		\left(q \left(\delta_q+1\right)-1\right) s^{{(q(\delta_q+1)-\widetilde{q}(\delta_{\widetilde{q}}+1))}}\|\nabla u\|_q^q
		-\frac{N^2}{4} \int_{\mathbb{R}^N} \frac {G\left(s^\frac{N}{2} u\right)} {\left(s^{ \frac{N}{2}}\right)^{{q}^{\#}}} d x\right).
	\end{aligned}
\end{equation*}
Note that, from \eqref{G2}, $G\left(e^{\frac{N}{2}s} u\right)>0$ a.e. in $\operatorname{supp} u$ for sufficiently large $s$. Therefore, from Fatou's lemma and \eqref{G2} again,
$$
\lim _{s \rightarrow \infty}\int_{\mathbb{R}^N} \frac {G\left(s^\frac{N}{2} u\right)} {\left(s^{ \frac{N}{2}}\right)^{{q}_{\#}}} d x=\infty.
$$
Since $\widetilde{q}=\min\{2,q\}$,
we have $2>\widetilde{q}(\delta_{\widetilde{q}}+1)$ and $q(\delta_q+1)>\widetilde{q}(\delta_{\widetilde{q}}+1)$.
So
$\lim _{s \rightarrow 0 }{s^{(2-\widetilde{q}(\delta_{\widetilde{q}}+1))}}=0$ and $\lim _{s \rightarrow 0 }s^{(q(\delta_q+1)-\widetilde{q}(\delta_{\widetilde{q}}+1))}=0$. This implies that $	\phi^{\prime \prime}(s)>0$ for sufficiently large $s>0$.

Therefore, there exist $s>0$ such that $\phi^{\prime \prime}(s)=0$. Now we prove that
$\phi^{\prime \prime}(s)$ has has exactly one zero point. If $\phi^{\prime \prime}(s)=0$, then the following equality hold:
$$
s^{2-\widetilde{q}(\delta_{\widetilde{q}}+1)}\|\nabla u\|_2^2+
\left(q \left(\delta_q+1\right)-1\right) s^{{(q(\delta_q+1)-\widetilde{q}(\delta_{\widetilde{q}}+1))}}\|\nabla u\|_q^q
=\frac{N^2}{4} \int_{\mathbb{R}^N} \frac {G\left(s^\frac{N}{2} u\right)} {\left(s^{ \frac{N}{2}}\right)^{{q}^{\#}}} d x.
$$
Since $\widetilde{q}=\max\{2,q\}$,
it's easy to see that
\begin{equation*}
	s^{2-\widetilde{q}(\delta_{\widetilde{q}}+1)}\|\nabla u\|_2^2+
	\left(q \left(\delta_q+1\right)-1\right) s^{{(q(\delta_q+1)-\widetilde{q}(\delta_{\widetilde{q}}+1))}}\|\nabla u\|_q^q
\end{equation*}
 is strictly decreasing for $s\in(0,\infty)$. Similarly, we can deduce from \eqref{G3}, that
 $$
 \frac{N^2}{4} \int_{\mathbb{R}^N} \frac {G\left(s^\frac{N}{2} u\right)} {\left(s^{ \frac{N}{2}}\right)^{{q}^{\#}}} dx
 $$
  is increasing for $s\in(0,\infty)$.
  Therefore,
$\phi^{\prime \prime}(s)$ has exactly one zero point. So, there exist $s_0\in[0,\infty)$ such that $\phi^{\prime \prime}(s)<0$ on $[0,s_0)$,
$\phi^{\prime \prime}(s)>0$ on $(s_0, \infty)$
and $\phi^{\prime \prime}(s_0)=0$.

Suppose by contradiction, that \eqref{J0} do not hold. Then there exist $t_1,t_2\in
(-\infty, \infty)$ and
 $t_1\neq t_2$ such that $t_1$ and $t_2$ are local maximum points of
 $\psi(t)$. Since $\phi(s)= \psi( \ln s) $, we have that
 $$\phi^{\prime}(s)= \frac{\psi^{\prime}( \ln s) }{s}.$$
Therefore, we have \(\phi^{\prime}(e^{t_1}) = \phi^{\prime}(e^{t_2}) = 0\). By applying Rolle's Theorem, there exists \(s_1 \in (e^{t_1}, e^{t_2})\) such that \(\phi^{\prime\prime}(s_1) = 0\).
Since \(t_2\) is a local maximum point of \(\psi(t)\), so \(\psi^{\prime\prime}({t_2}) < 0\), it follows that \(\phi^{\prime\prime}(e^{t_2}) < 0\).
Due to the monotonicity of \(\phi^{\prime\prime}(s)\), there exists \(s_2 > e^{t_2}\) such that \(\phi^{\prime\prime}(s_2) = 0\). This result contradicts the uniqueness of \(s_0\). Therefore, we conclude that \eqref{J0} holds, and it is easy to verify that \eqref{J1} also holds.

In this case, we propose the following example:
$$
F(t)=\frac{1}{a}|t|^{a}+\frac{1}{b}|t|^{b}.
$$
where $a\leq\min\{{2+\frac{2}{N}}, q_{\#}\}<q^{\#}<b$. $F(t)$ satisfies \eqref{F0}-\eqref{F5}. Then
$$
G(t)=\left(1-\frac{2}{a}\right)\left(a-\left(2+\frac{2}{N}\right)\right)|t|^a+\left(1-\frac{2}{b}\right)\left(b-\left(2+\frac{2}{N}\right)\right)|t|^b.
$$
satisfies \eqref{G0}--\eqref{G3}.

Additionally, we provide the following example, which includes cases other than power growth:

$$
F(t)=\frac{3}{7}|t|^{7 / 3} \ln (e+|t|)+\frac{3}{13}|t|^{13 / 3}, \quad N=3, q=\frac{3}{2},2<\frac{7}{3}<q_{\#}=3<q^{\#}=\frac{10}{3}<\frac{13}{3}<q^{\prime}=6,
$$
which satisfies \eqref{F0}-\eqref{F5}. Then
$$
G(t)=-\frac{1}{21}|t|^{7 / 3} \ln (e+|t|)+\frac{3}{7} \frac{|t|^{10 / 3}}{e+|t|}-\frac{3}{7} \frac{|t|^{13 / 3}}{(e+|t|)^2}+\frac{35}{39}|t|^{13 / 3}
$$
satisfies \eqref{G0}--\eqref{G3}.

Finally, we observe that, if we consider $f$ as in Appendix B.1 with $q_K \leq 2+2 / N$, then $f$ satisfies \eqref{G0}--\eqref{G3} even for real exponents $q_k, p_{\ell}$.

\subsection*{Conflict of interest}

On behalf of all authors, the corresponding author states that there is no conflict of interest.

\subsection*{Ethics approval}
 Not applicable.

\subsection*{Data Availability Statements}
Data sharing not applicable to this article as no datasets were generated or analysed during the current study.

\subsection*{Acknowledgements}
C. Ji was partially supported by National Natural Science Foundation of China (No. 12171152).  P. Pucci is a member of the {\em Gruppo Nazionale per
l'Analisi Ma\-te\-ma\-ti\-ca, la Probabilit\`a e le loro Applicazioni}
(GNAMPA) of the {\em Instituto Nazionale di Alta Matematica} (INdAM)
and this paper was written under the auspices of   GNAMPA--INdAM.

\end{document}